\documentclass[a4paper, 12pt]{article}

\usepackage[top = 2.5cm, bottom = 2.5cm, left = 2.4cm, right = 2.4cm]{geometry}
\usepackage{amsfonts, amsmath, amsthm, amssymb, mathtools, cases, bbm, array}
\usepackage{lmodern, indentfirst, setspace, changepage, fancyhdr}
\usepackage{tikz, pgfplots, float, subcaption}
\usepackage[title]{appendix}
\usepackage[english]{babel}
\usepackage{cite}

\usepackage{hyperref}
\hypersetup{
	colorlinks = true,
	linkcolor = blue,
	citecolor = blue,
}

\newtheorem{proposition}{Proposition}
\newtheorem{corollary}{Corollary}
\newtheorem{theorem}{Theorem}
\newtheorem{lemma}{Lemma}
\theoremstyle{definition}
\newtheorem{assumption}{Assumption}
\newtheorem{definition}{Definition}

\newtheorem{remark}{Remark}

\newcommand{\argmin}[1]{\underset{#1}{\text{argmin}}}

\newcommand{\ind}[1]{\mathbbm{1}_{\left\{#1\right\}}}
\newcommand{\norm}[1]{\left|\left|#1\right|\right|}
\newcommand{\floor}[1]{\left\lfloor#1\right\rfloor}

\newcommand{\ceil}[1]{\left\lceil#1\right\rceil}
\newcommand{\map}[3]{#1 : #2 \longrightarrow #3}
\newcommand{\set}[2]{\left\{#1 : #2\right\}}
\newcommand{\sett}[2]{\left(#1 : #2\right)}

\newcommand{\defeq}{\vcentcolon=}

\newcommand{\supp}{\mathrm{supp}}

\newcommand{\bq}{\boldsymbol{q}}
\newcommand{\br}{\boldsymbol{r}}
\newcommand{\bs}{\boldsymbol{s}}
\newcommand{\bu}{\boldsymbol{u}}
\newcommand{\bv}{\boldsymbol{v}}
\newcommand{\bw}{\boldsymbol{w}}
\newcommand{\bx}{\boldsymbol{x}}
\newcommand{\by}{\boldsymbol{y}}

\newcommand{\bE}{\boldsymbol{E}}
\newcommand{\bG}{\boldsymbol{G}}

\newcommand{\bL}{\boldsymbol{L}}
\newcommand{\bM}{\boldsymbol{M}}

\newcommand{\bX}{\boldsymbol{X}}
\newcommand{\bY}{\boldsymbol{Y}}

\newcommand{\prob}{\mathbbm{P}}
\newcommand{\calA}{\mathcal{A}}

\newcommand{\calD}{\mathcal{D}}

\newcommand{\calF}{\mathcal{F}}
\newcommand{\calG}{\mathcal{G}}
\newcommand{\calH}{\mathcal{H}}
\newcommand{\calI}{\mathcal{I}}

\newcommand{\calK}{\mathcal{K}}

\newcommand{\calN}{\mathcal{N}}

\newcommand{\calR}{\mathcal{R}}

\newcommand{\N}{\mathbbm{N}}

\newcommand{\R}{\mathbbm{R}}
\newcommand{\Z}{\mathbbm{Z}}
\newcommand{\e}{\mathrm{e}}

\newcommand{\expect}{E\expectarg}
\newcommand{\var}{\mathrm{Var}\expectarg}
\DeclarePairedDelimiterX{\expectarg}[1]{[}{]}{%
	\ifnum\currentgrouptype=16 \else\begingroup\fi
	\activatebar#1
	\ifnum\currentgrouptype=16 \else\endgroup\fi
}
\newcommand{\cprob}{P\probarg}
\DeclarePairedDelimiterX{\probarg}[1]{(}{)}{%
	\ifnum\currentgrouptype=16 \else\begingroup\fi
	\activatebar#1
	\ifnum\currentgrouptype=16 \else\endgroup\fi
}
\newcommand{\innermid}{\nonscript\;\delimsize\vert\nonscript\;}
\newcommand{\activatebar}{%
	\begingroup\lccode`\~=`\|
	\lowercase{\endgroup\let~}\innermid 
	\mathcode`|=\string"8000
}

\usetikzlibrary{automata, arrows, positioning, calc, external, babel}
\usepgfplotslibrary{fillbetween}
\pgfplotsset{
	compat = 1.16,
	ticklabel style = {font = \footnotesize},
	every axis/.append style = {
		grid style = {dashed, gray, opacity = 0.2},
		label style = {font = \footnotesize}, 
		width = \columnwidth,
		height = 0.618 * 1 * \columnwidth
	}
}

\definecolor{britishracinggreen}{rgb}{0.0, 0.26, 0.15}
\definecolor{bostonuniversityred}{rgb}{0.8, 0.0, 0.0}
\definecolor{ceruleanblue}{rgb}{0.16, 0.32, 0.75}
\definecolor{airforceblue}{rgb}{0.36, 0.54, 0.66}
\definecolor{cadmiumgreen}{rgb}{0.0, 0.42, 0.24}
\definecolor{ao(english)}{rgb}{0.0, 0.5, 0.0}
\definecolor{coolblack}{rgb}{0.0, 0.18, 0.39}
\definecolor{byzantine}{rgb}{0.74, 0.2, 0.64}
\definecolor{alizarin}{rgb}{0.82, 0.1, 0.26}
\definecolor{arsenic}{rgb}{0.23, 0.27, 0.29}
\definecolor{cobalt}{rgb}{0.0, 0.28, 0.67}
\definecolor{amber}{rgb}{1.0, 0.75, 0.0}

\title{Fluid limits for interacting queues\\in sparse dynamic graphs\vspace{\baselineskip}}

\author{
\begin{tabular}{ccc}
	\normalsize Diego Goldsztajn & \hspace{5mm} \normalsize Sem C. Borst & \hspace{5mm} \normalsize Johan S.H. van Leeuwaarden \\
	\scriptsize Universidad ORT Uruguay & \hspace{5mm} \scriptsize Eindhoven University of Technology & \hspace{5mm} \scriptsize Tilburg University \\
	\scriptsize\texttt{goldsztajn@ort.edu.uy} & \hspace{5mm} \scriptsize\texttt{s.c.borst@tue.nl} & \hspace{5mm} \scriptsize\texttt{j.s.h.vanleeuwaarden@uvt.nl} \\
\end{tabular}
}

\date{\vspace{\baselineskip} October 11, 2025}

\begin{document}

	
\maketitle

\noindent\rule{\textwidth}{1pt}

\vspace{1\baselineskip}

\onehalfspacing

\begin{adjustwidth}{0.5cm}{0.5cm}
	\begin{center}
		\textbf{Abstract}
	\end{center}
	
	\vspace{0.3\baselineskip}
	
	\noindent Consider a network of $n$ single-server queues where tasks arrive independently at each server at rate $\lambda_n$. The servers are connected by a graph that is resampled at rate $\mu_n$ in a way that is symmetric with respect to the servers, and each task is dispatched to the shortest queue in the graph neighborhood where it appears. We aim to gain insight in the impact of the dynamic network structure on the load balancing dynamics in terms of the occupancy process which describes the empirical distribution of the number of tasks across the servers. This process evolves on the underlying dynamic graph, and its dynamics depend on the number of tasks at each individual server and the neighborhood structure of the graph. We establish that this dependency disappears in the limit as $n \to \infty$ when $\lambda_n / n \to \lambda$ and $\mu_n \to \infty$, and prove that the limit of the occupancy process is given by a system of differential equations that depends solely on $\lambda$ and the limiting degree distribution of the graph. We further show that the stationary distribution of the occupancy process converges to an equilibrium of the differential equations, and derive properties of this equilibrium that reflect the impact of the degree distribution. Our focus is on truly sparse graphs where the maximum degree is uniformly bounded across $n$, which is natural in load balancing systems.
	
	\vspace{0.5\baselineskip}
	
	
	\vspace{0.25\baselineskip}
	
	\small{\noindent  Research carried out while Diego Goldsztajn was with Eindhoven University of Technology. The authors were supported by the Netherlands Organisation for Scientific Research (NWO) through Gravitation-grant NETWORKS-024.002.003 and Vici grant 202.068. Support was also provided by ANII Uruguay through the fellowship PD\_NAC\_2024\_1\_182118.} 
\end{adjustwidth}

\newpage


\section{Introduction}
\label{sec: introduction}

We consider a network of $n$ single-server queues where tasks arrive at each server as independent Poisson processes of the same intensity $\lambda_n / n$. The servers are interconnected by a graph $\bG_n(t)$ that is resampled from some fixed random graph law at rate $\mu_n$. The tasks are dispatched to the shortest queue in the graph neighborhood where they appear and have exponentially distributed service requirements with unit mean.

We are interested in the occupancy process $\bq_n$ which describes the empirical distribution of the number of tasks, i.e., $\bq_n(t, i)$ is the fraction of servers with at least $i$ tasks at time~$t$. The dynamics of this process can be described by equations of the form:
\begin{equation}
	\label{eq: high-level stochastic equations}
	\bq_n(t, i) = \bq_n(0, i) + \frac{1}{n}\calA_n\left(\bG_n, \bX_n, t, i\right) - \frac{1}{n}\calD_n\left(\bq_n, t, i\right).
\end{equation}
The second and third terms count the numbers of arrivals to servers with exactly $i - 1$ tasks and departures from servers with exactly $i$ tasks, respectively. While the third term only depends on the occupancy process itself, the second term depends on the dynamic graph $\bG_n$ and the process $\bX_n$ that describes the number of tasks at each server, i.e., $\bX_n(t, u)$ denotes the number of tasks queueing in front of server $u$ at time $t$.

Due to the complexity of \eqref{eq: high-level stochastic equations}, the occupancy process is not amenable to exact analysis, and thus we will analyze a limiting regime where $n \to \infty$ and $\lambda_n / n \to \lambda$, offering far greater tractability. For example, the limit of the occupancy process is given by an infinite system of differential equations when the graph is complete. However, the proof of this result hinges critically on the fact that the right-hand side of \eqref{eq: high-level stochastic equations} depends solely on the occupancy process itself. This property holds because servers with the same number of tasks are exchangeable when the graph is complete, which is not true in the general case where the neighborhoods of different servers are typically different.

Several papers have successfully treated the lack of exchangeability when the queues interact over a static network, but the situation where the network is dynamic has received little attention, and to the best of our knowledge had not been rigorously analyzed before. From a mathematical perspective, the dynamic case is interesting since the dynamics of the network impact how the occupancy process evolves over time, leading to different proof arguments for characterizing the limiting behavior, as discussed in Sections \ref{sub: main results and design implications} and \ref{sub: related work}.

Moreover, there are multiple real-world examples of queues interacting over dynamic network structures. For instance, the connections between the nodes of mobile ad-hoc and peer-to-peer networks are constantly changing, and the network topology of data centers can be adjusted over time, particularly with optical devices.

We prove that the fluid limit of the occupancy process is governed by an infinite system of differential equations if $\mu_n \to \infty$, the law of the graph is invariant under permutations of the nodes and the degree distribution converges weakly. Remarkably, the limit depends solely on the probability generating function of the limiting degree distribution and not on any other structural properties of the graph. In addition, the system of differential equations has a global attractor when the limiting degree distribution has finite mean, and the stationary distribution of the occupancy process converges weakly to the global attractor when the graph is resampled according to a Poisson process. Our analysis focuses on the truly sparse regime where the maximum degree remains uniformly bounded across~$n$, which is natural in a load balancing context in view of the associated communication overhead; however, we study dense graphs as well.

\subsection{Overview of the main result}
\label{sub: main results and design implications}

As alluded to above, we assume that the random graph law from which the graph is resampled is invariant under permutations of the nodes. Although this property implies that the resampling procedure is symmetric with respect to the servers, it does not impose any restrictions on the graph topology. For example, the topology can be star-shaped at all times if we define the random graph law by randomly relabeling the nodes of a deterministic star-shaped graph. More precisely, we can construct a random graph law that is invariant under permutations of the nodes from any given random graph law, by sampling a graph from the given random graph law and relabeling the nodes uniformly at random.

In the special case where the graph is resampled at every arrival time, the invariance under permutations of nodes implies that the dynamics of the system are equivalent to the following generalized power-of-$(d + 1)$ scheme. When a task arrives, a number $d$ is sampled from the degree distribution of the graph and the task is sent to a server with the least number of tasks among $d + 1$ servers chosen uniformly at random among all the servers. In this case the servers are exchangeable and the right-hand side of \eqref{eq: high-level stochastic equations} depends solely on the occupancy process and the degree distribution. When the graph is regular so the degree distribution is deterministic, we recover the classical power-of-$d$ scheme of \cite{vvedenskaya1996queueing,mitzenmacher2001power}.

We focus instead on the more challenging and interesting situation where the graph remains fixed throughout many consecutive arrivals. In particular, the average number of arrivals between two consecutive resampling times is $\lambda_n / \mu_n$ and may approach infinity as $n \to \infty$. In this case the servers are not exchangeable and the right-hand side of \eqref{eq: high-level stochastic equations} depends on $\bG_n$ and $\bX_n$, because the current graph and the precise location of the tasks influence the dispatching decisions and their effect on the occupancy process.

The main challenge is to establish that the dependence of the right-hand side of \eqref{eq: high-level stochastic equations} on $\bG_n$ and $\bX_n$ disappears in the limit, which we do by using two key insights. First, the dynamics of the occupancy process are asymptotically equivalent if on the right-hand side of \eqref{eq: high-level stochastic equations} the current state of the system is replaced by the state at the most recent resampling time. Second, if the occupancy state is given and the graph is unknown, then the effect of the next dispatching decision on the occupancy state is statistically determined by the current occupancy state and the graph law, rather than the specific graph in effect. Combining these two insights, we show that the dynamics of the occupancy process are asymptotically determined by the occupancy states at the resampling times and the random graph law. Then we prove that the stochastic equations \eqref{eq: high-level stochastic equations} are asymptotically equivalent to those of the generalized power-of-$(d + 1)$ scheme, leading to the same fluid limit.

Informally speaking, the first of the above insights is obtained by carefully bounding the average number of dispatching decisions that would be different if all the queue lengths remained fixed between successive resampling times. This provides a bound for the mean of the total number of different dispatching decisions over any finite interval of time. We prove that this bound approaches zero as $n \to \infty$ and $\mu_n \to \infty$, after we normalize it by the number of servers $n$. This implies that the limit of the occupancy process indeed remains the same when on the right-hand side of \eqref{eq: high-level stochastic equations} the current state of the system is replaced by the state of the system at the most recent resampling time.

The second insight allows to identify suitable vanishing and nonvanishing terms on the right-hand side of \eqref{eq: high-level stochastic equations}. In particular, the second term can be decomposed into two terms. The first one counts dispatched tasks as if the state of the system remained fixed between successive resampling times and the graph was resampled at each arrival time. This term is nonvanishing and similar to a term that appears in the stochastic equations of a generalized power-of-$(d + 1)$ scheme. The second term accounts for the error in the simplification of the dispatching dynamics used to define the first term. By focusing on the resampling times, we identify a discrete-time martingale embedded in this error process; the proof of the martingale property relies on the independence of the graphs used between different couples of consecutive resampling times. We show that this martingale vanishes in the limit and then prove that the entire error process also approaches zero.

\subsection{Related work}
\label{sub: related work}

The situation where the graph is complete corresponds to the so-called supermarket model, which has received immense attention in the literature; see \cite{van2018scalable} for an extensive survey. The fluid limit for the classical power-of-$d$ scheme was obtained in \cite{mitzenmacher2001power,vvedenskaya1996queueing}, whereas \cite{mukherjee2018universality} derived the fluid limit of the occupancy process when $d$ approaches infinity as $n \to \infty$; this covers as a special case the situation where each task is dispatched to the shortest of all the queues. As observed earlier, the servers are exchangeable when the graph is complete, which allows to leverage mean-field arguments similar to those in the pioneering studies \cite{kurtz1970solutions,kurtz1971limit,kurtz1978strong,barbour1980density,norman1972markov,norman1974central}; more recent papers on mean-field limits include \cite{benaim2008class,allmeier2022mean,allmeier2024accuracy,allmeier2023bias}.

Occupancy processes that evolve on static and not complete graphs were considered in \cite{mukherjee2018asymptotically}, where tasks are sent to the shortest queue in the neighborhood where they appear. Assuming that the graphs have suitable edge-expansion properties, \cite{mukherjee2018asymptotically} shows that the fluid limit of the occupancy process is the same as when the graphs are complete. The latter edge-expansion properties imply that the mean degree goes to infinity as $n \to \infty$ and ensure that the queue length distribution is approximately the same for each neighborhood. Static graphs with diverging minimum degree and suitable regularity properties were analyzed in \cite{budhiraja2019supermarket}, where each task is dispatched to the shortest of $d$ queues selected uniformly at random from the neighborhood where the task appeared. In this case the fluid limit of the occupancy process coincides with that for complete graphs derived in \cite{mitzenmacher2001power,vvedenskaya1996queueing}.

In the context of interacting particle systems, \cite{oelschlager1984martingale} considered the empirical distribution of particles in the Euclidean space. Assuming that the dynamics of each particle depend on the state of the particle and the empirical distribution of all the particles, \cite{oelschlager1984martingale} proved that the process describing the empirical distribution of the particles converges to a deterministic process as the number of particles approaches infinity. In this scenario the strength of the interaction between any two particles is inversely proportional to the number of particles and thus vanishes in the limit. In contrast, neighboring particles remain strongly correlated when the particles interact through a truly sparse graph as in \cite{ganguly2022non,ganguly2024hydrodynamic}. Despite the lack of propagation of chaos, \cite{ganguly2024hydrodynamic,ganguly2022non} establish that the limit of the empirical distribution is deterministic and coincides with the limiting law of a typical particle under mild conditions. It is further proved that the limiting law of a particle and its neighborhood can be described by a so-called local equation when the graph converges locally weakly to an infinite regular tree. The local equation was introduced by \cite{lacker2023marginal} in the context of diffusive particles, and describes a stochastic process whose infinitesimal evolution depends on the conditional law of the current state given the past states of the neighboring nodes.

Several papers have recently considered the situation where different streams of tasks are distributed across servers that are connected to the streams through a bipartite graph. General stability results have been provided in \cite{foss1998stability}, and the fluid limit of the occupancy process has been derived in \cite{rutten2022load,rutten2023mean,zhao2024exploiting,zhao2023optimal} under different assumptions on the bipartite graph and the dynamics, including the arrival and service rates. A similar setup has been considered in \cite{weng2020boptimal} but focusing on the mean amount of time that tasks stay in the system in stationarity. In all the cases the graph is static, the degree of the streams diverges and the limiting behavior coincides with the case where the bipartite graph is complete.

In \cite{mukherjee2018asymptotically,budhiraja2019supermarket} and the above papers, the strength of interactions between neighboring servers vanishes because the queue length distribution across a neighborhood is hardly affected in the limit by the queue length of an individual server. In contrast, servers can have strong interactions over a diverging number of consecutive arrivals when the graph is dynamic and truly sparse as in the present paper. However, as we will show, only the average of these interactions matters in the limit due to the way in which the graph changes over time.

\subsection{Some basic notation}
\label{sub: basic notation}

The symbols $P$ and $E$ are used to denote the probability of events and the expectation of functions, respectively. The underlying probability measure to which these symbols refer is always clear from the context or explicitly indicated.

For random variables with values on a common metric space $S$, we denote the weak convergence of $\set{X_n}{n \geq 1}$ to $X$ by $X_n \Rightarrow X$. If $X$ is deterministic, then the weak limit holds if and only if the random variables $X_n$ converge in probability to $X$. In this situation we use the terms \emph{converges weakly} and \emph{converges in probability} interchangeably.

The left and right limits of a function $\map{f}{[0, \infty)}{S}$ are denoted by
\begin{equation*}
f\left(x^-\right) \defeq \lim_{y \to x^-} f(y) \quad \text{for all} \quad x > 0 \quad \text{and} \quad f\left(x^+\right) \defeq \lim_{y \to x^+} f(y) \quad \text{for all} \quad x \geq 0,
\end{equation*}
respectively. We say that $f$ is c\`adl\`ag if the left limits exist for all $x > 0$ and the right limits exist and satisfy $f\left(x^+\right) = f(x)$ for all $x \geq 0$.

We include zero in the set of natural numbers $\N$. Also, we let
\begin{equation*}
\floor{x} \defeq \max \set{n \in \Z}{n \leq x} \quad \text{and} \quad \ceil{x} \defeq \min \set{n \in \Z}{n \geq x} \quad \text{for all} \quad x \in \R.
\end{equation*}

\subsection{Organization of the paper}
\label{sub: organization of the paper}

The rest of the paper is organized as follows. In Section \ref{sec: model description} we specify the dynamics of the occupancy process. In Section \ref{sec: fluid limit} we formulate the fluid limit result. Sections \ref{sec: properties of fluid trajectories} and \ref{sec: performance in equilibrium} focus on sparse graphs where the average degree is upper bounded by some given constant. In Section \ref{sec: properties of fluid trajectories} we establish certain dynamical properties of the differential equation that characterizes the fluid limit, including existence of a globally attractive equilibrium. In Section \ref{sec: performance in equilibrium} we prove that the stationary distribution of the occupancy process converges to this equilibrium when the graph is resampled according to a Poisson process, and establish certain properties that are interesting from a queueing perspective. In Section \ref{sec: proof of the fluid limit} we prove the fluid limit. Appendix \ref{app: simulations} contains simulations involving static and dynamic graphs. Appendices \ref{app: auxiliary results}, \ref{app: construction of sample paths} and \ref{app: tightness of occupancy processes} contain the proofs of several intermediate results.

\section{Model description}
\label{sec: model description}

Consider a network of $n$ servers with infinite buffers. Tasks arrive locally at each of the servers as independent Poisson processes of rate $\lambda_n / n$ and service times are exponentially distributed with unit mean. At time $t$ the number of tasks present in server $u$ is denoted by $\bX_n(t, u)$ and the fraction of servers with at least $i$ tasks is given by
\begin{equation*}
\bq_n(t, i) \defeq \frac{1}{n} \sum_{u = 1}^n \ind{\bX_n(t, u) \geq i}.
\end{equation*}
The stochastic process $\bq_n$ is called occupancy process and the infinite sequence $\bq_n(t)$ is referred to as the occupancy state of the system at time $t$.

A simple directed graph on the set of servers $V_n \defeq \{1, \dots, n\}$ governs the dynamics of the occupancy process; all the results apply to undirected graphs as well since they have natural directed graphs counterparts. The graph is resampled over time from a given random graph law, with every new sample being independent from all the previous samples. At time $t$ the current graph is denoted by $\bG_n(t)$ and $\calR_n(t)$ denotes the number of times that the graph has been resampled so far. The stochastic process $\calR_n$ is called resampling process and its jumps coincide with the times at which the graph is resampled.

The set of edges at time $t$ is denoted by $\bE_n(t)$ and the neighborhood of a server $u$ at time $t$ consists of itself and all the servers $v$ such that $(u, v) \in \bE_n(t)$. The graph structure is used to distribute the tasks as follows. If a task arrives at server $u$ at time $t$, then the task is placed in the queue of an arbitrary server $v(u)$ contained in the set
\begin{equation*}
\argmin{v} \set{\bX_n(t, v)}{v = u\ \text{or}\ (u, v) \in \bE_n(t)},
\end{equation*}
which consists of the servers in the neighborhood of $u$ that have the least number of tasks.

We assume that the graph is invariant under permutations of nodes, i.e.,
\begin{equation*}
P\left(\bE_n(t) = \left\{(u_1, v_1), \dots, (u_m, v_m)\right\}\right) = P\left(\bE_n(t) = \left\{\left(\pi(u_1), \pi(v_1)\right), \dots, \left(\pi(u_m), \pi(v_m)\right)\right\}\right)
\end{equation*}
for all sets of edges $\left\{(u_1, v_1), \dots, (u_m, v_m)\right\}$ and all permutations $\map{\pi}{V_n}{V_n}$. This makes the resampling procedure symmetric with respect to the servers.

\begin{remark}
	\label{rem: invariance under permutations}
	A random graph law satisfying the above condition can be obtained as follows. Let $H$ be any random graph distribution with node set $V_n$. If we draw a graph $h$ from $H$ and a permutation $\map{\pi}{V_n}{V_n}$ uniformly at random, then we can define a graph $g$ by permuting the labels of the nodes of $h$ according to $\pi$. The random graph law $G$ of the graph $g$ obtained in this way is invariant under permutations of the nodes. In other words, the arbitrary random graph law $H$ determines the graph topology and labels are attached to the nodes uniformly at random. For example, suppose that $H$ is the point mass at the undirected graph $h$ such that some node has degree $n - 1$ and all the other nodes have degree one. The above-described random graph law $G$ assigns probability $1 / n!$ to each of the undirected graphs that result from permuting the nodes of $h$. Thus, the topology of $G$ is star-shaped almost surely and each node has probability $1 / n$ of having degree $n - 1$.
\end{remark}

\section{Fluid limit}
\label{sec: fluid limit}

As the number of servers goes to infinity, the asymptotic behavior of the occupancy process can be described by a system of differential equations if certain conditions on the outdegree distribution $D_n$ and the resampling process hold. We define $D_n$ as the distribution of the outdegree of a node selected uniformly at random and we denote the probability mass function of $D_n$ by $p_n(d) \defeq P(D_n = d)$. The fluid limit is proved under the following set of conditions.

\begin{assumption}
\label{ass: fluid limit conditions}
There exist constants $\lambda > 0$ and $\set{p(d) \in [0, 1]}{d \in \N}$ such that
\begin{equation}
\label{eq: limiting degree distribution condition}
\lim_{n \to \infty} \frac{\lambda_n}{n} = \lambda \quad \text{and} \quad \lim_{n \to \infty} p_n(d) = p(d) \quad \text{for all} \quad d \in \N.
\end{equation}
In addition, the resampling processes satisfy a technical property that we define later: we assume that they \emph{pseudo-separate events}.
\end{assumption}

The pseudo-separation property mentioned above is formally stated in Section \ref{sub: events separation property} using notation that we introduce later. Informally, the property implies that the holding time and the total number of arrivals and departures between any two successive resampling times are suitably bounded. The following proposition shows that this property is rather general and holds in many cases of interest; the proof is deferred to Section \ref{sub: events separation property}. In particular, the resampling process can be a renewal process with a rate $\mu_n$ that approaches infinity at an arbitrarily slow rate, and the number of arrivals between successive resampling times can approach infinity with $n$ at any sublinear rate.

\begin{proposition}
	\label{prop: admissible resampling processes}
	Suppose that $\lambda_n / n \to \lambda$ as $n \to \infty$ and there exist $\set{\kappa_n \in \N}{n \geq 1}$ and $\set{\mu_n > 0}{n \geq 1}$ such that the processes $\calR_n$ satisfy one of the following conditions.
	\begin{enumerate}
		\item[(a)] If $s < t$ are any two consecutive resampling times, then exactly $\kappa_n + 1$ tasks arrive in the interval $(s, t]$. Also, $\calR_n$ is independent of the departure times of tasks.
		
		\item[(b)] The resampling processes are independent of the history of the system and the amount of time elapsed between any two consecutive resampling times is at most $1 / \mu_n$.
		
		\item[(c)] We have $\calR_n(t) = \calR(\mu_n t)$ for all $t \geq 0$, where $\calR$ is a fixed independent renewal process with a holding time distribution that has unit mean and finite variance.
	\end{enumerate}
	Also, assume that there exist constants $\set{d_n^- \geq 0}{n \geq 1}$ such that in the system with $n$ servers the indegree of the servers is at most $d_n^-$ with probability one and we have:
	\begin{equation}
	\label{eq: arrival and resampling rates conditions}
	\lim_{n \to \infty} \kappa_n \frac{d_n^- + 1}{n} = 0 \quad \text{if (a) holds,} \quad \text{or} \quad \lim_{n \to \infty} \frac{d_n^- + 1}{\mu_n} = 0 \quad \text{if (b) or (c) hold}.
	\end{equation}
	Then the resampling processes pseudo-separate events.
\end{proposition}

Note that (a) covers the situation where the graph is always resampled after a given number of arrivals, whereas (b) and (c) cover resampling processes that are governed by an independent clock. If condition (b) holds, then the distributions of the amounts of time between two successive ticks of the clock can be arbitrary as long as they remain supported in $[0, 1 / \mu_n]$. In particular, these holding time distributions are not required to be identical. In contrast, (c) implies that the holding times between successive resampling times are identically distributed, but allows for holding time distributions with infinite support.

\begin{remark}
	\label{rem: resampling rate}
	When conditions (b) or (c) of Proposition \ref{prop: admissible resampling processes} hold, the mean number of tasks that arrive between two successive resampling times is at most $\lambda_n / \mu_n$. In addition, if $\lambda_n / n \to \lambda$ as $n \to \infty$, then \eqref{eq: arrival and resampling rates conditions} is equivalent to
	\begin{equation*}
	\lim_{n \to \infty} \kappa_n \frac{d_n^- + 1}{n} = 0 \quad \text{if (a) holds,} \quad \text{or} \quad \lim_{n \to \infty} \frac{\lambda_n}{\mu_n}\frac{d_n^- + 1}{n} = 0 \quad \text{if (b) or (c) hold}.
	\end{equation*}
	Hence, the condition on the resampling rate is essentially the same under (a), (b) and (c) of Proposition \ref{prop: admissible resampling processes}. If (a) holds and $\kappa_n = 0$ for all $n$ or (b) or (c) hold and $\lambda_n / \mu_n \to 0$ as $n \to \infty$, then \eqref{eq: arrival and resampling rates conditions} holds regardless of how the maximum indegrees $d_n^-$ behave asymptotically.
\end{remark}

\begin{remark}
	\label{rem: strongest possible assumption regarding resampling rate}
	As noted earlier, the truly sparse regime where the maximum indegrees $d_n^-$ are uniformly bounded across $n$ is the focus of this paper. In this case \eqref{eq: arrival and resampling rates conditions} simply states that $\kappa_n = o(n)$ and $\mu_n \to \infty$ as $n \to \infty$. Thus, the number of arrivals between successive resampling times can approach infinity at any sublinear rate. These conditions are tight in the sense that they cannot be weakened without entering into the realm of fluid limits for static graphs. For example, if $\mu_n$ is bounded and the resampling process is deterministic, then there exists $\varepsilon > 0$ such that the initial graph remains fixed in $[0, \varepsilon]$. A fluid limit here would yield a fluid limit over $[0, \varepsilon]$ for a static random graph.
\end{remark}

\begin{remark}
	\label{rem: maximum indegree condition}
	The pseudo-separation property and \eqref{eq: arrival and resampling rates conditions} involve the maximum indegrees $d_n^-$. While we do not believe the pseudo-separation property to be a necessary condition for the fluid limit to hold, the numerical experiments of Appendix \ref{app: simulations} suggest that the dependence of this property on the maximum indegrees could be a manifestation of some fundamental condition that is in fact necessary for the fluid limit, and not just an artifact of our proof technique. Note however that the dependence of the pseudo-separation property on the maximum indegrees is trivial in the truly sparse regime that is the focus of this paper, i.e., when the maximum indegrees are uniformly bounded across $n$.
\end{remark}

It follows from \eqref{eq: limiting degree distribution condition} and Fatou's lemma that
\begin{equation*}
\sum_{d = 0}^\infty p(d) \leq \liminf_{n \to \infty} \sum_{d = 0}^\infty p_n(d) = 1.
\end{equation*}
If equality is attained on the left, then $D_n$ converges weakly as $n \to \infty$ to a distribution $D$ that has probability mass function $p$. We refer to $D$ as the limiting outdegree distribution and we let $\varphi$ denote its probability generating function. In general, we define
\begin{equation*}
\varphi(x) \defeq \sum_{d = 0}^\infty x^dp(d) \quad \text{for all} \quad x \in [0, 1] \quad \text{and} \quad p(\infty) \defeq 1 - \sum_{d = 0}^\infty p(d).
\end{equation*}
We say that the limiting outdegree distribution is nondegenerate and given by $D$ when $p(\infty) = 0$. Otherwise, we say that the limiting outdegree distribution is degenerate.

Let $\ell_1$ be the space of all absolutely summable $x \in \R^\N$ with the norm
\begin{equation*}
\norm{x}_1 \defeq \sum_{i = 0}^\infty |x(i)| \quad \text{for all} \quad x \in \ell_1.
\end{equation*}
The sample paths of $\bq_n$ lie in the space $D_{\ell_1}[0, \infty)$ of c\`adl\`ag functions from $[0, \infty)$ into $\ell_1$, which we endow with the metric of uniform convergence over compact sets. The following fluid limit is proved in Section \ref{sec: proof of the fluid limit}.

\begin{theorem}
\label{the: fluid limit}
Suppose that Assumption \ref{ass: fluid limit conditions} holds and that the sequence of initial occupancy states $\set{\bq_n(0)}{n \geq 1}$ is tight in $\ell_1$. Then every subsequence of $\set{\bq_n}{n \geq 1}$ has a further subsequence that converges weakly in $D_{\ell_1}[0, \infty)$. Furthermore, the limit $\bq$ of any convergent subsequence is almost surely continuous from $[0, \infty)$ into $\ell_1$ and satisfies
\begin{equation}
\label{eq: fluid dynamics}
\begin{split}
\bq(t, i) &= \bq(0, i) + \lambda\int_0^t \left[a_{i - 1}\left(\bq(s)\right) - a_i\left(\bq(s)\right)\right]ds - \int_0^t \left[\bq(s, i) - \bq(s, i + 1)\right]ds
\end{split}
\end{equation}
for all $i \geq 1$ and $t \geq 0$ with probability one. The quantity $a_i(q)$ can be interpreted as the asymptotic probability of a task being dispatched to a server with at least $i$ tasks when the occupancy state is $q$. We have $a_0(q) \defeq 1$ and
\begin{equation*}
	a_i(q) \defeq q(i)\varphi\left(q(i)\right) + \left[p(\infty) - \frac{1 - q(i + 1)}{\lambda}\right]^+\ind{q(i) = 1} \quad \text{for all} \quad i \geq 1.
\end{equation*}
\end{theorem}

Note that the fluid limit \eqref{eq: fluid dynamics} depends on the limiting outdegree distribution through the probability generating function $\varphi$. Only this property affects the asymptotic behavior of the occupancy process and the impact of any other structural properties of the random graph law used to sample the graph disappears in the limit.

Theorem \ref{the: fluid limit} covers the special case where the graph is resampled between successive arrivals as Assumption \ref{ass: fluid limit conditions} holds by Proposition \ref{prop: admissible resampling processes} with $\kappa_n = 0$.
In this case every task is sent to the shortest of $D_n + 1$ queues selected uniformly at random without replacement, which is statistically equivalent to the generalized power-of-$(d + 1)$ supermarket model introduced in Section \ref{sub: main results and design implications}. Proposition \ref{prop: admissible resampling processes} and Theorem \ref{the: fluid limit} show that the fluid limit remains the same under far more general assumptions on the resampling process.

Some arithmetic manipulations yield
\begin{equation*}
	a_i(q) = q(i) \varphi\left(q(i)\right) + \left[p(\infty) - \min\left\{p(\infty), \frac{1 - q(i + 1)}{\lambda}\right\}\right]\ind{q(i) = 1} \quad \text{for all} \quad i \geq 1.
\end{equation*}
If we define $m(q) \defeq \max\set{i \geq 0}{q(i) = 1}$ and note that $\varphi(1) + p(\infty) = 1$, then we obtain the following more explicit expressions:
\begin{equation*}
	a_i(q) =
	\begin{cases}
		1 & \text{if} \quad i < m(q) \quad \text{or} \quad i = 0, \\
		1 - \min\left\{p(\infty), \frac{1 - q(i + 1)}{\lambda}\right\} & \text{if} \quad i = m(q) > 0, \\
		q(i)\varphi\left(q(i)\right) & \text{if} \quad i > m(q).
	\end{cases}
\end{equation*}
As noted earlier, $a_i(q)$ can be interpreted as the limiting probability of a task being sent to a server with at least $i$ tasks when the occupancy state is $q$. In order to explain this, we consider a generalized power-of-$(d + 1)$ supermarket model where the probability mass function of $d$ is $p$ and the $d + 1$ queues are chosen with replacement; this is equivalent in the limit as $n \to \infty$. If $q(i) < 1$ or $p(\infty) = 0$, then the probability that a task is dispatched to a server with at least $i$ tasks is
\begin{equation*}
	\sum_{d = 0}^\infty \left[q(i)\right]^{d + 1} p(d) = q(i) \sum_{d = 0}^\infty \left[q(i)\right]^d p(d) = q(i)\varphi\left(q(i)\right) = a_i(q).
\end{equation*}
If $q(i) = 1$, then $a_i(q) = 1$, unless $p(\infty) > 0$ and $q(i + 1) < 1$, i.e., $i = m(q)$. In this case the service completions at servers with exactly $i$ tasks leave them with only $i - 1$ tasks. Furthermore, servers with just $i - 1$ tasks appear at rate $1 - q(i + 1)$, and loosely speaking, tasks are dispatched to the shortest among all the queues at rate $\lambda p(\infty)$. If the latter rate is smaller, then $q(i)$ ceases to be one immediately, but otherwise tasks are dispatched to servers with exactly $i - 1$ tasks at rate $1 - q(i + 1)$ and $q(i)$ remains equal to one. In the latter case, tasks are sent to servers with exactly $i - 1$ tasks with probability $\left[1 - q(i + 1)\right] / \lambda$ and $a_i(q) = 1 - \left[1 - q(i + 1)\right] / \lambda < 1$.

\begin{remark}
	\label{rem: skorohod equations}
	An alternative expression for \eqref{eq: fluid dynamics} can be obtained by means of the Skorohod reflection mapping with upper reflecting barrier at $1$. Let $\br(t, 0) \defeq \lambda t$ for all $t \geq 0$ and
	\begin{equation*}
		\br(t, i) \defeq \int_0^t \left[\lambda p(\infty) - 1 + \bq(s, i + 1)\right]^+\ind{\bq(s, i) = 1}ds \quad \text{for all} \quad t \geq 0 \quad \text{and} \quad i \geq 1.
	\end{equation*}
	It is clear that for all $i \geq 1$ and $t \geq 0$, we have
	\begin{align*}
		\bq(t, i) &= \bq(0, i) + \lambda \int_0^t \left[\bq(s, i - 1) \varphi\left(\bq(s, i - 1)\right) - \bq(s, i) \varphi\left(\bq(s, i)\right)\right]ds \\
		&+ \br(t, i - 1) - \br(t, i) - \int_0^t \left[\bq(s, i) - \bq(s, i + 1)\right]ds.
	\end{align*}
	Here the reflection terms $\br(i)$ are the unique nondecreasing functions that are flat off $\set{t \geq 0}{\bq(t, i) = 1}$ and prevent $\bq(i)$ from increasing beyond one; see \cite[Remark 5.5]{mythesis} for details. The reflection terms are identically zero for all $i \geq 1$ when $p(\infty) = 0$, and when $p(\infty) = 1$, the above equation coincides with the fluid limit derived in \cite{bhamidi2022near}.
\end{remark}


Suppose that the initial occupancy states $\bq_n(0)$ converge weakly to some deterministic limit $q$ and that \eqref{eq: fluid dynamics} has a unique solution such that $\bq(0) = q$. In this case Theorem \ref{the: fluid limit} says that the occupancy processes $\bq_n$ approach the unique solution of \eqref{eq: fluid dynamics} for the initial condition $q$. In Section \ref{sec: properties of fluid trajectories} we prove that \eqref{eq: fluid dynamics} has a unique solution for any given initial condition when the limiting outdegree distribution is nondegenerate and has finite mean, and we show that all the solutions, regardless of the initial condition, converge over time to a unique equilibrium point. In Section \ref{sec: performance in equilibrium} we assume that $\calR_n$ is a Poisson process of rate $\mu_n$ and we use the global attractivity result to characterize the stationary behavior of the system as $n \to \infty$. As noted earlier, the fluid limit is proved in Section \ref{sec: proof of the fluid limit}.

\section{Properties of fluid trajectories}
\label{sec: properties of fluid trajectories}

In most of this section we assume that
\begin{equation}
\label{ass: sparsity condition}
p(\infty) = 0 \quad \text{and} \quad \sum_{d = 0}^\infty dp(d) < \infty,
\end{equation}
which means that the limiting outdegree distribution is nondegenerate and has finite mean.

Since $p(\infty) = 0$, the differential form of \eqref{eq: fluid dynamics} is given by
\begin{equation}
\label{eq: fluid differential equation}
\dot{\bq}(i) = \lambda\left[\bq(i - 1)\varphi\left(\bq(i - 1)\right) - \bq(i)\varphi\left(\bq(i)\right)\right] - \left[\bq(i) - \bq(i + 1)\right] \quad \text{for all} \quad i \geq 1,
\end{equation}
where the equations hold almost everywhere with respect to the Lebesgue measure. A fluid trajectory is a function $\bq$ from $[0, \infty)$ into
\begin{equation*}
Q \defeq \set{q \in \ell_1}{0 \leq q(i + 1) \leq q(i) \leq q(0) = 1\ \text{for all}\ i \geq 1}
\end{equation*}
that satisfies \eqref{eq: fluid differential equation}. The fact that the limiting outdegree distribution has a finite mean gives the following lemma, which we prove in Appendix \ref{app: auxiliary results}.

\begin{lemma}
	\label{lem: derivative of phi}
	If \eqref{ass: sparsity condition} holds, then
	\begin{equation*}
	\lim_{x \to 1^-} \frac{\varphi(1) - \varphi(x)}{1 - x} = \sum_{d = 0}^\infty dp(d) = \lim_{x \to 1^-} \varphi'(x).
	\end{equation*} 
	In other words, $\varphi$ is continuously differentiable on $[0, 1]$.
\end{lemma}

This lemma implies that the functions $\varphi$ and $x \mapsto x\varphi(x)$ are Lipschitz on $[0, 1]$, which makes it possible to derive certain properties of fluid trajectories.

\subsection{Existence, uniqueness and monotonicity}
\label{sub: uniqueness and monotonicity}

We begin with an existence and uniqueness result, which is proved in Appendix \ref{app: auxiliary results}.

\begin{proposition}
	\label{prop: existence and uniqueness of solutions}
	Suppose that condition \eqref{ass: sparsity condition} holds. For each $q \in Q$ there exists a unique fluid trajectory $\bq$ such that $\bq(0) = q$. Moreover, $\bq$ is continuous from $[0, \infty)$ into $\ell_1$ and the fluid trajectories are continuous in $D_{\ell_1}[0, \infty)$ with respect to the initial condition.
\end{proposition}

\begin{remark}
	\label{rem: uniqueness}
	The existence part of Proposition \ref{prop: existence and uniqueness of solutions} holds even if we do not assume \eqref{ass: sparsity condition}, and the proof provided in Appendix \ref{app: auxiliary results} does not require any modifications. The uniqueness part is more delicate when \eqref{ass: sparsity condition} does not hold. This property implies that $x \mapsto x\varphi(x)$ is Lipschitz in $[0, 1]$, which plays an important role in the proof of Proposition \ref{prop: existence and uniqueness of solutions}. If condition \eqref{ass: sparsity condition} is replaced by $p(\infty) = 1$, then this also implies uniqueness; see \cite{bhamidi2022near}.
\end{remark}

The following corollary is a consequence of Theorem \ref{the: fluid limit} and Proposition \ref{prop: existence and uniqueness of solutions}.

\begin{corollary}
	\label{cor: unique fluid limit}
	Assume that condition \eqref{ass: sparsity condition} holds and consider a random variable $q$ defined on a probability space $(\Omega, \calF, \prob)$ and with values in $Q$. In addition, let $\bq$ be the stochastic process such that $\bq(\omega)$ is the unique fluid trajectory with initial condition $\bq(\omega, 0) = q(\omega)$. If $\bq_n(0) \Rightarrow q$ in $\ell_1$ as $n \to \infty$, then $\bq_n \Rightarrow \bq$ in $D_{\ell_1}[0, \infty)$ as $n \to \infty$.
\end{corollary}

\begin{proof}
	Consider the function $\map{\Phi}{Q}{D_{\ell_1}[0, \infty)}$ that maps initial conditions to fluid trajectories. By Proposition \ref{prop: existence and uniqueness of solutions}, this function is well-defined and continuous if \eqref{ass: sparsity condition} holds.
	
	Theorem \ref{the: fluid limit} implies that every subsequence of $\set{\bq_n}{n \geq 1}$ has a further subsequence that converges weakly in $D_{\ell_1}[0, \infty)$ to a process $\br$ such that $\br = \Phi\left(\br(0)\right)$ almost surely. The projection $\bx \mapsto \bx(0)$ is continuous from $D_{\ell_1}[0, \infty)$ into $\ell_1$. Therefore, the continuous mapping theorem implies that $\br(0)$ has the same distribution as $q$ and we conclude that $\br$ and $\bq = \Phi(q)$ have the same distribution as well. Thus, every subsequence of $\set{\bq_n}{n \geq 1}$ has a further subsequence that converges weakly in $D_{\ell_1}[0, \infty)$ to $\bq$.
\end{proof}

Consider now the partial order in $\ell_1$ defined by
\begin{equation*}
x \leq y \quad \text{if and only if} \quad x(i) \leq y(i) \quad \text{for all} \quad i \geq 1.
\end{equation*}
The following lemma says that the solutions of \eqref{eq: fluid differential equation} are monotone with respect to this ordering. The lemma is proved in Appendix \ref{app: auxiliary results} using similar arguments as in \cite{vvedenskaya1996queueing}, which proves the monotonicity property when $\varphi(x) = x$ for all $x \in [0, 1]$.

\begin{lemma}
	\label{lem: monotonicity}
	Suppose that assumption \eqref{ass: sparsity condition} holds. If $\bx$ and $\by$ are fluid trajectories such that $\bx(0) \leq \by(0)$, then $\bx(t) \leq \by(t)$ for all future times $t > 0$ as well.
\end{lemma}

The above monotonicity property will be important in the next section, where we will use it to prove that \eqref{eq: fluid differential equation} has a globally attractive equilibrium point.

\subsection{Global attractivity}
\label{sub: global attractivity}

In a system with $n$ servers, the condition $\lambda_n < n$ means that the total arrival rate of tasks is smaller than the combined service rate of all the servers. This stability condition turns into $\lambda < 1$ as $n \to \infty$. In this section we assume that the latter condition holds and we study the stability of the differential equation \eqref{eq: fluid differential equation}. First we derive the unique equilibrium of the more general equation \eqref{eq: fluid dynamics}. Recall that this equation is equivalent to \eqref{eq: fluid differential equation} when \eqref{ass: sparsity condition} holds, but note that the following result does not require that \eqref{ass: sparsity condition} holds.

\begin{proposition}
	\label{prop: fixed point}
	If $\lambda < 1$, then the infinite sequence
	\begin{equation*}
	q^*(i) \defeq \begin{cases}
	1 & \text{if} \quad i = 0, \\
	\lambda & \text{if} \quad i = 1, \\
	\lambda q^*(i - 1)\varphi\left(q^*(i - 1)\right) & \text{if} \quad i > 1,
	\end{cases}
	\end{equation*}
	is the unique equilibrium of \eqref{eq: fluid dynamics} within $Q$.
\end{proposition}

\begin{proof}
	The differential version of \eqref{eq: fluid dynamics} is
	\begin{equation}
	\label{aux: general fluid differential equation}
	\dot{\bq}(i) = \lambda\left[a_{i - 1}(\bq) - a_i(\bq)\right] - \left[\bq(i) - \bq(i + 1)\right] \quad \text{for all} \quad i \geq 1.
	\end{equation}
	It is clear that $q^*(i)$ decreases with $i$, and $q^*$ is an equilibrium point since $q^*(i) = \lambda a_{i - 1}\left(q^*\right)$ for all $i \geq 1$. In order to show that $q^* \in Q$, we prove that $q^*(i) \leq \lambda^i$ by induction; this implies that $q^* \in \ell_1$. The base case $i = 1$ holds by definition, and the inductive step also holds: if the property holds for $i$, then it also holds for $i + 1$ since
	\begin{equation*}
	q^*(i + 1) = \lambda q^*(i) \varphi\left(q^*(i)\right) \leq \lambda q^*(i) \leq \lambda^{i + 1}.
	\end{equation*}
	
	Suppose now that $q \in Q$ is an equilibrium point and let us prove that $q = q^*$. For an arbitrary $\varepsilon \in (0, 1)$, the function $\varphi$ is continuously differentiable in $[0, \varepsilon]$. In addition, $q(i) \leq \varepsilon$ for all large enough $i$ since $q \in \ell_1$. Thus, there exists $L \geq 0$ such that
	\begin{equation*}
	a_i(q) = q(i)\varphi\left(q(i)\right) \leq Lq(i) \quad \text{for all large enough} \quad i \geq 1.
	\end{equation*}
	In particular, $a_i(q) \to 0$ as $i \to \infty$. If we replace $\bq$ by $q$ in the right-hand side of \eqref{aux: general fluid differential equation}, then we can set the resulting expression equal to zero because $q$ is an equilibrium. Therefore,
	\begin{equation*}
	\lambda - q(1) = \lambda a_0(q) - q(1) = \sum_{i = 1}^\infty \lambda\left[a_{i - 1}(q) - a_i(q)\right] - \sum_{i = 1}^\infty \left[q(i) - q(i + 1)\right] = 0,
	\end{equation*}
	so $q(1) = \lambda$. Moreover, we have
	\begin{equation*}
	q(i + 1) = q(i) -\lambda\left[a_{i - 1}(q) - a_i(q)\right] \quad \text{for all} \quad i \geq 1.
	\end{equation*}
	
	Since $q(i) \leq q(1) = \lambda < 1$ for all $i \geq 1$, it follows that $a_i(q) = q(i)\varphi\left(q(i)\right)$ for all $i \geq 1$. We conclude that the right-hand side of the above equation is completely determined by $q(i - 1)$ and $q(i)$. But $q(0) = q^*(0)$ and $q(1) = q^*(1)$, so we must have $q = q^*$. 
\end{proof}

Below we prove that if condition \eqref{ass: sparsity condition} holds, then all fluid trajectories converge to the unique equilibrium point $q^*$ over time. The proof strategy is as in \cite{gamarnik2018delay} and \cite{vvedenskaya1996queueing}. First we note that the monotonicity property established in Lemma \ref{lem: monotonicity} implies that any fluid trajectory can be sandwiched between two solutions of \eqref{eq: fluid differential equation} that remain below and above the equilibrium $q^*$, respectively. Then we prove that both of these solutions converge to $q^*$ over time; we defer the proof of this proposition to Appendix \ref{app: auxiliary results}.

\begin{proposition}
	\label{prop: global attractivity}
	If \eqref{ass: sparsity condition} holds and $\lambda < 1$, then every fluid trajectory $\bq$ satisfies
	\begin{equation*}
	\lim_{t \to \infty} \bq(t, i) = q^*(i) \quad \text{for all} \quad i \geq 0.
	\end{equation*}
\end{proposition}

It follows from Corollary \ref{cor: unique fluid limit} that if the initial occupancy states $\bq_n(0)$ converge weakly to a deterministic $q \in Q$, then the occupancy processes $\bq_n$ approach the unique fluid trajectory with initial condition $q$ as $n \to \infty$. Hence, the equilibrium point $q^*$ provides information about the equilibrium behavior of large systems. In the next section we formalize this idea.

\section{Equilibrium behavior}
\label{sec: performance in equilibrium}

In this section we assume that $\calR_n$ is a Poisson process of rate $\mu_n$ that is independent of all the other stochastic primitives, listed in the beginning of Section \ref{sec: proof of the fluid limit}. This implies that the process $(\bX_n, \bG_n)$ is a continuous-time Markov chain. We establish that this process is ergodic provided that $\lambda_n < n$ and we show that the sequence of stationary occupancy states $q_n$ converges weakly to the equilibrium point $q^*$ when \eqref{ass: sparsity condition} holds and $\lambda < 1$. Then we provide a lower bound for $q^*$ when the mean of the limiting outdegree distribution is upper bounded by a constant and we establish when the lower bound is tight. In particular, we give a tight lower bound for the fraction of servers with at least $i$ tasks in equilibrium.

\subsection{Convergence of stationary distributions}
\label{sub: convergence of stationary distributions}

Suppose that $\bG_n$ is sampled from a random graph law $G_n$ with support $\supp(G_n)$. Then the continuous-time Markov chain $(\bX_n, \bG_n)$ takes values in $\N^n \times \supp(G_n)$, but we define its state space as the set of all elements of $\N^n \times \supp(G_n)$ that can be reached from a state of the form $(0, g)$ with $g \in \supp(G_n)$. In this way we obtain an irreducible Markov chain. The following proposition establishes that this Markov chain is also positive-recurrent provided that $\lambda_n < n$. This natural stability condition says that the total arrival rate of tasks is smaller than the combined service rate of all the servers.

\begin{proposition}
	\label{prop: ergodicity}
	If $\lambda_n < n$, then $(\bX_n, \bG_n)$ is positive-recurrent.
\end{proposition}

\begin{proof}
	It suffices to establish that $(0, g)$ is a positive-recurrent state of $(\bX_n, \bG_n)$ for any arbitrary $g \in \supp(G_n)$. For this purpose, let $\bY_n$ denote the process that describes the number of tasks across $n$ independent single-server queues, each with exponential service times of unit mean and Poisson arrivals of intensity $\rho_n \defeq \lambda_n / n < 1$. We will bound the mean recurrence time of $(0, g)$ for $(\bX_n, \bG_n)$ using the mean recurrence time of the empty system for $\bY_n$, which is finite since the Markov chain $\bY_n$ is ergodic.
	
	Define the occupancy process of $\bY_n$ by
	\begin{equation*}
	\br_n(t, i) \defeq \frac{1}{n}\sum_{j = 1}^n \ind{\bY_n(t, j) \geq i} \quad \text{for all} \quad t \geq 0 \quad \text{and} \quad i \geq 0.
	\end{equation*}
	The systems $(\bX_n, \bG_n)$ and $\bY_n$ can be constructed on a common probability space such that the arrivals and departures are coupled as in \cite[Proposition 2.1]{mukherjee2018asymptotically}. Specifically, at any given time let us attach the labels $\{1, \dots, n\}$ to the servers in each system, in increasing order of the queue lengths, with ties broken arbitrarily. The labels change over time and are unrelated to the identities of the servers, they are just auxiliary objects used for coupling the two systems. Both systems have the same arrival times and every task appears at a server with the same label in both systems. Also, for each label potential departures occur simultaneously in both systems as a Poisson process of unit rate, and a server finishes a task if and only if a potential departure occurs for the attached label and the server has at least one task. If $\bX_n(0) = \bY_n(0)$, then this construction is such that
	\begin{equation}
	\label{eq: coupling}
	\sum_{i = j}^\infty \bq_n(t, i) \leq \sum_{i = j}^\infty \br_n(t, i) \quad \text{for all} \quad t \geq 0 \quad \text{and} \quad j \geq 1
	\end{equation}
	with probability one. This holds because the label attached to the server to which the task is dispatched in $\left(\bX_n, \bG_n\right)$ is always smaller than or equal to the label attached to the server to which the task is dispatched in $\bY_n$; we refer to \cite[Appendix A]{mukherjee2018asymptotically} for details. Note that the resampling times and the graphs selected at each resampling time are independent of the history of $\bX_n$, and therefore also independent of $\bY_n$.
	
	We adopt the above construction with $\left(\bX_n(0), \bG_n(0)\right) = (0, g)$ and $\bY_n(0) = 0$. By \eqref{eq: coupling},
	\begin{equation*}
	\sum_{u = 1}^n \bX_n(t, u) = n \sum_{i = 1}^\infty \bq_n(t, i) \leq n \sum_{i = 1}^\infty \br_n(t, i) = \sum_{u = 1}^n \bY_n(t, u) \quad \text{for all} \quad t \geq 0.
	\end{equation*}
	If $\supp(G_n) = \{g\}$, then the fact that $\bY_n$ is positive-recurrent implies that $(0, g)$ is a positive-recurrent state of $(\bX_n, \bG_n)$ because $\bY_n = 0$ implies that $\bX_n = 0$. Therefore, we assume from now on that $G_n$ can take more than one value, or equivalently $P(G_n = g) < 1$.
	
	Denote the first recurrence time of state $(0, g)$ of $(\bX_n, \bG_n)$ by $\tau$ and let $\zeta_k$ denote the $k$-th passage time of state zero of $\bY_n$. Also, consider the disjoint events
	\begin{equation*}
	A_k \defeq \left\{\bG_n\left(\zeta_j\right) \neq g\ \text{for all}\ 1 \leq j < k\ \text{and}\ \bG_n\left(\zeta_k\right) = g\right\}
	\end{equation*}
	and let $\theta_k \defeq P\left(A_k\right)$ for all $k \geq 1$. Let $\nu \defeq P(Z < \zeta_1)$ denote the probability that $\bG_n$ is resampled between two successive visits to state zero of the process $\bY_n$, where $Z$ is exponentially distributed with mean $1 / \mu_n$ and independent of everything else. The union of the disjoint sets $A_k$ has probability one because
	\begin{align*}
	P\left(\bG_n(\zeta_j) \neq g\ \text{for all}\ j \geq 1\right) &= \lim_{k \to \infty} P\left(\bG_n(\zeta_k) \neq g\ \text{for all}\ 1 \leq j \leq k\right) \\
	&= \lim_{k \to \infty} P\left(\bG_n\left(\zeta_1\right) \neq g\right) \prod_{i = 1}^{k - 1}\cprob*{\bG_n\left(\zeta_{i + 1}\right) \neq g | \bG_n\left(\zeta_i\right) \neq g} \\
	&= \lim_{k \to \infty} \nu P\left(G_n \neq g\right)\left[1 - \nu + \nu P\left(G_n \neq g\right)\right]^{k - 1} = 0.
	\end{align*}
	Moreover, recall that $\bY_n = 0$ implies that $\bX_n = 0$. Hence, we have
	\begin{align*}
	E[\tau] = \sum_{k = 1}^\infty \expect*{\tau|A_k} \theta_k \leq \sum_{k = 1}^\infty \expect*{\zeta_k|A_k} \theta_k = \expect*{\zeta_1|A_1}\theta_1 + \sum_{k = 2}^\infty \expect*{\zeta_k|A_k} \theta_k.
	\end{align*}
	
	Note that $\zeta_k$ is not independent of $A_k$. For example, $A_k$ implies that $\bG_n$ is resampled before $\zeta_1$ and between $\zeta_{k - 1}$ and $\zeta_k$ when $k > 1$; in this case $\zeta_1$ is larger than the first resampling time and $\zeta_k$ is larger than the resampling time that follows $\zeta_{k - 1}$. But if we let
	\begin{align*}
	&B_1 \defeq \left\{\bG_n(0) = g,\ \bG_n(\zeta_1) \neq g\right\}, \\
	&B_2 \defeq \left\{\bG_n(0) \neq g,\ \bG_n(\zeta_1) \neq g\right\}, \\
	&B_3 \defeq \left\{\bG_n(0) \neq g,\ \bG_n(\zeta_1) = g\right\},
	\end{align*}
	then it is possible to write
	\begin{align*}
	\expect*{\zeta_k | A_k} &= \expect*{\zeta_1 + \sum_{i = 1}^{k - 2} \left(\zeta_{i + 1} - \zeta_i\right) + \zeta_k - \zeta_{k - 1} \Big| A_k} \\
	&= \expect*{\zeta_1 | B_1} + (k - 2) \expect*{\zeta_1 | B_2} + \expect*{\zeta_1 | B_3} \quad \text{for all} \quad k \geq 1.
	\end{align*}
	where the expectations in the last two expressions are with respect to coupled processes $(\bX_n, \bG_n)$ and $\bY_n$ with initial states such that $\bX_n(0) = \bY_n(0) = 0$ and $\bG_n(0) \neq g$; the specific value of $\bG_n(0)$ does not affect these expectations. For the equality in the second line, let $\zeta_0 \defeq 0$ and note that the distribution of $\zeta_{i + 1} - \zeta_i$ given $A_k$ is the same as the distribution of the first recurrence time $\zeta_1$ of state zero for $\bY_n$ given specific conditions on the graph $\bG_n$ at the start and end of the recurrence period, captured by the events $B_j$.
	\begin{enumerate}
		\item[(a)] If $i = 0$, then $\bG_n = g$ at the start of the recurrence period since $\bG_n(0) = g$, and $\bG_n \neq g$ at the end of the recurrence period by definition of $A_k$.
		\item[(b)] If $1 \leq i \leq k - 2$, then $\bG_n \neq g$ at the start and end of the recurrence period since $A_k$ states that this holds at the $i$-th and $(i + 1)$-th passage times of state zero for $\bY_n$.
		\item[(c)] Similarly, $\bG_n \neq g$ at the start of the recurrence period and $\bG_n = g$ at the end of the recurrence period by definition of $A_k$ if $i = k - 1$.
	\end{enumerate}
	
	The probabilities $\theta_k$ add up to one, thus
	\begin{equation*}
	E[\tau] \leq \expect*{\zeta_1 | A_1} \theta_1 + \expect*{\zeta_1 | B_1} + \expect*{\zeta_1 | B_2}\sum_{k = 2}^\infty (k - 2)\theta_k + \expect*{\zeta_1 | B_3}.
	\end{equation*}
	Since $\bY_n$ is positive-recurrent, all the conditional expectations on the right-hand side are finite, so it only remains to prove that the summation is finite as well.
	
	Note that for each $k > 1$ we have
	\begin{equation*}
	\theta_k = P\left(\bG_n\left(\zeta_1\right) \neq g\right) \prod_{i = 1}^{k - 2}\cprob*{\bG_n\left(\zeta_{i + 1}\right) \neq g | \bG_n\left(\zeta_i\right) \neq g} \cprob*{\bG_n\left(\zeta_k\right) = g | \bG_n\left(\zeta_{k - 1}\right) \neq g}.
	\end{equation*}
	Recall that $\nu = P(Z < \zeta_1)$ is the probability that the graph $\bG_n$ is resampled between two successive visits to state zero of $\bY_n$. Since $\bG_n(0) = g$, we have
	\begin{align*}
	\theta_k &= \nu P\left(G_n \neq g\right) \left[1 - \nu + \nu P\left(G_n \neq g\right)\right]^{k - 2} \nu P\left(G_n = g\right) \\
	&= \nu^2 P\left(G_n \neq g\right) P\left(G_n = g\right) \left[1 - \nu + \nu P\left(G_n \neq g\right)\right]^{k - 2}.
	\end{align*}
	
	We conclude that $E[\tau] < \infty$ because $\delta \defeq 1 - \nu + \nu P(G_n \neq g) < 1$ and thus
	\begin{equation*}
	\sum_{k = 2}^\infty (k - 2)\theta_k = \nu^2 P\left(G_n \neq g\right) P\left(G_n = g\right) \sum_{k = 1}^\infty k \delta^k = \frac{\nu^2 P\left(G_n \neq g\right) P\left(G_n = g\right) \delta}{\left(1 - \delta\right)^2}  < \infty.
	\end{equation*}
	This completes the proof.
\end{proof}

At any given time $t$, the occupancy state $\bq_n(t)$ is a deterministic function of $\bX_n(t)$, thus the distribution of $\bX_n(t)$ determines the distribution of $\bq_n(t)$. If $\lambda_n < n$, then the above proposition implies that the Markov chain $\left(\bX_n, \bG_n\right)$ has a unique stationary distribution. We define the stationary distribution of the occupancy state as the distribution of $\bq_n(t)$ determined by $\bX_n(t)$ when $\left(\bX_n, \bG_n\right)$ has the stationary distribution.

\begin{lemma}
	\label{lem: tightness and uniform integrability}
	Suppose that $\lambda < 1$ and $\lambda_n < n$ for all $n$. The sequence of stationary occupancy states $\set{q_n}{n \geq 1}$ is tight in $\ell_1$ and $\set{\norm{q_n}_1}{n \geq 1}$ is uniformly integrable.
\end{lemma}

\begin{proof}
	In order to prove that $\set{q_n}{n \geq 1}$ is tight in $\ell_1$, it suffices to show that
	\begin{equation*}
	\lim_{m \to \infty} \limsup_{n \to \infty} P\left(\sum_{i > m} q_n(i) > \varepsilon\right) = 0 \quad \text{for all} \quad \varepsilon > 0;
	\end{equation*}
	this follows from \cite[Lemma 2]{mukherjee2018universality}, which is also stated in Lemma \ref{lem: tightness in l1} of Appendix \ref{app: tightness of occupancy processes}.
	
	Consider the coupled construction introduced in the proof of Proposition \ref{prop: ergodicity}, with any initial state such that $\bX_n(0) = \bY_n(0)$. By ergodicity and \eqref{eq: coupling}, we have
	\begin{equation}
	\label{aux: bound for stationary distributions}
	P\left(\sum_{i > m} q_n(i) > \varepsilon\right) \leq P\left(\sum_{i > m} r_n(i) > \varepsilon\right) \quad \text{for all} \quad m \geq 0 \quad \text{and} \quad \varepsilon > 0.
	\end{equation}
	If $Y_n$ has the stationary distribution of $\bY_n$, then $P(Y_n(u) = k) = (1 - \rho_n)\rho_n^k$ and thus
	\begin{align*}
	P\left(\sum_{i > m} r_n(i) > \varepsilon\right) &= P\left(\frac{1}{n}\sum_{u = 1}^n \left[Y_n(u) - m\right]^+ > \varepsilon\right) \\
	&\leq \frac{1}{\varepsilon} E\left[\frac{1}{n}\sum_{u = 1}^n \left[Y_n(u) - m\right]^+\right] \\
	&= \frac{1}{\varepsilon} E\left[Y_n(1) - m\right]^+ = \frac{1}{\varepsilon} \sum_{k = m}^\infty (k - m)(1 - \rho_n)\rho_n^k = \frac{\rho_n^{m + 1}}{\varepsilon (1 - \rho_n)}.
	\end{align*}
	The first step uses Markov's inequality and the subsequent steps use the independence of the single-server queues and the steady-state distribution of each one such queue. Since
	\begin{equation*}
	\lambda < 1 \quad \text{and} \quad \lim_{m \to \infty} \lim_{n \to \infty} \frac{\rho_n^{m + 1}}{\varepsilon (1 - \rho_n)} = \lim_{m \to \infty} \frac{\lambda^{m + 1}}{\varepsilon (1 - \lambda)} = 0,
	\end{equation*}
	we see that $\set{q_n}{n \geq 1}$ is tight in $\ell_1$; recall that we had defined $\rho_n = \lambda_n / n$.
	
	In order to prove that the sequence $\set{\norm{q_n}_1}{n \geq 1}$ is uniformly integrable, consider Bernoulli trials with success probability $\rho_n$. The probability that $Y_n(u) = k$ is equal to the probability that a failure occurs after $k$ successful trials. Hence,
	\begin{equation*}
	P\left(\sum_{u = 1}^n Y_n(u) = k\right)
	\end{equation*}
	is equal to the probability that there are $k$ successful trials before $n$ failures occur;  i.e., the total number of tasks is negative binomial with parameters $n$ and $\rho_n$. By \eqref{aux: bound for stationary distributions},
	\begin{equation*}
	E\left[\norm{q_n}_1^2\right] = \sum_{k = 1}^\infty P\left(\norm{q_n}_1^2 \geq k\right) \leq \sum_{k = 1}^\infty P\left(\norm{r_n}_1^2 \geq k\right) = E\left[\norm{r_n}_1^2\right]. 
	\end{equation*}
	Moreover, the total number of tasks in the system is
	\begin{equation*}
	 n\left(\norm{r_n}_1 - 1\right) = n \sum_{i = 1}^\infty r_n(i) = \sum_{i = 1}^\infty ni\left[r_n(i) - r_n(i + 1)\right] = \sum_{u = 1}^n Y_n(u).
	\end{equation*}
	Indeed, $ni\left[r_n(i) - r_n(i + 1)\right]$ is the number of tasks in servers with exactly $i$ tasks. Thus,
	\begin{align*}
	E\left[\norm{q_n}_1^2\right] \leq E\left[\norm{r_n}_1^2\right] &= E\left[\left(1 + \frac{1}{n}\sum_{u = 1}^n Y_n(u)\right)^2\right] \\
	&= 1 + \frac{2}{n}E\left[\sum_{u = 1}^n Y_n(u)\right] + \frac{1}{n^2}E\left[\left(\sum_{u = 1}^n Y_n(u)\right)^2\right]\\
	&= 1 + \frac{2\rho_n}{1 - \rho_n} + \frac{1}{n^2}\left[\left(\frac{n\rho_n}{1 - \rho_n}\right)^2 + \frac{n\rho_n}{(1 - \rho_n)^2}\right].
	\end{align*}
	The right-hand side has a finite limit as $n \to \infty$ and thus the left-hand side is uniformly bounded across $n$, which implies that $\set{\norm{q_n}_1}{n \geq 1}$ is uniformly integrable.
\end{proof} 

As noted in Section \ref{sub: global attractivity}, the equilibrium point $q^*$ provides some information about the behavior of a large system in steady state. The following theorem formalizes this idea in the situation where the graph is resampled as a Poisson process. The proof relies on an interchange of limits argument.

\begin{theorem}
	\label{the: interchange of limits}
	Assume that $\calR_n$ is a Poisson process of rate $\mu_n$ for all $n$ and that conditions \eqref{eq: limiting degree distribution condition}, \eqref{eq: arrival and resampling rates conditions} and \eqref{ass: sparsity condition} hold. In addition, suppose that $\lambda < 1$ and $\lambda_n < n$ for all $n$. Then the sequence of stationary occupancy states $q_n$ converges weakly in $\ell_1$ to the equilibrium $q^*$.
\end{theorem}

\begin{proof}
	By Lemma \ref{lem: tightness and uniform integrability}, the sequence of stationary occupancy states $\set{q_n}{n \geq 1}$ is tight in $\ell_1$. It follows from Prohorov's theorem that every increasing sequence of natural numbers has a subsequence $\calK$ such that $\set{q_k}{k \in \calK}$ converges weakly in $\ell_1$ to some random variable~$q$. Therefore, it is enough to prove that $q = q^*$ almost surely for each $\calK$.
	
	Let $\calK \subset \N$ be an increasing sequence such that $q_k \Rightarrow q$ in $\ell_1$ as $k \to \infty$ for some random variable $q$. In addition, let $\bq_k$ be the stationary occupancy process for each $k \in \calK$. It follows from Corollary \ref{cor: unique fluid limit} that $\bq_k \Rightarrow \bq$ in $D_{\ell_1}[0, \infty)$ where $\bq$ solves \eqref{eq: fluid differential equation} almost surely.
	
	By \cite[Theorem 23.9]{kallenberg2021foundations}, there exists $T \subset [0, \infty)$ such that $\bq_k(t) \Rightarrow \bq(t)$ in $\ell_1$ as $k \to \infty$ for all $t \in T$ and $T$ is dense in $[0, \infty)$. Note that $\bq_k(t)$ has the same distribution as $q_k$ for all $t$, thus $\bq(t)$ has the same distribution as $q$ for all $t \in T$. Also, Proposition \ref{prop: global attractivity} yields
	\begin{equation*}
	\lim_{t \to \infty} \bq(t, i) = q^*(i) \quad \text{for all} \quad i \geq 0
	\end{equation*}
	with probability one. Hence, $\bq(t, i) \Rightarrow q^*(i)$ in $\R$ as $t \to \infty$. Since $\bq(t, i)$ has the same distribution as $q(i)$ for all $i \geq 0$ and $t \in T$, this implies that $q(i)$ has the same distribution as the point mass at $q^*(i)$. We conclude that $q(i) = q^*(i)$ almost surely for all $i \geq 1$.
\end{proof}

Suppose that $\lambda_n < n$ and denote the stationary occupancy state by $q_n$. We define
\begin{equation*}
R_n \defeq \frac{n}{\lambda_n}E\left[\norm{q_n}_1 - 1\right] = \frac{n}{\lambda_n}E\left[\sum_{i = 1}^\infty q_n(i)\right].
\end{equation*}
Because $n\left(\norm{q_n}_1 - 1\right)$ is the total number of tasks in the system, Little's law implies that $R_n$ is the mean response time of tasks in steady state. Next we compute the limit of this quantity as the number of servers grows large; we defer the proof to Appendix~\ref{app: auxiliary results}.

\begin{corollary}
	\label{cor: total number of tasks in equilibrium}
	Suppose that the conditions of Theorem \ref{the: interchange of limits} hold. If $q_n$ denotes the stationary occupancy state of the system with $n$ servers, then
	\begin{equation*}
	R \defeq \frac{\norm{q^*}_1 - 1}{\lambda} = \lim_{n \to \infty} R_n.
	\end{equation*}
\end{corollary}

\subsection{Phase transition for $q^*$}
\label{sub: isolated servers and performance}

By Theorem \ref{the: interchange of limits}, the stationary occupancy state approaches the equilibrium point $q^*$ as the size of the system grows large. In addition, recall that $q^*$ is determined by the limiting outdegree distribution through the probability generating function $\varphi$. Thus, we may reach some conclusions about the impact of the outdegree distribution on the performance of a large load balancing system by studying how different properties of the limiting outdegree distribution affect $q^*$. For example, the next result establishes that there exists a phase transition as the limiting outdegree distribution ceases to have mass at zero. As a result, performance is particularly negative when the outdegree distribution has mass at zero, no matter how small; the result holds also when condition \eqref{ass: sparsity condition} does no hold.

\begin{proposition}
	\label{prop: isolated servers}
	Suppose that $m \defeq \min\set{d \geq 0}{p(d) > 0} < \infty$. For all $i \geq 2$,
	\begin{center}
		\begin{tabular}{l c r}
			$\lambda\left[\lambda p(0)\right]^{i - 1} \leq q^*(i) \leq \lambda\left[\lambda \left(1 - p(\infty)\right)\right]^{i - 1}$ & if & $m = 0,$ \\
			$\lambda^{(m + 1)^{i - 1}}\left[\lambda p(m)\right]^{\frac{(m + 1)^{i - 1} - 1}{m}} \leq q^*(i) \leq \lambda^{(m + 1)^{i - 1}}\left[\lambda \left(1 - p(\infty)\right)\right]^{\frac{(m + 1)^{i - 1} - 1}{m}}$ & if & $m > 0.$
		\end{tabular}
	\end{center}
	In particular, $q^*$ is bounded between two geometric sequences if $m = 0$ and $q^*$ is bounded between two sequences that decay doubly exponentially if $m > 0$.
\end{proposition}

\begin{proof}
	Note that $p(m)x^m \leq \varphi(x) \leq \left[1 - p(\infty)\right]x^m$ for all $x \in [0, 1]$. Hence,
	\begin{equation*}
	\lambda p(m) \left[q^*(i - 1)\right]^{m + 1} \leq q^*(i) \leq \lambda \left[1 - p(\infty)\right] \left[q^*(i - 1)\right]^{m + 1} \quad \text{for all} \quad i \geq 2.
	\end{equation*}
	It follows by induction that
	\begin{equation*}
	\left[\lambda p(m)\right]^{\sum_{j = 0}^{i - 2} (m + 1)^j} \lambda^{(m + 1)^{i - 1}} \leq q^*(i) \leq \left[\lambda \left(1 - p(\infty)\right)\right]^{\sum_{j = 0}^{i - 2} (m + 1)^j} \lambda^{(m + 1)^{i - 1}} \quad \text{for all} \quad i \geq 2.
	\end{equation*}
	The claim is a straightforward consequence of these two inequalities.
\end{proof}

\begin{figure}
	\centering
	\includegraphics{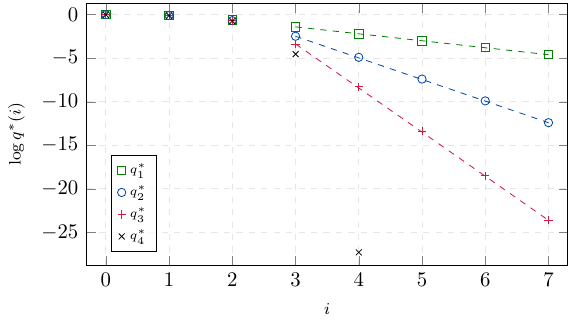}
	\caption{Equilibrium point for $\lambda = 0.9$ and distinct limiting outdegree distributions with mean $d = 5$: for $q_1^*$ the limiting outdegree distribution has mass only at $0$ and $2d$, for $q_2^*$ a uniform distribution on outdegrees between $0$ and $2d$ was used, a Poisson distribution was used for $q_3^*$ and a deterministic distribution for $q_4^*$. The tail of $\log q^*(i)$ decays almost linearly for the limiting outdegree distributions with mass at zero.}
	\label{fig: fixed point for different degree distributions}
\end{figure}

The situation where $m = 0$ corresponds to a random graph law such that the average fraction of servers that cannot forward arriving tasks to other servers is positive. In this case $q^*$ is bounded between two geometric sequences, regardless of how small the average fraction of isolated servers is. However, $q^*$ is bounded between two sequences that decay doubly exponentially if the mean fraction of isolated servers is zero; i.e., $m > 0$. Figure \ref{fig: fixed point for different degree distributions}, shows how $q^*$ decays nearly geometrically for several limiting outdegree distributions with mass at zero. In addition, Table \ref{tab: W and P} illustrates the stark contrast between $m = 0$ and $m > 0$. Differences in the mean delay are fairly minor, but $q^*$ decays much slower if $m = 0$.

The geometric lower and upper bounds may be intuitively understood as follows. First, observe that tasks depart from servers with at least $i$ tasks at rate $q^*(i)$, and tasks are dispatched to a server with at least $i - 1$ tasks at a rate that is larger than or equal to $\lambda p(0) q^*(i - 1)$. In steady state, the departure rate from servers with at least $i$ tasks equals the rate at which tasks are dispatched to servers with at least $i - 1$ tasks. Thus,
\begin{align*}
q^*(i) \geq \lambda p(0) q^*(i - 1) \geq \left[\lambda p(0)\right]^{i - 1} q^*(1) = \lambda \left[\lambda p(0)\right]^{i - 1}.
\end{align*}
The geometric upper bound may be explained using a similar heuristic argument, noting that the rate at which tasks are dispatched to servers with at least $i - 1$ tasks is at most $\lambda [1 - p(\infty)]q^*(i - 1)$.  Observe here that the probability that an arriving task is diverted away from a busy server is at least $p(\infty)$ since $m = 0$ implies that in the limiting regime there are idle servers in the system with probability one.

The geometric decay of the equilibrium occupancy state holds even when the average fraction of isolated servers is arbitrarily small but positive. In other words, to achieve favorable performance, it does not matter so much to have a large average outdegree, but rather to avoid situations where some nodes have outdegree zero. For example, consider a topology with $p(2) = 1$ consisting entirely of isolated couples of servers that can forward tasks from one to the other. The decay is substantially faster in this case than in a topology where 1\% of the servers cannot forward tasks to other servers and the other 99\% of the servers are fully connected; i.e., $p(0) = 0.01$ and $p(\infty) = 0.99$.

\newcolumntype{C}[1]{>{\centering\arraybackslash}m{#1}}
\begin{table}
	\centering
	{\footnotesize
		\begin{tabular}{|C{7mm}|C{14mm}|C{14mm}|C{14mm}|C{14mm}|C{14mm}|C{14mm}|C{14mm}|}
			\hline
			$k$ & $R$ & $q^*(1)$ & $q^*(2)$ & $q^*(3)$ & $q^*(4)$ & $q^*(5)$ & $q^*(6)$ \\
			\hline
			$0$ & $1.7778$ & $0.9000$ & $0.5905$ & $0.1094$ & $0.0001$ & $0.0000$ & $0.0000$ \\
			\hline
			$1$ & $1.7941$ & $0.9000$ & $0.5927$ & $0.1214$ & $0.0006$ & $0.0000$ & $0.0000$ \\
			\hline
			$2$ & $1.8528$ & $0.9000$ & $0.5993$ & $0.1603$ & $0.0079$ & $0.0000$ & $0.0000$ \\
			\hline
			$3$ & $2.0513$ & $0.9000$ & $0.6103$ & $0.2342$ & $0.0712$ & $0.0214$ & $0.0064$ \\
			\hline
		\end{tabular}
	}
	\caption{$R$ and $q^*(i)$ for $\lambda = 0.9$ and limiting outdegree distributions that are uniform in $\{3 - k, 3, 3 + k\}$.}
	\label{tab: W and P}
\end{table}

\subsection{Lower bound for $q^*$}
\label{sub: optimal performance}

Corollary \ref{cor: total number of tasks in equilibrium} gives the limit of the steady-state mean response time of tasks as $n \to \infty$. The value $R$ of this limit is minimal if and only if $q^*(i) = 0$ for all $i > 1$, or equivalently $p(\infty) = 1$. Note that this corresponds to a somewhat dense limiting regime where the mean outdegree approaches infinity as $n \to \infty$. Indeed, $p(\infty) = 1$ implies that for each $k$ there exists $m$ such that $p_n(d) < 1 / 2k$ for all $d < k$ and $n \geq m$. This implies that
\begin{equation*}
E\left[D_n\right] = \sum_{d = 0}^{n - 1} d p_n(d) \geq \sum_{d = k}^{n - 1} d p_n(d) \geq \frac{k}{2} \quad \text{for all} \quad n \geq m,
\end{equation*}
and since $k$ is arbitrary, we conclude that $E[D_n] \to \infty$ as $n \to \infty$ when $p(\infty) = 1$

While the steady-state mean response time is theoretically optimal in this dense regime, in practice the communication overhead increases with the average outdegree; because this quantity determines how many neighbors a server needs to poll on average before it can forward a task. Therefore, from a practical perspective it is more relevant to consider the situation where $E[D_n]$ is bounded. Below we derive the limiting outdegree distribution $p$ that minimizes the steady-state mean response time $R$ when
\begin{equation}
\label{ass: bounded mean degree}
p(\infty) = 0 \quad \text{and} \quad \sum_{i = 0}^\infty ip(i) \leq d
\end{equation}
for some given $d \geq 0$; i.e., the mean of the limiting outdegree distribution is at most $d$.

\begin{lemma}
	\label{lem: optimal fixed point}
	Fix $c \in (0, 1)$ and suppose that \eqref{ass: bounded mean degree} holds. Then
	\begin{equation*}
	\varphi(c) \geq \left(\floor{d} + 1 - d\right) c^{\floor{d}} + \left(d - \floor{d}\right) c^{\floor{d} + 1} \geq c^d.
	\end{equation*}
	Furthermore, $\varphi(c) = c^d$ if and only if $d \in \N$ and $p(d) = 1$.
\end{lemma}

\begin{proof}
	Consider the function $\map{f}{\R}{\R}$ such that, for all $k \in \Z$, we have $f(k) = c^k$ and the restriction of $f$ to $[k, k + 1]$ is linear. Specifically,
	\begin{equation*}
	f(x) \defeq \left(\floor{x} + 1 - x\right) c^{\floor{x}} + \left(x - \floor{x}\right) c^{\floor{x} + 1} \quad \text{for all} \quad x \in \R.
	\end{equation*}
	The convexity of $x \mapsto c^x$ implies that $f(x) \geq c^x$ for all $x \in \R$. Moreover,
	\begin{equation*}
	\varphi(c) = \sum_{i = 0}^\infty p(i)c^i = \sum_{i = 0}^\infty p(i)f(i) \geq f\left(\sum_{i = 0}^\infty p(i)i\right) \geq f(d) \geq c^d,
	\end{equation*}
	since $f$ is convex and decreasing.
	
	Suppose that $p(j) < 1$ for some $j$. The strict convexity of $x \mapsto c^x$ implies that
	\begin{equation*}
	\varphi(c) = \sum_{i \neq j} p(i)c^i + p(j)c^j \geq \left[1 - p(j)\right]c^{\frac{1}{1 - p(j)}\sum_{i \neq j} p(i)i} + p(j)c^j > c^{\sum_{i = 0}^\infty p(i)i} \geq c^d.
	\end{equation*}
	Therefore, it is necessary that $p$ is a deterministic probability measure in order to achieve the lower bound $c^d$, and it is straightforward to check that for a deterministic $p$ the lower bound is only attained if $d \in \N$ and $p(d) = 1$.
\end{proof}

\begin{proposition}
	\label{prop: optimal fixed point}
	Suppose that \eqref{ass: bounded mean degree} holds. Then
	\begin{equation*}
	q^*(i) \geq \begin{cases}
	\lambda^i & \text{if} \quad d = 0, \\
	\lambda^{\frac{(d + 1)^i - 1}{d}} & \text{if} \quad d > 0,
	\end{cases}
	\quad \text{for all} \quad i \geq 1.
	\end{equation*}
	If $d \in \N$ and $p(d) = 1$, then we have equality for all $i \geq 1$, and if the latter conditions do not hold, then the inequality is strict for all $i \geq 2$. 
\end{proposition}

\begin{proof}
	By Proposition \ref{prop: fixed point} and Lemma \ref{lem: optimal fixed point},
	\begin{equation*}
	q^*(i) = \lambda q^*(i - 1)\varphi\left(q^*(i - 1)\right) \geq \lambda \left[q^*(i - 1)\right]^{d + 1} \quad \text{for all} \quad i \geq 2.
	\end{equation*}
	Since $q^*(1) = \lambda$, it follows by induction that
	\begin{equation*}
	q^*(i) \geq \lambda^{\sum_{j = 0}^{i - 1} (d + 1)^j} = \begin{cases}
	\lambda^i & \text{if} \quad d = 0, \\
	\lambda^{\frac{(d + 1)^i - 1}{d}} & \text{if} \quad d > 0,
	\end{cases}
	\quad \text{for all} \quad i \geq 1.
	\end{equation*}
	By Lemma \ref{lem: optimal fixed point}, the above inequality is strict for all $i \geq 2$ unless $d \in \N$ and $p(d) = 1$, which results in equality for all $i \geq 1$.
\end{proof}

Under the sparsity constraint \eqref{ass: bounded mean degree}, the equilibrium $q^*$ is minimized coordinatewise if and only if $d \in \N$ and the limiting outdegree distribution is deterministic with $p(d) = 1$. In particular, the minimum value of $R$ is only attained for this limiting outdegree distribution and the fraction of servers with at least $i$ tasks is minimal for each $i$. Also, the numerical results in Figure~\ref{fig: fixed point for different degree distributions} and Table~\ref{tab: W and P} suggest that $q^*$ decreases coordinatewise as the limiting outdegree distribution becomes more concentrated around $d$.

\begin{remark}
	The lower bound in Proposition~\ref{prop: optimal fixed point} is only tight when $d \in \N$, but it is possible to derive a lower bound that is tight also when $d \notin \N$.  This lower bound is obtained in a similar fashion but invoking the inequality $\varphi(c) \geq \left(\floor{d} + 1 - d\right)c^{\floor{d}} + \left(d - \floor{d}\right)c^{\floor{d} + 1}$ instead of the inequality $\varphi(c) \geq c^d$, which is only tight for $d \in \N$.  However, this more refined lower bound is rather unwieldy.
\end{remark}

\section{Proof of the fluid limit}
\label{sec: proof of the fluid limit}

In this section we prove Theorem \ref{the: fluid limit}. As a first step, we define the processes $\set{\bq_n}{n \geq 1}$ and $\set{\bX_n}{n \geq 1}$ as deterministic functions of the following stochastic primitives.

\begin{enumerate}
	\item[(a)] \emph{Driving Poisson processes:} independent Poisson processes $\calN^a$ and $\set{\calN_i^d}{i \geq 1}$ of unit intensity, for counting the arrivals and departures of tasks, respectively.
	
	\item[(b)] \emph{Selection variables:} independent random variables $\set{u_n^m, U_{i, n}^m}{i, m, n \geq 1}$ such that $u_n^m$ is uniform in $V_n$ and $U_{i, n}^m$ is uniform in $[0, 1)$ for all $m$ and $n$.
	
	\item[(c)] \emph{Initial conditions:} a sequence $\set{X_n}{n \geq 1}$ of random vectors describing the initial number of tasks at each server and such that the corresponding sequence of occupancy states $\set{q_n}{n \geq 1}$ is tight in $\ell_1$.
	
	\item[(d)] \emph{Random graphs:} independent random graphs $\set{G_n^m}{m \geq 0, n \geq 1}$ such that for each fixed $n$ all the graphs $\set{G_n^m}{m \geq 0}$ have node set $V_n$ and a common distribution that satisfies Assumption \ref{ass: fluid limit conditions} and is invariant under permutations of the nodes.
	
	\item[(e)] \emph{Resampling processes:} c\`adl\`ag processes $\set{\calR_n}{n \geq 1}$ satisfying Assumption \ref{ass: fluid limit conditions}.
\end{enumerate}
The sample paths of $\bq_n$ and $\bX_n$ are constructed on the completion of the product of the probability spaces where the stochastic primitives are defined. This construction is such that certain stochastic equations hold, as we explain in the following section.

\subsection{Stochastic equations}
\label{sub: stochastic equations}

For each fixed $n$, the times at which the graph is sampled are $\sigma_n^0 \defeq 0$ and the jump times $\set{\sigma_n^m}{m \geq 1}$ of the resampling process $\calR_n$. Specifically,
\begin{equation*}
\bG_n(t) = G_n^0 \quad \text{if} \quad \sigma_n^0 \leq t \leq \sigma_n^1 \quad \text{and} \quad \bG_n(t) = G_n^m \quad \text{if} \quad \sigma_n^m < t \leq \sigma_n^{m + 1}.
\end{equation*}
In addition, tasks arrive at the jump times $\set{\tau_n^m}{m \geq 1}$ of the arrival process $\calN_n^a$ defined by $\calN_n^a(t) \defeq \calN^a(\lambda_n t)$. At time $\tau_n^m$, a task appears at server $u_n^m$ and we let
\begin{equation*}
I_n^m(X, i) \defeq \ind{\min\set{X(v)}{v = u_n^m\ \text{or}\ \left(u_n^m, v\right) \in \bE_n\left(\tau_n^{m -}\right)} \geq i} \quad \text{for all} \quad X \in \N^n.
\end{equation*}
If $X(v)$ represents the number of tasks at server $v$ right before $\tau_n^m$, then $I_n^m(X, i) = 1$ if and only if the task arriving at time $\tau_n^m$ is dispatched to a server with at least $i$ tasks.

The processes $\bq_n$ and $\bX_n$ are constructed in Appendix \ref{app: construction of sample paths} as deterministic functions of the stochastic primitives within a set of probability one. Both are piecewise constant c\`adl\`ag  processes defined on $[0, \infty)$ and have jumps of size $1 / n$ and jumps of unit size, respectively. Moreover, the following stochastic equations hold:
\begin{equation}
\label{eq: stochastic equations}
\begin{split}
\bq_n(t, i) &= \bq_n(0, i) + \frac{1}{n} \sum_{m = 1}^{\calN_n^a(t)} \left[I_n^m\left(\bX_n\left(\tau_n^{m-}\right), i - 1\right) - I_n^m\left(\bX_n\left(\tau_n^{m-}\right), i\right)\right] \\
&- \frac{1}{n}\calN_i^d\left(n\int_0^t\left[\bq_n(s, i) - \bq_n(s, i + 1)\right]ds\right)
\end{split}
\end{equation}
for all $i \geq 1$ and $t \geq 0$ with probability one. Indeed, the first term on the right is the initial occupancy state, the second term counts the arrivals to servers with exactly $i - 1$ tasks and the third term counts the departures from servers with exactly $i$ tasks.

\subsection{Decomposition of the equations}
\label{sub: decomposition of the equations}

Next we decompose the right-hand side of \eqref{eq: stochastic equations} in a way that simplifies its analysis. For this purpose we need to introduce some notation. Consider the function defined by
\begin{equation}
\label{eq: uniformly random subsets}
\alpha_n(d, x) \defeq \begin{cases}
\prod_{m = 0}^{d - 1} \left(\frac{n x - m}{n - m}\right)^+ & \text{if} \quad d \leq n, \\
0 & \text{if} \quad d > n,
\end{cases} \quad \text{for all} \quad x \in [0, 1].
\end{equation}
If $n x \in \N$ and a subset of $\{1, \dots, n\}$ consisting of $d \leq n$ elements is drawn uniformly at random, then $\alpha_n(d, x)$ is the probability that this subset is contained in $\{1, \dots, n x\}$.

Suppose that the fraction of servers with at least $i$ tasks is $x$ when a task arrives and there is no information about the current graph. The server $u$ that initially receives the task is uniformly random and thus has outdegree $d$ with probability $p_n(d)$. Also, given that the outdegree of $u$ is $d$, the probability that all the servers in the neighborhood of $u$ have at least $i$ tasks is $\alpha_n(d + 1, x)$ since the law of the graph is invariant under permutations of the nodes. Thus, the probability that the task is sent to a server with at least $i$ tasks is
\begin{equation}
\label{eq: dispatching probabilities}
\beta_n(x) \defeq \sum_{d = 0}^{n - 1} \alpha_n\left(d + 1, x\right) p_n(d) \quad \text{for all} \quad x \in [0, 1].
\end{equation}

In particular, we have
\begin{equation}
\label{eq: expectation of assigning indicators}
E\left[I_n^m(X, i)\right] = \beta_n\left(q(i)\right) \quad \text{for all} \quad X \in \N^n \quad \text{and} \quad q(i) \defeq \frac{1}{n}\sum_{u = 1}^n \ind{X(u) \geq i}.
\end{equation}
The expectation is taken with respect to the stochastic primitives, or more precisely just with respect to the graph right before $\tau_n^m$ and the server $u_n^m$ at which the task originally appears; indeed, note that $I_n^m(X, i)$ only depends on these two random variables.

\begin{remark}
	\label{rem: arrivals with a given graph}
	The above arguments break down if the graph at the time of the arrival is given. In that case the probability that a task is dispatched to a server with at least $i$ tasks depends on the given graph and the number of tasks at each individual server.
\end{remark}

Consider the processes defined by
\begin{equation*}
\begin{split}
\bar{\bq}_n(t) \defeq \sum_{m = 0}^\infty \bq_n\left(\sigma_n^m\right)\ind{\sigma_n^m \leq t < \sigma_n^{m + 1}} \quad \text{and} \quad \bar{\bX}_n(t) \defeq \sum_{m = 0}^\infty \bX_n\left(\sigma_n^m\right)\ind{\sigma_n^m \leq t < \sigma_n^{m + 1}},
\end{split}
\end{equation*}
which correspond to sampling the state of the system at the resampling times. Also, let
\begin{equation*}
q_n^m \defeq \bq_n\left(\tau_n^{m-}\right), \quad \bar{q}_n^m \defeq \bar{\bq}_n\left(\tau_n^{m-}\right), \quad X_n^m \defeq \bX_n\left(\tau_n^{m-}\right) \quad \text{and} \quad \bar{X}_n^m \defeq \bar{\bX}_n\left(\tau_n^{m-}\right).
\end{equation*}
We define processes $\bL_n$, $\bM_n$ and $\bu_n$ as follows. If at most one task arrives between any two successive resampling times, then we say that $\calR_n$ \emph{separates arrivals fully} and we let
\begin{equation}
\label{eq: decomposition for graph resampled between every two arrivals}
\begin{split}
&\bL_n(t, i) \defeq 0, \\
&\bM_n(t, i) \defeq \frac{1}{n} \sum_{m = 1}^{\calN_n^a(t)} \left[I_n^m\left(X_n^m, i\right) - \beta_n\left(q_n^m(i)\right)\right], \\
&\bu_n(t, i) \defeq \frac{1}{n}\sum_{m = 1}^{\calN_n^a(t)} \beta_n\left(q_n^m(i)\right),
\end{split}
\end{equation}
for all $i \geq 0$ and $t \geq 0$. If $\calR_n$ does not separate arrivals fully, then we define
\begin{equation}
\label{eq: general decomposition}
\begin{split}
&\bL_n(t, i) \defeq \frac{1}{n} \sum_{m = 1}^{\calN_n^a(t)} \left[I_n^m\left(X_n^m, i\right) - I_n^m\left(\bar{X}_n^m, i\right)\right], \\
&\bM_n(t, i) \defeq \frac{1}{n} \sum_{m = 1}^{\calN_n^a(t)} \left[I_n^m\left(\bar{X}_n^m, i\right) - \beta_n\left(\bar{q}_n^m(i)\right)\right], \\
&\bu_n(t, i) \defeq \frac{1}{n}\sum_{m = 1}^{\calN_n^a(t)} \beta_n\left(\bar{q}_n^m(i)\right).
\end{split}
\end{equation}
Note that $X_n^m$ and $q_n^m$ have been replaced by $\bar{X}_n^m$ and $\bar{q}_n^m$ in the definitions of $\bM_n$ and $\bu_n$ provided in \eqref{eq: general decomposition}. Also, the sum of the three processes is the same under \eqref{eq: decomposition for graph resampled between every two arrivals} and \eqref{eq: general decomposition}.

\begin{remark}
	\label{rem: decomposition for resampling with every arrival}
	If the resampling process separates arrivals fully, then the graph is resampled between any two consecutive arrival times. The definitions provided in \eqref{eq: decomposition for graph resampled between every two arrivals} significantly simplify the proof of the fluid limit when all the resampling processes $\calR_n$ separate arrivals fully. But this simplification is no longer possible when successive arrivals have a positive probability of being dispatched using the same graph. In this case we must resort to \eqref{eq: general decomposition}.
\end{remark}

The stochastic equations \eqref{eq: stochastic equations} can now be expressed as follows:
\begin{equation}
\label{eq: decomposition of stochastic equations}
\bq_n = \bq_n(0) + \bv_n + \bw_n,
\end{equation}
where for all $i \geq 1$ and $t \geq 0$,  the vanishing process $\bv_n$ is defined by
\begin{equation}
\label{eq: vanishing process}
\begin{split}
\bv_n(t, i) &\defeq \bL_n(t, i - 1) - \bL_n(t, i) + \bM_n(t, i - 1) - \bM_n(t, i) \\
&+ \int_0^t \left[\bq_n(s, i) - \bq_n(s, i + 1)\right]ds - \frac{1}{n}\calN_i^d\left(n\int_0^t\left[\bq_n(s, i) - \bq_n(s, i + 1)\right]ds\right),
\end{split}
\end{equation}
and the drift process $\bw_n$ is defined by
\begin{equation}
\label{eq: drift process}
\begin{split}
\bw_n(t, i) \defeq \bu_n(t, i - 1) - \bu_n(t, i) - \int_0^t \left[\bq_n(s, i) - \bq_n(s, i + 1)\right]ds.
\end{split}
\end{equation}

The road map for proving Theorem \ref{the: fluid limit} is as follows. In Section \ref{sub: events separation property} we formally define the pseudo-separation property mentioned in Assumption \ref{ass: fluid limit conditions} and we prove Proposition \ref{prop: admissible resampling processes}. In Section \ref{sub: vanishing processes} we show that $\bv_n \Rightarrow 0$ as $n \to \infty$ with respect to a suitable topology. Informally, this implies that the asymptotic behavior of $\bq_n$ is essentially captured by \eqref{eq: drift process} in the limit as $n \to \infty$. Then we show in Section \ref{sub: drift processes} that $\set{\bq_n}{n \geq 1}$ is relatively compact in $D_{\ell_1}[0, \infty)$, i.e., every subsequence of $\set{\bq_n}{n \geq 1}$ has a further subsequence that converges weakly in $D_{\ell_1}[0, \infty)$ to some process $\bq$. Finally, \eqref{eq: drift process} is used to establish that the limit $\bq$ of any convergent subsequence satisfies \eqref{eq: fluid dynamics} almost surely. Essentially, the first two terms of \eqref{eq: drift process} yield the first term of \eqref{eq: fluid dynamics} and the last term of \eqref{eq: drift process} gives the last term of \eqref{eq: fluid dynamics}.

\subsection{Pseudo-separation property}
\label{sub: events separation property}

Below we define the pseudo-separation property mentioned in Assumption \ref{ass: fluid limit conditions}. This property applies to sequences of resampling processes $\calR_n$ and concerns the asymptotic behavior of the processes as $n \to \infty$. In contrast, the property of separating arrivals fully applies to individual resampling processes $\calR_n$; i.e., the number of servers $n$ is fixed.

\begin{definition}
	\label{def: separate events sufficiently}
	The resampling process $\calR_n$ is said to separate arrivals fully if at most one task arrives between any two successive resampling times with probability one. Consider the following random variables:
	\begin{equation*}
	\begin{split}
	&\Delta_n(T) \defeq \sup \set{t - \sigma_n^m}{m \leq \calR_n(T)\ \text{and}\ \sigma_n^m \leq t \leq \min\left\{\sigma_n^{m + 1}, T\right\}}, \\
	&\Sigma_n(T) \defeq \sum_{m = 1}^{\calR_n(T) + 1} \frac{1}{n^2}\left[\left(d_n^- + 1\right)\left(A_n^m + D_n^m - 1\right)A_n^m + \left(A_n^m\right)^2\right],
	\end{split}
	\end{equation*}
	where $A_n^m$ and $D_n^m$ are the number of arrivals and departures in $\left(\sigma_n^{m - 1}, \sigma_n^m\right]$, respectively. Also, let $\calK$ be the set of indexes $k$ such that $\calR_k$ does not separate arrivals fully. The resampling processes $\set{\calR_n}{n \geq 1}$ are said to pseudo-separate events if $\calK$ is finite or $\calK$ is infinite and the following limits hold:
	\begin{equation}
	\label{eq: event separation property}
	\Delta_k(T) \Rightarrow 0 \quad \text{as} \quad k \to \infty \quad \text{and} \quad \lim_{k \to \infty} E\left[\Sigma_k(T)\right] = 0 \quad \text{for all} \quad T \geq 0,
	\end{equation}
	where both limits are taken over the indexes $k \in \calK$.  
\end{definition}

It is possible that all the resampling processes $\calR_n$ separate arrivals fully and $E\left[\Sigma_n(T)\right]$ does not approach zero with $n$ for any $T \geq 0$. For example, if the resampling times coincide with the arrival times, then $A_n^m = 1$ and $E\left[D_n^m\right]$ is of order $n \left(\sigma_n^m - \sigma_n^{m - 1}\right)$. Therefore,
\begin{equation*}
E\left[\Sigma_n(T)\right] \geq E\left[\sum_{m = 1}^{\calR_n(T) + 1} \frac{\left(d_n^- + 1\right)D_n^m}{n^2}\right]
\end{equation*}
is lower bounded by a quantity of order $\left(d_n^- + 1\right) T / n$, which does not approach zero as $n \to \infty$ if $d_n^- / n \nrightarrow 0$. However, Theorem \ref{the: fluid limit} covers sequences of resampling processes such that $\calR_n$ separates arrivals fully for infinitely many $n$. For this reason we require that \eqref{eq: event separation property} holds only for the subsequence of processes that do not separate arrivals fully.

The next lemma gathers some useful properties of the random variables $A_n^m$ and $D_n^m$, and will be used to prove Proposition \ref{prop: admissible resampling processes}; we prove the lemma in Appendix \ref{app: auxiliary results}.

\begin{lemma}
	\label{lem: moments of number of arrivals and departures between resampling times}
	Let $A_n^m$ denote the number of tasks that arrive in $\left(\sigma_n^{m - 1}, \sigma_n^m\right]$, let $D_n^m$ be the number of tasks that depart in $\left(\sigma_n^{m - 1}, \sigma_n^m\right]$ and let $\calH_n \defeq \sigma\sett{\calR_n(t)}{t \geq 0}$ be the $\sigma$-algebra generated by the resampling times. If the resampling process is independent of the arrival times of tasks or is independent of the departure times of tasks, then
	\begin{subequations}
		\label{ch5-aux: arrivals and departures}
		\begin{align}
			&\expect*{A_n^m|\calH_n} = \var*{A_n^m|\calH_n} = \lambda_n \left(\sigma_n^m - \sigma_n^{m - 1}\right), \label{ch5-sau: arrivals} \\
			&\expect*{A_n^mD_n^m|\calH_n} \leq \expect*{A_n^m|\calH_n} n \left(\sigma_n^m - \sigma_n^{m - 1}\right), \label{ch5-sau: departures}
		\end{align}
	\end{subequations}
	respectively. If condition (c) of Proposition \ref{prop: admissible resampling processes} holds, then
	\begin{equation}
	\label{eq: bounds for condition c}
	\lim_{n \to \infty} \mu_nE\left[\sum_{m = 1}^{\calR_n(t) + 1} \left(\sigma_n^m - \sigma_n^{m - 1}\right)^2\right] = E\left[\left(\sigma_1^1\right)^2\right] t \quad \text{for all} \quad t \geq 0.
	\end{equation}
\end{lemma}

We now prove Proposition \ref{prop: admissible resampling processes}.

\begin{proof}[Proof of Proposition \ref{prop: admissible resampling processes}]
	In order to prove that $\set{\calR_n}{n \geq 1}$ pseudo-separates events, we must show that the limits in \eqref{eq: event separation property} hold when we only consider the indexes $n$ such that the resampling process does not separate arrivals fully. Hence, we may assume without loss of generality that the resampling process $\calR_n$ does not separate arrivals fully for any $n$; i.e., if~(a) holds, then we assume that $\kappa_n \geq 1$ for all $n$.
	
	Let us fix an arbitrary $T \geq 0$. First we establish that $\Delta_n(T) \Rightarrow 0$ as $n \to \infty$ when any of the conditions stated in the proposition holds. The latter limit clearly holds when (b) holds, which implies that $\Delta_n(T) \leq 1 / \mu_n$. If condition (a) holds instead, then
	\begin{equation*}
	\kappa_n + 1 \geq \left|\calN_n^a\left(t\right) - \calN_n^a\left(\sigma_n^m\right)\right| \geq \lambda_n \left|t - \sigma_n^m\right| - 2 \sup_{u \in [0, T]} \left|\calN_n^a(u) - \lambda_n u\right|
	\end{equation*}
	for all $m \leq \calR_n(T)$ and $\sigma_n^m \leq t \leq \min \left\{\sigma_n^{m + 1}, T\right\}$. It follows that
	\begin{equation}
	\label{eq: bound on delta n t}
	\Delta_n(T) \leq \frac{\kappa_n + 1}{\lambda_n} + \frac{2}{\lambda_n}\sup_{u \in [0, T]} \left|\calN_n^a(u) - \lambda_n u\right|.
	\end{equation}
	The right-hand side goes to zero in probability by \eqref{eq: arrival and resampling rates conditions} and the law of large numbers for the Poisson process, hence $\Delta_n(T) \Rightarrow 0$ as $n \to \infty$ also in this case. A similar argument applies when condition (c) holds. Indeed, note that
	\begin{equation*}
	1 \geq \left|\calR_n(t) - \calR_n\left(\sigma_n^m\right)\right| \geq \mu_n \left|t - \sigma_n^m\right| - 2 \sup_{u \in [0, T]} \left|\calR_n(u) - \mu_n u\right|
	\end{equation*}
	for all $m \leq \calR_n(T)$ and $\sigma_n^m \leq t \leq \min \left\{\sigma_n^{m + 1}, T\right\}$. Arguing as above, we conclude from the law of large numbers for the renewal process $\calR$ that $\Delta_n(T) \Rightarrow 0$ as $n \to \infty$; we refer to \cite[Theorem 5.10]{chen2013fundamentals} for a proof of the latter law of large numbers.
	
	We now prove that $E\left[\Sigma_n(T)\right] \to 0$ as $n \to \infty$. For this purpose we first note that
	\begin{equation*}
	E\left[\Sigma_n(T)\right] = \frac{1}{n^2}E\left[\sum_{m = 1}^{\calR_n(T) + 1} \expect*{\left(d_n^- + 1\right)\left(A_n^m + D_n^m - 1\right)A_n^m + \left(A_n^m\right)^2 | \calH_n}\right],
	\end{equation*}
	where $\calH_n \defeq \sigma\sett{\calR_n(t)}{t \geq 0}$ is the $\sigma$-algebra generated by the resampling times. Let
	\begin{equation*}
	Y_n^m \defeq \left(d_n^- + 1\right)\left(\expect*{A_n^m\left(A_n^m - 1\right) | \calH_n} + \expect*{A_n^mD_n^m | \calH_n}\right) + \expect*{\left(A_n^m\right)^2 | \calH_n}
	\end{equation*}
	denote term $m$ in the above summation. Then we may write
	\begin{equation*}
	E\left[\Sigma_n(T)\right] \leq \frac{1}{n^2}E\left[\sum_{m = 1}^{\calR_n(T)} Y_n^m\right] + \frac{1}{n^2}E\left[Y_n^{\calR_n(T) + 1}\right].
	\end{equation*}
	Next we prove that the first term on the right-hand side approaches zero as $n \to \infty$, and it is straightforward to check that the second term also vanishes; considering the sum of $Y_n^m$ over $m = 1, \dots, \calR_n(T)$ instead of $m = 1, \dots, \calR_n(T) + 1$ simplifies calculations.

	If (a) holds, then Lemma \ref{lem: moments of number of arrivals and departures between resampling times} yields
	\begin{equation*}
	Y_n^m \leq \left(d_n^- + 1\right)\left[\left(\kappa_n + 1\right)\kappa_n + (\kappa_n + 1)n\left(\sigma_n^m - \sigma_n^{m - 1}\right)\right] + \left(\kappa_n + 1\right)^2.
	\end{equation*}
	Moreover, $E\left[\calR_n(T)\right] \leq \lambda_n T / \left(\kappa_n + 1\right)$ and thus
	\begin{equation*}
	\frac{1}{n^2}E\left[\sum_{m = 1}^{\calR_n(T)} Y_n^m\right] \leq \frac{\left(d_n^- + 1\right)\left[\kappa_n \lambda_n T + \left(\kappa_n + 1\right) n T\right] + \left(\kappa_n + 1\right) \lambda_n T}{n^2}.
	\end{equation*}
	If $\kappa_n \geq 1$ for all $n$, then the right-hand side approaches zero as $n \to \infty$ by \eqref{eq: arrival and resampling rates conditions}.
	
	Suppose now that conditions (b) or (c) hold. Lemma \ref{lem: moments of number of arrivals and departures between resampling times} implies that
	\begin{equation*}
	Y_n^m \leq \left(d_n^- + 1\right)\left(\lambda_n^2 + \lambda_n n\right)\left(\sigma_n^m - \sigma_n^{m - 1}\right)^2 + \lambda_n \left(\sigma_n^m - \sigma_n^{m - 1}\right) + \lambda_n^2 \left(\sigma_n^m - \sigma_n^{m - 1}\right)^2.
	\end{equation*}
	If (b) holds, then $\sigma_n^m - \sigma_n^{m - 1} \leq 1 / \mu_n$ and $\left(\sigma_n^m - \sigma_n^{m - 1}\right)^2 \leq \left(\sigma_n^m - \sigma_n^{m - 1}\right) / \mu_n$. Therefore,
	\begin{equation*}
	\frac{1}{n^2}E\left[\sum_{m = 1}^{\calR_n(T)} Y_n^m\right] \leq \frac{\left(d_n^- + 1\right)\left(\lambda_n^2 + \lambda_n n\right)T + \lambda_n^2 T}{\mu_nn^2} + \frac{\lambda_n T}{n^2}.
	\end{equation*}
	It follows from \eqref{eq: arrival and resampling rates conditions} that the right-hand side vanishes as $n \to \infty$. Finally, if (c) holds, then
	\begin{equation*}
	\frac{1}{n^2}E\left[\sum_{m = 1}^{\calR_n(T)} Y_n^m\right] \leq \frac{\left(d_n^- + 1\right)\left(\lambda_n^2 + \lambda_n n\right) + \lambda_n^2}{n^2}E\left[\sum_{m = 1}^{\calR_n(t) + 1} \left(\sigma_n^m - \sigma_n^{m - 1}\right)^2\right] + \frac{\lambda_n T}{n^2},
	\end{equation*}
	and the right-hand side approaches zero as $n \to \infty$ by \eqref{eq: arrival and resampling rates conditions} and \eqref{eq: bounds for condition c}.
\end{proof}

The next corollary says that if $\set{\calR_n}{n \geq 1}$ pseudo-separates events, then $\Delta_n(T) \Rightarrow 0$ as $n \to \infty$ for all $T \geq 0$. In other words, this means that the limit holds without considering only the resampling processes that do not separate arrivals fully.

\begin{corollary}
	\label{cor: limit of delta n t}
	If $\set{\calR_n}{n \geq 1}$ pseudo-separates events, then
	\begin{equation*}
	\Delta_n(T) \Rightarrow 0 \quad \text{as} \quad n \to \infty \quad \text{for all} \quad T \geq 0.
	\end{equation*}
\end{corollary}

\begin{proof}
	Note that \eqref{eq: bound on delta n t} with $\kappa_n = 0$ holds when $\calR_n$ separates arrivals fully.
\end{proof}

\subsection{Vanishing processes}
\label{sub: vanishing processes}

Endow $\R^\N$ with the metric
\begin{equation*}
d(x, y) \defeq \sum_{i = 0}^\infty \frac{\min\left\{|x(i) - y(i)|, 1\right\}}{2^i} \quad \text{for all} \quad x, y \in \R^\N,
\end{equation*}
which is compatible with the product topology. Also, let $D_{\R^\N}[0, \infty)$ be the space of c\`adl\`ag functions from $[0, \infty)$ into $\R^\N$ with the topology of uniform convergence over compact sets. In this section we establish that $\bv_n \Rightarrow 0$ in $D_{\R^\N}[0, \infty)$ as $n \to \infty$.

For this purpose, let $D_\R[0, T]$ be the space of real c\`adl\`ag functions defined on $[0, T]$, which we endow with the uniform norm, defined by
\begin{equation*}
\norm{\bx}_T \defeq \sup_{t \in [0, T]} |\bx(t)| \quad \text{for all} \quad \bx \in D_\R[0, T].
\end{equation*}
The following lemma is proved in Appendix \ref{app: auxiliary results}.

\begin{lemma}
	\label{lem: convergence in probability from components}
	Suppose that $\set{\bx_n}{n \geq 1}$ are random variables with values in $D_{\R^\N}[0, \infty)$. The following properties are equivalent.
	\begin{enumerate}
		\item[(a)] $\bx_n \Rightarrow 0$ in $D_{\R^\N}[0, \infty)$ as $n \to \infty$.
		
		\item[(b)] $\bx_n(i) \Rightarrow 0$ in $D_\R[0, T]$ as $n \to \infty$ for all $i \geq 0$ and $T \geq 0$.
	\end{enumerate}
\end{lemma}

By Lemma \ref{lem: convergence in probability from components}, we can prove that $\bv_n \Rightarrow 0$ in $D_{\R^\N}[0, \infty)$ by showing that $\bv_n(i) \Rightarrow 0$ in $D_\R[0, T]$ for all $i \geq 0$ and $T \geq 0$. We prove this by showing that the first four terms and the difference between the last two terms on the right-hand side of \eqref{eq: vanishing process} converge to zero in probability. In the next two sections we show that $\bL_n(i) \Rightarrow 0$ and $\bM_n(i) \Rightarrow 0$ in $D_\R[0, T]$ for all $i \geq 0$ and $T \geq 0$. Then we invoke the law of large numbers for the Poisson process to prove that the difference between the last two terms of \eqref{eq: vanishing process} also converges to zero.

\subsubsection{Limit of the processes $\bL_n$}
\label{subsub: limit of the processes l}

For each $t \geq 0$, we define
\begin{equation*}
K_n(t) \defeq \set{u \in V_n}{\bX_n(t, v) = \bar{\bX}_n(t, v)\ \text{if}\ v = u\ \text{or}\ (u, v) \in \bE_n(t)}.
\end{equation*}
Note that all the servers in the neighborhood of a server $u \in K_n(t)$ have the same number of tasks as they had at the last resampling time. Hence,
\begin{equation*}
\left|I_n^m\left(X_n^m, i\right) - I_n^m(\bar{X}_n^m, i)\right| \leq \ind{u_n^m \notin K_n\left(\tau_n^{m-}\right)} \quad \text{for all} \quad i \geq 0 \quad \text{and} \quad m \geq 1,
\end{equation*}
where we recall that $u_n^m$ is the server where a task appears at time $\tau_n^m$.

\begin{remark}
	\label{rem: connection with bramson}
	If a task appears in the complement $K_n^c(t)$ of $K_n(t)$, then the dispatching decision is influenced by a server that experienced an arrival or departure between time $t$ and the preceding resampling time. The set $K_n^c(t)$ is reminiscent of the \emph{influence process} introduced in the proof of \cite[Proposition 7.1]{bramson2012asymptotic}; the setup considered there is a system of parallel single-server queues where the classical power-of-$d$ policy is used to balance the load. The influence process of a server $u$ describes the set of servers that influence the queue length of $u$ over $[0, t]$. This process is used in \cite{bramson2012asymptotic} to prove that a fixed and finite set of queue lengths observed at a fixed time $t$ become asymptotically independent and identically distributed as the number of servers approaches infinity, provided that all the queue lengths in the system are independent and identically distributed at time zero. The proof relies on approximating the number of servers in the influence process of a single server by a continuous-time branching process where each parent has $d$ children. However, the present paper uses the sets $K_n^c(t)$ to show that $\norm{\bL_n(i)}_T$ converges in probability to zero. For this purpose we provide a bound for the size of the set $K_n^c(t)$. The bound increases linearly with the number of arrivals since the preceding resampling time, as in a continuous-time branching process, but depends on the number of departures as well.
\end{remark}

Let $A_n^m$ denote the number of tasks that arrive in $\left(\sigma_n^{m - 1}, \sigma_n^m\right]$ and let $D_n^m$ denote the number of tasks that depart. If $\sigma_n^{m - 1} < t \leq \sigma_n^m$ and $k$ tasks arrive in $\left(\sigma_n^{m - 1}, t\right]$, then at time $t$ at most $k + D_n^m$ servers have a number of tasks that is different from the number of tasks that they had at time $\sigma_n^{m - 1}$. Since each of these servers can be in the neighborhood of at most $d_n^-$ servers, it follows that at most $\left(k + D_n^m\right)\left(d_n^- + 1\right)$ servers are not in $K_n(t)$. Thus, the random variables $A_n^m$ and $D_n^m$ can be used to upper bound $\norm{\bL_n(i)}_T$ for all $i \geq 0$ and $T \geq 0$. This observation is used in the following proposition.

\begin{proposition}
	\label{prop: limit of the processes l}
	We have
	\begin{equation*}
	\bL_n(i) \Rightarrow 0 \quad \text{in} \quad D_\R[0, T] \quad \text{as} \quad n \to \infty \quad \text{for all} \quad i \geq 0 \quad \text{and} \quad T \geq 0,
	\end{equation*}
	and in particular $\bL_n \Rightarrow 0$ in $D_{\R^\N}[0, \infty)$ as $n \to \infty$.
\end{proposition}

\begin{proof}
We must prove that
\begin{equation*}
\lim_{n \to \infty} P\left(\norm{\bL_n(i)}_T \geq \varepsilon\right) = 0 \quad \text{for all} \quad \varepsilon > 0, \quad i \geq 0 \quad \text{and} \quad T \geq 0.
\end{equation*}
For this purpose, let us fix $\varepsilon > 0$, $i \geq 0$ and $T \geq 0$, and note that
\begin{equation*}
\norm{\bL_n(i)}_T \leq \frac{1}{n} \sum_{m = 1}^{\calN_n^a(T)} \left|I_n^m\left(X_n^m, i\right) - I_n^m\left(\bar{X}_n^m, i\right)\right| \leq \frac{1}{n} \sum_{m = 1}^{\calN_n^a(T)} \ind{u_n^m \notin K_n\left(\tau_n^{m-}\right)}.
\end{equation*}
By Markov's inequality, we may focus on bounding the expectation of the right-hand side:
\begin{equation*}
P\left(\norm{\bL_n(i)}_T \geq \varepsilon\right) \leq P\left(\frac{1}{n} \sum_{m = 1}^{\calN_n^a(T)} \ind{u_n^m \notin K_n\left(\tau_n^{m-}\right)} \geq \varepsilon\right) \leq \frac{1}{n \varepsilon} E \left[\sum_{m = 1}^{\calN_n^a(T)} \ind{u_n^m \notin K_n\left(\tau_n^{m-}\right)}\right].
\end{equation*}

Let $A_n^l$ and $D_n^l$ be the number of arrivals and departures in $\left(\sigma_n^{l - 1}, \sigma_n^l\right]$, respectively, and suppose that  $\tau_n^m < \tau_n^{m + 1} < \dots < \tau_n^{m + A_n^l - 1}$ are all the arrival times in this interval. Then
\begin{equation*}
\left|K_n^c\left(\tau_n^{m + k -}\right)\right| \leq \left(k + D_n^l\right) \left(d_n^- + 1\right) \quad \text{for all} \quad 0 \leq k \leq A_n^l - 1,
\end{equation*}
where $K_n^c(t)$ denotes the complement of $K_n(t)$. Recall that this holds since the number of tasks may have changed in at most $k + D_n^l$ servers between $\sigma_n^{l - 1}$ and right before $\tau_n^{m + k}$, and each server can be in the neighborhood of at most $d_n^-$ other servers.

Let $\calG_n \defeq \sigma \sett{\calN_n^a(t), \calR_n(t)}{t \geq 0}$ denote the $\sigma$-algebra generated by the arrival and resampling times. Since $u_n^m$ is uniformly distributed in $V_n$, the above observation about the sets $K_n^c\left(\tau_n^{m-}\right)$ implies that
\begin{align*}
\frac{1}{n}E \left[\sum_{m = 1}^{\calN_n^a(T)} \ind{u_n^m \notin K_n\left(\tau_n^{m-}\right)}\right] &= \frac{1}{n}E \left[\sum_{m = 1}^{\calN_n^a(T)} \expect*{\ind{u_n^m \notin K_n\left(\tau_n^{m-}\right)}|\calG_n}\right] \\
&\leq \frac{1}{n}E\left[\sum_{l = 1}^{\calR_n(T) + 1} \sum_{k = 0}^{A_n^l - 1} \frac{\left(k + D_n^l\right)\left(d_n^- + 1\right)}{n}\right] \\
&\leq \frac{1}{n^2}E\left[\sum_{l = 1}^{\calR_n(T) + 1} \left(d_n^- + 1\right)\left[A_n^l\left(A_n^l - 1\right) + A_n^lD_n^l\right]\right].
\end{align*}

The right-hand side is upper bounded by $E\left[\Sigma_n(T)\right]$. As a result, if there are infinitely many indexes $n$ such that $\calR_n$ does not separate arrivals fully, then the right-hand side of the above equation converges to zero as $n \to \infty$ by \eqref{eq: event separation property}. Moreover, $\norm{\bL_n(i)}_T = 0$ if $\calR_n$ separates arrivals fully by \eqref{eq: decomposition for graph resampled between every two arrivals}. Therefore,
\begin{equation*}
\lim_{n \to \infty} P\left(\norm{\bL_n(i)}_T \geq \varepsilon\right) = 0,
\end{equation*}
and this completes the proof.
\end{proof}

\subsubsection{Limit of the processes $\bM_n$}
\label{subsub: limit of the processes m}

Let $\calF_{n, t} \defeq \sigma \sett{\calR_n(s), \bG_n(s), \bX_n(s)}{0 \leq s \leq t}$ denote the $\sigma$-algebra generated by the resampling times and the history of the system up to time $t$. The resampling times are stopping times with respect to this filtration because $\left\{\sigma_n^m \leq t\right\} = \left\{\calR_n(t) \geq m\right\}$ for all $m, t \geq 0$. Therefore, the $\sigma$-algebra $\calF_n^m \defeq \calF_{n, \sigma_n^m}$ is well-defined for all $m \geq 0$.


\begin{lemma}
	\label{lem: martingale}
	Let $M_n^m(i) \defeq \bM_n\left(\sigma_n^m, i\right)$ for $i \geq 0$ and $m \geq 0$. The process $\set{M_n^m(i)}{m \geq 0}$ is a discrete-time martingale with respect to the filtration $\set{\calF_n^m}{m \geq 0}$ for all $i \geq 0$.
\end{lemma}

\begin{proof}
Suppose first that $\bM_n$ is given by \eqref{eq: general decomposition}, and let $\calG_n^m \defeq \calF_n^m \vee \sigma \sett{\calN_n^a(t), \calR_n(t)}{t \geq 0}$ be the smallest $\sigma$-algebra that contains $\calF_n^m$ and the $\sigma$-algebra generated by all the arrival and resampling times. For each $m \geq 0$, we have
\begin{equation*}
\begin{split}
\expect*{M_n^{m + 1}(i) - M_n^m(i) | \calF_n^m} &= \expect*{\expect*{M_n^{m + 1}(i) - M_n^m(i) | \calG_n^m} | \calF_n^m} \\
&= \expect*{\frac{1}{n}\sum_{\sigma_n^m < \tau_n^l \leq \sigma_n^{m + 1}} \expect*{I_n^l\left(\bar{X}_n^l, i\right) - \beta_n\left(\bar{q}_n^l(i)\right) | \calG_n^m} | \calF_n^m}.
\end{split}
\end{equation*}
The random variables $\bar{X}_n^l$ are all equal and $\calF_n^m$-measurable, thus also $\calG_n^m$-measurable. But the graph $G_n^m$ used throughout $\left(\sigma_n^m, \sigma_n^{m + 1}\right]$ is independent of $\calG_n^m$. It follows from \eqref{eq: expectation of assigning indicators} that each term in the above summation is zero, thus the right-hand side of the equation is zero and this proves that $\set{M_n^m(i)}{m \geq 0}$ is a martingale.

Suppose now that the resampling process separates arrivals fully and \eqref{eq: decomposition for graph resampled between every two arrivals} applies, then
\begin{equation*}
\expect*{M_n^{m + 1}(i) - M_n^m(i) | \calF_n^m} = \expect*{\frac{1}{n}\sum_{\sigma_n^m < \tau_n^l \leq \sigma_n^{m + 1}} \expect*{I_n^l\left(X_n^l, i\right) - \beta_n\left(q_n^l(i)\right) | \calG_n^m} | \calF_n^m}.
\end{equation*}
Since $\calR_n$ separates arrivals fully, the sum has zero terms or one term. In the latter case:
\begin{equation*}
\expect*{I_n^l\left(X_n^l, i\right) - \beta_n\left(q_n^l(i)\right) | \calG_n^m} = \expect*{\expect*{I_n^l\left(X_n^l, i\right) - \beta_n\left(q_n^l(i)\right) | X_n^l, \calG_n^m} | \calG_n^m} = 0,
\end{equation*}
because the graph $G_n^m$ used in $\left(\sigma_n^m, \sigma_n^{m + 1}\right]$ is independent of the $\sigma$-algebra $\calG_n^m \vee \sigma \left(X_n^l\right)$ generated by $\calG_n^m$ and the state $X_n^l$ of the system prior to the first arrival following $\sigma_n^m$.
\end{proof}

\begin{remark}
	\label{rem: martingale property proof}
	The argument at the end of the proof of Lemma \ref{lem: martingale} only works because $\tau_n^l$ is the time of the first arrival after $\sigma_n^m$. Suppose that several tasks arrive in $\left(\sigma_n^m, \sigma_n^{m + 1}\right]$ and let $\tau_n^l < \tau_n^{l + 1} < \dots < \tau_n^{l + k}$ denote the arrival times. If $0 < j \leq k$, then the difference between the random variables $\bar{X}_n^{l + j}$ and $X_n^{l + j}$ depends on how the graph $G_n^m$ was used to dispatch the first $j$ tasks. Since these random variables are measurable with respect to $\calG_n^m \vee \sigma \left(X_n^{l + j}\right)$, it follows that $G_n^m$ and $\calG_n^m \vee \sigma \left(X_n^{l + j}\right)$ are not independent; knowing how $\bX_n$ changed over a certain number of arrivals provides information about the graph.
\end{remark}

The next lemma implies that we can use the discrete-time martingale $\set{M_n^m(i)}{m \geq 0}$ to prove that the continuous-time process $\bM_n(i)$ converges weakly to zero.

\begin{lemma}
	\label{lem: reduction to discrete time}
	For each $i \geq 0$ and $T \geq 0$, we have
	\begin{equation}
	\label{eq: bound for supremum over time of m}
	\norm{\bM_n(i)}_T \leq \max_{m \leq \calR_n(T)} \left|M_n^m(i)\right| + \frac{\lambda_n \Delta_n(T)}{n} + \frac{2}{n}\sup_{t \in [0, T]} \left|\calN_n^a(t) - \lambda_n t\right|,
	\end{equation}
	where $\Delta_n(T)$ is as in Definition \ref{def: separate events sufficiently}. Furthermore, the last two terms on the right-hand side converge in probability to zero as $n \to \infty$.
\end{lemma}

\begin{proof}
For each $i \geq 0$ and $T \geq 0$, we have
\begin{equation*}
\norm{\bM_n(i)}_T \leq \max_{m \leq \calR_n(T)} \left|M_n^m(i)\right| + \sup_{s, t \in [0, T]} \set{\left|\bM_n(t, i) - \bM_n(s, i)\right|}{|t - s| \leq \Delta_n(T)}.
\end{equation*}
We conclude that \eqref{eq: bound for supremum over time of m} holds by noting that if $s, t \in [0, T]$ and $|t - s| \leq \Delta_n(T)$, then
\begin{align*}
\left|\bM_n(t, i) - \bM_n(s, i)\right| \leq \frac{1}{n}\left|\calN_n^a(t) - \calN_n^a(s)\right| &\leq \frac{\lambda_n\left|t - s\right|}{n} + \frac{2}{n}\sup_{u \in [0, T]} \left|\calN_n^a(u) - \lambda_n u\right| \\
&\leq \frac{\lambda_n\Delta_n(T)}{n} + \frac{2}{n}\sup_{u \in [0, T]} \left|\calN_n^a(u) - \lambda_n u\right|.
\end{align*}
The second term on the right-hand side of \eqref{eq: bound for supremum over time of m} converges to zero in probability as $n \to \infty$ by Corollary \ref{cor: limit of delta n t}. Moreover, the third term on the right-hand side of \eqref{eq: bound for supremum over time of m} also converges to zero in probability by the law of large numbers for the Poisson process.
\end{proof}

By the above lemma, we can prove that $\bM_n(i) \Rightarrow 0$ in $D_\R[0, T]$ by showing that the first term of \eqref{eq: bound for supremum over time of m} converges in probability to zero. This is done in the following proposition. First we note that Lemma \ref{lem: martingale} and Doob's maximal inequality imply that it is enough to establish that the second moment of $M_n^{\calR_n(T)}(i)$ vanishes as $n \to \infty$. Then we prove this by noting that the summands in the definition of $\bM_n(i)$ are conditionally independent if they correspond to arrival times that are separated by a resampling time.

\begin{proposition}
	\label{prop: limit of the processes m}
	We have
	\begin{equation*}
	\bM_n(i) \Rightarrow 0 \quad \text{in} \quad D_\R[0, T] \quad \text{as} \quad n \to \infty \quad \text{for all} \quad i \geq 0 \quad \text{and} \quad T \geq 0,
	\end{equation*}
	and in particular $\bM_n \Rightarrow 0$ in $D_{\R^\N}[0, \infty)$ as $n \to \infty$.
\end{proposition}

\begin{proof}
Fix $i \geq 0$ and $T \geq 0$. By Lemma \ref{lem: reduction to discrete time}, it suffices to prove that
\begin{equation*}
\lim_{n \to \infty} P\left(\max_{m \leq \calR_n(T)} \left|M_n^m(i)\right| \geq \varepsilon\right) = 0 \quad \text{for all} \quad \varepsilon > 0.	
\end{equation*}
Using the same arguments as in the proof of Lemma \ref{lem: martingale}, we may establish that
\begin{equation*}
\expect*{M_n^{m + 1}(i) - M_n^m(i) | \calR_n(T), \calF_n^m} = 0 \quad \text{for all} \quad m \geq 0.
\end{equation*}
This means that $\set{M_n^m(i)}{m \geq 0}$ is a martingale also when $\calR_n(T)$ is given. If we fix some arbitrary $\varepsilon > 0$, then it follows from Doob's maximal inequality that
\begin{align*}
P\left(\max_{m \leq \calR_n(T)} \left|M_n^m(i)\right| \geq \varepsilon\right) &= E\left[\cprob*{\max_{m \leq \calR_n(T)} \left|M_n^m(i)\right| \geq \varepsilon | \calR_n(T)}\right] \\
&\leq E\left[\frac{\expect*{\left|M_n^{\calR_n(T)}(i)\right|^2 | \calR_n(T)}}{\varepsilon^2}\right] = \frac{E\left|M_n^{\calR_n(T)}(i)\right|^2}{\varepsilon^2}.
\end{align*}

In order to prove the proposition, it is enough to show that the right-hand side of the above equation goes to zero as $n \to \infty$. Suppose first that the resampling process $\calR_n$ does not separate arrivals fully and thus $\bM_n$ is given by \eqref{eq: general decomposition}. Also, let
\begin{equation*}
Y_n^m \defeq I_n^m\left(\bar{X}_n^m, i\right) - \beta_n\left(\bar{q}_n^m(i)\right) \quad \text{and} \quad \calG_n \defeq \sigma \sett{\calN_n^a(t), \calR_n(t)}{t \geq 0}.
\end{equation*}
In addition, define $m_n(T) \defeq \max \set{m \geq 1}{\tau_n^m \leq \sigma_n^{\calR_n(T)}}$ and note that
\begin{equation*}
E\left|M_n^{\calR_n(T)}(i)\right|^2 = \frac{1}{n^2}E \left[\sum_{l, m = 1}^{m_n(T)} Y_n^l Y_n^m\right] = \frac{1}{n^2}E \left[\sum_{l, m = 1}^{m_n(T)} \expect*{Y_n^l Y_n^m | \calG_n}\right].
\end{equation*}

If $\tau_n^l \leq \sigma_n^k < \tau_n^m$ for some $k \geq 1$, then
\begin{equation}
\label{aux: zero covariance}
\begin{split}
\expect*{Y_n^l Y_n^m | \calG_n} &= \expect*{\expect*{Y_n^l Y_n^m | \bar{X}_n^m, \calG_n} | \calG_n} \\
&= \expect*{\expect*{Y_n^l | \bar{X}_n^m, \calG_n} \expect*{Y_n^m | \bar{X}_n^m, \calG_n} | \calG_n} \\
&= \expect*{\expect*{Y_n^l | \bar{X}_n^m, \calG_n} \expect*{Y_n^m | \bar{X}_n^m} | \calG_n} = 0.
\end{split}
\end{equation}
For the second equality, note that $Y_n^m$ can be expressed as a function of $\bar{X}_n^m$, the graph at $\tau_n^m$ and a random variable $u_n^m$ that selects a server uniformly at random and independently of everything else. The graph at $\tau_n^m$ is independent of $Y_n^l$, and also of $\calG_n$, since $\tau_n^l \leq \sigma_n^k < \tau_n^m$, which yields the second equality; the graph is resampled right after $\sigma_n^k$, thus the graph used to dispatch the task that arrives at time $\tau_n^m$ is different from the one used at time $\tau_n^l$, and independent of $Y_n^l$. The fourth identity holds because $\expect*{Y_n^m | \bar{X}_n^m} = 0$ by \eqref{eq: expectation of assigning indicators}.

For the second equality observe that $Y_n^m$ is a function of $\bar{X}_n^m$ and the graph at $\tau_n^m$. This graph is independent of $Y_n^l$, and also of $\calG_n$, since $\tau_n^l \leq \sigma_n^k < \tau_n^m$ , which yields the second equality; the graph is resampled right after $\sigma_n^k$, thus the graph used to dispatch the task that arrives at time $\tau_n^m$ is different from the one used at time $\tau_n^l$, and independent of $Y_n^l$. The fourth identity holds because $\expect*{Y_n^m | \bar{X}_n^m} = 0$ by \eqref{eq: expectation of assigning indicators}.

Consider the sets
\begin{equation*}
L_n^m \defeq \set{l \geq 1}{\sigma_n^{m - 1} < \tau_n^l \leq \sigma_n^m},
\end{equation*}
and let $A_n^m = \left|L_n^m\right|$ denote the total number of arrivals in the interval $\left(\sigma_n^{m - 1}, \sigma_n^m\right]$. Since $\left|Y_n^m\right| \leq 1$ for all $m \geq 1$, it follows from \eqref{aux: zero covariance} that
\begin{equation*}
E\left|M_n^{\calR_n(T)}(i)\right|^2 = \frac{1}{n^2} E\left[\sum_{m = 1}^{\calR_n(T)} \sum_{k, l \in L_n^m} \expect*{Y_n^k Y_n^l | \calG_n}\right] \leq \frac{1}{n^2} E \left[\sum_{m = 1}^{\calR_n(T)} \left(A_n^m\right)^2\right] \leq E\left[\Sigma_n(T)\right].
\end{equation*}
If there exist infinitely many indexes $n$ such that $\calR_n$ does not separate arrivals fully, then \eqref{eq: event separation property} implies that the right-hand side vanishes as $n \to \infty$ within this set of indexes.

Finally, suppose that $\calR_n$ separates arrivals fully. In this case $\bM_n$ is defined by \eqref{eq: decomposition for graph resampled between every two arrivals}, so we must set $Y_n^m \defeq I_n^m\left(X_n^m, i\right) - \beta_n\left(q_n^m(i)\right)$. Because the graph is resampled between any two consecutive arrivals, \eqref{aux: zero covariance} holds for all $l \neq m$. Hence,
\begin{equation*}
E\left|M_n^{\calR_n(T)}(i)\right|^2 \leq \frac{1}{n^2}E\left[\sum_{m = 1}^{m_n(T)} \left(Y_n^m\right)^2\right] \leq\frac{E \left[m_n(T)\right]}{n^2} \leq \frac{E\left[\calN_n^a(T)\right]}{n^2} = \frac{\lambda_n T}{n^2}.
\end{equation*}
Since the right-hand side goes to zero as $n \to \infty$, this completes the proof.
\end{proof}

The following result is a corollary of Propositions \ref{prop: limit of the processes l} and \ref{prop: limit of the processes m}.

\begin{corollary}
	\label{cor: vanishing processes converge to zero.}
	We have $\bv_n \Rightarrow 0$ in $D_{\R^\N}[0, \infty)$ as $n \to \infty$.
\end{corollary}

\begin{proof}
Fix $i \geq 1$ and $T \geq 0$. By the law of large numbers for the Poisson process,
\begin{equation*}
t \mapsto \int_0^t \left[\bq_n(s, i) - \bq_n(s, i + 1)\right]ds - \frac{1}{n}\calN_i^d\left(n\int_0^t\left[\bq_n(s, i) - \bq_n(s, i + 1)\right]ds\right)
\end{equation*}
converges in probability to zero in $D_\R[0, T]$. It follows from Propositions \ref{prop: limit of the processes l} and \ref{prop: limit of the processes m} that $\bv_n(i) \Rightarrow 0$ in $D_\R[0, T]$. By Lemma \ref{lem: convergence in probability from components}, this completes the proof.
\end{proof} 

\subsection{Drift processes}
\label{sub: drift processes}

The following proposition is proved in Appendix \ref{app: tightness of occupancy processes}.

\begin{proposition}
	\label{prop: tightness}
	If $\set{\bq_n(0)}{n \geq 1}$ is tight in $\ell_1$, then each subsequence of $\set{\bq_n}{n \geq 1}$ has a further subsequence that converges weakly in $D_{\ell_1}[0, \infty)$. Furthermore, the weak limit of every convergent subsequence is a process that is almost surely continuous.
\end{proposition}

By assumption, $\set{\bq_n(0)}{n \geq 1}$ is tight in $\ell_1$, so every increasing sequence of natural numbers has a subsequence $\calK$ such that $\set{\bq_k}{k \in \calK}$ converges weakly in $D_{\ell_1}[0, \infty)$ to a process $\bq$ that is almost surely continuous. Let us fix the subsequence $\calK$ and the limit $\bq$. It remains to prove that $\bq$ satisfies \eqref{eq: fluid dynamics} with probability one.

\subsubsection{Characterization of a subsequential limit}
\label{subsub: characterization of a subsequential limit}

Let $S_{\ell_1}[0, \infty)$ and $S_{\R^\N}[0, \infty)$ denote the spaces $D_{\ell_1}[0, \infty)$ and $D_{\R^\N}[0, \infty)$, respectively, when they are equipped with the Skorohod $J_1$-topology instead of the uniform topology. By Corollary \ref{cor: vanishing processes converge to zero.} and \cite[Theorem 23.9]{kallenberg2021foundations},
\begin{equation}
\label{eq: weak limits in skorohod topology 1}
\bq_k \Rightarrow \bq \quad \text{in} \quad S_{\ell_1}[0, \infty) \quad \text{and} \quad \bv_k \Rightarrow 0 \quad \text{in} \quad S_{\R^\N}[0, \infty) \quad \text{as} \quad k \to \infty.
\end{equation}	
Indeed, the limits hold with respect to the uniform topology and the limiting processes are almost surely continuous. In addition, the law of large numbers for the Poisson process and Corollary \ref{cor: limit of delta n t} imply that the stochastic processes
\begin{equation}
\label{eq: weak limits in skorohod topology 2}
t \mapsto \frac{\calN_k^a(t)}{\lambda_k} - t \quad \text{and} \quad t \mapsto \Delta_k(t)
\end{equation}
converge weakly to zero as $k \to \infty$ in the uniform topology, and thus also in the Skorohod $J_1$-topology. The next lemma will be combined with Skorohod's representation theorem to construct $\bq$ and the processes $\left(\calN_k^a, \calR_k, \bq_k, \bv_k\right)$ on a common probability space where the above limits hold almost surely, which considerably simplifies the characterization of the subsequential limit $\bq$. The proof of the lemma is provided in Appendix \ref{app: auxiliary results}.

\begin{remark}
	\label{rem: separability}
	Suppose that $X_1$ and $X_2$ are random variables with values in separable metric spaces $S_1$ and $S_2$, respectively. Separability ensures that $(X_1, X_2)$ is a measurable function with values in the product space $S_1 \times S_2$, endowed with the product topology and the Borel $\sigma$-algebra; we refer to \cite[Appendix M10]{billingsley1999convergence}. This property is implicitly used in the statement of the following lemma. For this reason, we briefly switch from the uniform topologies to Skorohod $J_1$-topologies, which are separable. By \cite[Theorem 23.9]{kallenberg2021foundations}, limits with respect to these two topologies are equivalent if the limiting process is almost surely continuous.
\end{remark}

\begin{lemma}
	\label{lem: joint convergence in distribution}
	Consider separable metric spaces $S_1, \dots, S_m$ and define $\Pi \defeq S_1 \times \dots \times S_m$ with the product topology. Let $X_1$ be a random variable with values in $S_1$ and suppose that $x_j \in S_j$ is a constant for each $j = 2, \dots, m$. In addition, consider random variables $X_k^j$ with values in $S_j$ for each $j = 1, \dots, m$ and each $k \in \calK$. If
	\begin{equation*}
	X_k^1 \Rightarrow X_1 \quad \text{in} \quad S_1 \quad \text{and} \quad X_k^j \Rightarrow x_j \quad \text{in} \quad S_j \quad \text{for all} \quad j = 2, \dots, m \quad \text{as} \quad k \to \infty,
	\end{equation*}
	then $\left(X_k^1, X_k^2, \dots, X_k^m\right) \Rightarrow \left(X_1, x_2, \dots, x_m\right)$ in $\Pi$ as $k \to \infty$.
\end{lemma}

If Assumption \ref{ass: fluid limit conditions} holds, then Lemma \ref{lem: joint convergence in distribution} implies that the process
\begin{equation}
\label{eq: four tuple}
t \mapsto \left(\frac{\calN_k^a(t)}{\lambda_k} - t, \Delta_k(t), \bq_k(t), \bv_k(t)\right)
\end{equation}
converges weakly to $(0, 0, \bq, 0)$ in the product topology as $k \to \infty$. Hence, it follows from Skorohod's representation theorem that the processes $\set{\left(\calN_k^a, \calR_k, \bq_k, \bv_k\right)}{k \in \calK}$ and $\bq$ can be defined on a common probability space $(\Omega, \calF, \prob)$ where the limit holds with probability one and not just in distribution. In addition, \cite[Theorem 23.9]{kallenberg2021foundations} implies that Skorohod's $J_1$-topology can be replaced by the uniform topology in the limits, because the limiting processes are almost surely continuous. Namely,
\begin{subequations}
\label{eq: strong limits}
\begin{align}
&\lim_{k \to \infty} \bq_k = \bq \quad \text{in} \quad D_{\ell_1}[0, \infty), \label{seq1: strong limit of q} \\
&\lim_{k \to \infty} \bv_k = 0 \quad \text{in} \quad D_{\R^\N}[0, \infty), \label{seq2: strong limit of v}\\
&\lim_{k \to \infty} \sup_{t \in [0, T]} \left|\frac{\calN_k^a(t)}{\lambda_k} - t\right| = 0 \quad \text{for all} \quad T \geq 0, \label{seq3: strong limit of n}\\
&\lim_{k \to \infty} \Delta_k(T) = 0 \quad \text{for all} \quad T \geq 0, \label{seq4: strong limit of r}
\end{align}
\end{subequations}
almost surely. Moreover, \eqref{eq: decomposition of stochastic equations} and \eqref{eq: drift process} imply that
\begin{equation}
\label{eq: functional equation}
\begin{split}
\bq_k(t, i) &= \bq_k(0, i) + \bu_k(t, i - 1) - \bu_k(t, i) \\
&- \int_0^t \left[\bq_k(s, i) - \bq_k(s, i + 1)\right]ds + \bv_k(t, i) \quad \text{for all} \quad i \geq 1 \quad \text{and} \quad t \geq 0
\end{split}
\end{equation}
almost surely. Recall that $\bu_k(t, i)$ is defined by \eqref{eq: decomposition for graph resampled between every two arrivals} when $\calR_k$ separates arrivals fully, and $\bu_k(t, i)$ is defined by \eqref{eq: general decomposition} otherwise.

\begin{remark}
	\label{rem: measurability}
	Suppose that $\set{X_n}{n \geq 1}$ and $X$ are random variables with values in a common separable metric space, such that $X_n \Rightarrow X$ as $n \to \infty$. Skorohod's representation theorem states that versions of all these random variables (i.e., with the same laws) can be constructed on a common probability space so that the limit holds almost surely. The right-hand side of \eqref{eq: functional equation} is a measurable function of $\left(\calN_k^a, \calR_k, \bq_k, \bv_k\right)$; see Appendix \ref{app: auxiliary results} for more details. This implies that the probability that \eqref{eq: functional equation} holds only depends on the law of $\left(\calN_k^a, \calR_k, \bq_k, \bv_k\right)$, thus \eqref{eq: functional equation} holds with probability one in $(\Omega, \calF, \prob)$ by \eqref{eq: decomposition of stochastic equations} and \eqref{eq: drift process}.
\end{remark}

The following lemma says that the functions $\bu_k$ converge uniformly over compact sets.

\begin{lemma}
	\label{lem: limit of u}
	Fix $\omega \in \Omega$ in the set of probability one where \eqref{eq: strong limits} and \eqref{eq: functional equation} hold for all $k \in \calK$. There exists a function $\map{\bu(\omega)}{[0, \infty)}{\R^\N}$ such that
	\begin{equation}
	\label{eq: uniform convergence of u k}
	\lim_{k \to \infty} \sup_{t \in [0, T]} \left|\bu_k(\omega, t, i) - \bu(\omega, t, i)\right| = 0 \quad \text{for all} \quad i \geq 0 \quad \text{and} \quad T \geq 0.
	\end{equation}
	Moreover, $\bu(\omega, t, 0) = \lambda t$ for all $t \geq 0$.
\end{lemma}

\begin{proof}
	For brevity, let us omit $\omega$ from the notation. Since $\bu_k(t, 0) = \calN_k^a(t) / k$ for all $t \geq 0$, it follows from \eqref{seq3: strong limit of n} that the functions $\bu_k(0)$ converge uniformly over compact sets to the function $\bu(0)$ defined by $\bu(t, 0) \defeq \lambda t$ for all $t \geq 0$. Note that \eqref{seq1: strong limit of q} and \eqref{seq2: strong limit of v} imply that the functions $\bq_k(i)$ and $\bv_k(i)$ converge uniformly over compact sets for all $i \geq 0$. Hence, if \eqref{eq: uniform convergence of u k} holds for $i = j - 1$, then it must also hold for $i = j$ by \eqref{eq: functional equation}. We have already established that \eqref{eq: uniform convergence of u k} holds for $i = 0$, so we conclude that \eqref{eq: uniform convergence of u k} holds for all $i \geq 0$.
\end{proof}

The lemma implies that there exists a process $\bu$ on $(\Omega, \calF, \prob)$ such that \eqref{eq: uniform convergence of u k} holds and
\begin{equation}
\label{eq: functional equation in the limit}
\begin{split}
\bq(t, i) &= \bq(0, i) + \bu(t, i - 1) - \bu(t, i) \\
&- \int_0^t \left[\bq(s, i) - \bq(s, i + 1)\right]ds \quad \text{for all} \quad i \geq 1 \quad \text{and} \quad t \geq 0
\end{split}
\end{equation}
with probability one. The next lemma concerns the asymptotic behavior of the functions $\beta_n$ and will be used to characterize the process $\bu$; a proof is provided in Appendix \ref{app: auxiliary results}.

\begin{lemma}
	\label{lem: uniform convergence of sampling functions}
	The functions $\alpha_n$ satisfy that
	\begin{equation}
	\label{eq: limit of alpha n}
	\lim_{n \to \infty} \sup_{x \in [0, 1]} \left|\alpha_n(d + 1, x) - x^{d + 1}\right| = 0 \quad \text{for all} \quad d \geq 0.
	\end{equation}
	Also, the functions $\beta_n$ have the following limits:
	\begin{align}
	&\lim_{n \to \infty} \sup_{x \in [0, \theta]} \left|\beta_n(x) - x\varphi(x)\right| = 0 \quad \text{for all} \quad \theta \in [0, 1), \label{eq: limit of beta n} \\
	&\lim_{n \to \infty} \sup_{x \in [0, 1]} \left|\beta_n(x) - x\varphi(x)\right| = 0 \quad \text{if} \quad p(\infty) = 0. \label{eq: limit of beta n non-degenarate case}
	\end{align}
\end{lemma}

The following proposition characterizes the process $\bu$ in a set of probability one.

\begin{proposition}
	\label{prop: derivatives of u}
	Fix $\omega \in \Omega$ as in Lemma \ref{lem: limit of u} and such that $\bq(\omega)$ is continuous. The functions $\bu(\omega, i)$ are absolutely continuous for all $i \geq 0$. Furthermore, there exists a set $\calD(\omega) \subset (0, \infty)$ such that the complement of $\calD(\omega)$ in $(0, \infty)$ has zero Lebesgue measure and the functions $\bq(\omega, i)$ and $\bu(\omega, i)$ are differentiable for all $i \geq 0$ at every point in $\calD(\omega)$. In addition, the following properties hold.
	\begin{enumerate}
		\item[(a)] If $p(\infty) = 0$ and $t_0 \in \calD(\omega)$, then
		\begin{equation*}
			\dot{\bu}(\omega, t_0, i) = \lambda\bq(\omega, t_0, i) \varphi\left(\bq(\omega, t_0, i)\right) \quad \text{for all} \quad i \geq 1.
		\end{equation*}
		
		\item[(b)] If $p(\infty) > 0$ and $t_0 \in \calD(\omega)$, then
		\begin{equation*}
			\dot{\bu}(\omega, t_0, i) = \begin{cases}
				\lambda\bq(\omega, t_0, i) \varphi\left(\bq(\omega, t_0, i)\right) & \text{if} \quad \bq(\omega, t_0, i) < 1, \\
				\lambda - 1 + \bq(\omega, t_0, i + 1) & \text{if} \quad \bq(\omega, t_0, i) = 1,
			\end{cases}
			\quad \text{for all} \quad i \geq 1.
		\end{equation*}
		
		\item[(c)] $\bu(\omega, t, 0) = \lambda t$ for all $t \geq 0$. 
	\end{enumerate}
	Moreover, we may write
	\begin{equation}
		\label{eq: general expression for dot u}
		\begin{split}
			\dot{\bu}(\omega, t_0, i) &= \lambda \bq(\omega, t_0, i) \varphi\left(\bq(\omega, t_0, i)\right) + \left[\lambda p(\infty) - 1 + \bq(\omega, t_0, i + 1)\right]^+\ind{\bq(\omega, t_0, i) = 1}
		\end{split}
	\end{equation}
	for all $t_0 \in \calD(\omega)$, $i \geq 1$ and $p(\infty) \in [0, 1]$.
\end{proposition}

\begin{proof}
	For brevity, we omit $\omega$ from the notation. It follows from \eqref{eq: limiting degree distribution condition} that there exists $L \geq 0$ such that $\lambda_k \leq k L$ for all $k \in \calK$. This implies that if $s, t \in [0, T]$, then
	\begin{align*}
	\left|\bu_k(t, i) - \bu_k(s, i)\right| \leq \frac{1}{k}\left|\calN_k^a(t) - \calN_k^a(s)\right| &\leq \frac{\lambda_k}{k}\left|t - s\right| + \frac{2}{k} \sup_{u \in [0, T]} \left|\calN_k^a(u) - \lambda_k u\right| \\
	&\leq L\left|t - s\right| + \frac{2}{k} \sup_{u \in [0, T]} \left|\calN_k^a(u) - \lambda_k u\right|
	\end{align*}
	for all $i \geq 0$. If we take the limit as $k \to \infty$ on both sides of the latter inequality, then \eqref{seq3: strong limit of n} implies that $\bu(i)$ is Lipschitz of modulus $L$ for all $i$, and in particular absolutely continuous. It follows from \eqref{eq: functional equation in the limit} that $\bq(i)$ is Lipschitz of modulus $2L + 1$ for all $i$, and thus also absolutely continuous. The existence of $\calD$ is a straightforward consequence. 
	
	Note that property (c) was proved in Lemma \ref{lem: limit of u}, so it only remains to show that properties (a) and (b) hold. For this purpose we will assume that the processes $\bu_k$ are defined as in \eqref{eq: general decomposition}. The proof is similar when these processes are defined as in \eqref{eq: decomposition for graph resampled between every two arrivals}.

	Suppose that $p(\infty) = 0$ and fix $t_0 \in \calD$ and $i \geq 1$. By Abel's theorem, $\varphi$ is continuous on $[0, 1]$, and by Lemma \ref{lem: uniform convergence of sampling functions}, $\beta_k$ converges uniformly over $[0, 1]$ to the function $x \mapsto x\varphi(x)$. Given $\varepsilon > 0$, these observations imply that there exist $\delta_0 > 0$ and $k_0 \geq 1$ such that:
	\begin{equation}
	\label{aux: continuity of phi and uniform convergence of beta}
	\begin{split}
	&\left|x\varphi(x) - \bq(t_0, i) \varphi\left(\bq(t_0, i)\right)\right| \leq \frac{\varepsilon}{2} \quad \text{if} \quad x \in [0, 1] \quad \text{and} \quad \left|x - \bq(t_0, i)\right| \leq \delta_0, \\
	&\left|\beta_k(x) - x\varphi(x)\right| \leq \frac{\varepsilon}{2} \quad \text{if} \quad x \in [0, 1] \quad \text{and} \quad k \geq k_0.
	\end{split}
	\end{equation}
	By \eqref{seq1: strong limit of q}, the functions $\bq_k(i)$ converge uniformly over compact sets to the continuous function $\bq(i)$. Hence, there exist $\delta_1 > 0$ and $k_1 \geq k_0$ such that
	\begin{equation*}
	\left|\bq_k(t, i) - \bq(t_0, i)\right| \leq \delta_0 \quad \text{if} \quad \left|t - t_0\right| \leq 2\delta_1 \quad \text{and} \quad k \geq k_1.
	\end{equation*}
	Moreover, by \eqref{seq4: strong limit of r} there exists $k_2 \geq k_1$ such that $k \geq k_2$ implies that $t_0 - 2\delta_1 < \sigma_k^m < t_0 - \delta_1$ for some $m \geq 1$, and therefore
	\begin{equation}
	\label{aux: distance between resampling times}
	\left|\bar{\bq}_k(t, i) - \bq(t_0, i)\right| \leq \delta_0 \quad \text{if} \quad \left|t - t_0\right| \leq \delta_1 \quad \text{and} \quad k \geq k_2.
	\end{equation}
	Indeed, the resampling times $\sigma_k^m$ partition the interval $[0, t_0 - \delta_1]$ into subintervals of length upper bounded by $\Delta_k(t_0 - \delta_1)$, and the latter quantity approaches zero as $k \to \infty$.
	
	By \eqref{aux: continuity of phi and uniform convergence of beta} and \eqref{aux: distance between resampling times}, we have
	\begin{equation*}
	\left|\beta_k\left(\bar{q}_k^m\right) - \bq(t_0, i) \varphi\left(\bq(t_0, i)\right)\right| \leq \varepsilon \quad \text{for all} \quad \calN_k^a(t_0 - \delta_1) < m < \calN_k^a\left(t_0 + \delta_1\right) \quad \text{if} \quad k \geq k_2.
	\end{equation*}
	It follows that if $t_0 < t < t_0 + \delta_1$ and $k \geq k_2$, then
	\begin{align*}
	&\bu_k(t, i) - \bu_k(t_0, i) = \frac{1}{k}\sum_{m = \calN_k^a(t_0) + 1}^{\calN_k^a(t)} \beta_k\left(\bar{q}_k^m\right) \leq \frac{\calN_k^a(t) - \calN_k^a(t_0)}{k}\left[\bq(t_0, i) \varphi\left(\bq(t_0, i)\right) + \varepsilon\right], \\
	&\bu_k(t, i) - \bu_k(t_0, i) = \frac{1}{k}\sum_{m = \calN_k^a(t_0) + 1}^{\calN_k^a(t)} \beta_k\left(\bar{q}_k^m\right) \geq \frac{\calN_k^a(t) - \calN_k^a(t_0)}{k}\left[\bq(t_0, i) \varphi\left(\bq(t_0, i)\right) - \varepsilon\right].
	\end{align*}
	Therefore, \eqref{seq3: strong limit of n} implies that
	\begin{equation*}
	\lambda\left[\bq(t_0, i) \varphi\left(\bq(t_0, i)\right) - \varepsilon\right] \leq \lim_{t \to t_0^+} \lim_{k \to \infty} \frac{\bu_k(t, i) - \bu_k(t_0, i)}{t - t_0} \leq \lambda\left[\bq(t_0, i) \varphi\left(\bq(t_0, i)\right) + \varepsilon\right].
	\end{equation*}
	This proves (a) because $\varepsilon$ is arbitrary and the expression in the middle equals $\dot{\bu}(t_0, i)$.
	
	Assume now that $p(\infty) > 0$. Recall that the functions $\bq_k(i)$ converge uniformly over compact sets to the continuous function $\bq(i)$. Hence, $\bq(t_0, i) < 1$ implies that there exists $\theta \in [0, 1)$ such that $\bq_k(t, i) < \theta$ for all $t$ in a sufficiently small neighborhood of $t_0$ and all large enough $k \in \calK$. By Lemma \ref{lem: uniform convergence of sampling functions}, the functions $\beta_k$ converge to the function $x \mapsto x\varphi(x)$ uniformly over the interval $[0, \theta]$. Therefore, the expression in (b) for $\dot{\bu}(t_0, i)$ in the case where $\bq(t_0, i) < 1$ can be established using the same arguments as in the proof of (a).
	
	Suppose then that $\bq(t_0, i) = 1$. Then $\dot{\bq}(t_0, i) = 0$ since $\bq(i) \leq 1$. By \eqref{eq: functional equation in the limit},
	\begin{equation*}
	0 = \dot{\bu}(t_0, i - 1) - \dot{\bu}(t_0, i) - \left[\bq(t_0, i) - \bq(t_0, i + 1)\right].
	\end{equation*}
	In fact this holds for $1 \leq j \leq i$ because $\bq(t_0, i) = 1$ implies $\bq(t_0, j) = 1$ for all $j \leq i$. It follows from (c) that $\dot{\bu}(t_0, 0) = \lambda$, so we conclude that
	\begin{align*}
	\dot{\bu}(t_0, i) &= \dot{\bu}(t_0, i - 1) - \left[\bq(t_0, i) - \bq(t_0, i + 1)\right] \\
	&= \lambda - \sum_{j = 1}^i \left[\bq(t_0, j) - \bq(t_0, j + 1)\right] = \lambda - 1 + \bq(t_0, i + 1).
	\end{align*}
	This completes the proof of (b).
	
	Observe that \eqref{eq: general expression for dot u} holds if $p(\infty) = 0$ or $\bq(t_0, i) < 1$ by (a) and (b). Suppose then that $p(\infty) \in (0, 1]$ and $\bq(t_0, i) = 1$. Given $0 \leq \theta < 1$, there exists $\varepsilon > 0$ such that
	\begin{equation*}
		\bar{\bq}_k(t, i) \geq \theta \quad \text{for all} \quad t \in [t_0, t_0 + \varepsilon) \quad \text{and all large enough} \quad k \in \calK;
	\end{equation*}
	this follows as in the proof of \eqref{aux: distance between resampling times}. If $t$ and $k$ are as above, then
	\begin{equation*}
		\frac{\bu_k(t, i) - \bu_k(t_0, i)}{t - t_0} = \frac{1}{t - t_0} \sum_{m = \calN_k^a(t_0) + 1}^{\calN_k^a(t)} \frac{\beta_k\left(\bar{q}_k^m\right)}{k} \geq \frac{1}{t - t_0} \sum_{m = \calN_k^a(t_0) + 1}^{\calN_k^a(t)} \frac{\beta_k\left(\theta\right)}{k}
	\end{equation*}
	since $\beta_k$ is nondecreasing. By \eqref{seq3: strong limit of n} and Lemma \ref{lem: uniform convergence of sampling functions}, the right-hand side approaches $\lambda\theta \varphi(\theta)$ as $k \to \infty$. It follows from (b) and the above inequality that
	\begin{equation*}
		\lambda - 1 + \bq(t_0, i + 1) = \dot{\bu}(t_0, i) \geq \lambda \theta \varphi(\theta) \quad \text{for all} \quad \theta \in [0, 1).
	\end{equation*}
	This also holds with $\theta = 1$ since $\varphi$ is continuous in $[0, 1]$ by Abel's theorem. Thus,
	\begin{equation*}
		0 \leq \lambda\left[1 - \varphi(1)\right] - 1 + \bq(t_0, i + 1) = \lambda p(\infty) - 1 + \bq(t_0, i + 1),
	\end{equation*}
	and we conclude from (b) that \eqref{eq: general expression for dot u} holds.
\end{proof}

It follows from \eqref{eq: functional equation in the limit} and the above proposition that $\bq$ satisfies \eqref{eq: fluid dynamics} almost surely. We may now complete the proof of the fluid limit.

\begin{proof}[Proof of Theorem \ref{the: fluid limit}]
	By Proposition \ref{prop: tightness}, every subsequence of $\set{\bq_n}{n \geq 1}$ has a further subsequence that converges weakly in $D_{\ell_1}[0, \infty)$. By the earlier arguments in this section and Proposition \ref{prop: derivatives of u}, every convergent subsequence converges to a process $\bq$ such that
	\begin{equation*}
	\begin{split}
	\bq(t, i) &= \bq(0, i) + \lambda\int_0^t \left[a_{i - 1}\left(\bq(s)\right) - a_i\left(\bq(s)\right)\right]ds \\
	&- \int_0^t \left[\bq(s, i) - \bq(s, i + 1)\right]ds \quad \text{for all} \quad i \geq 1 \quad \text{and} \quad t \geq 0
	\end{split}
	\end{equation*}
	almost surely. Also, $\bq$ is almost surely continuous from $[0, \infty)$ into $\ell_1$ by Proposition \ref{prop: tightness}.
\end{proof}

\begin{appendices}
	
\section{Simulations}
\label{app: simulations}

\begin{figure}
	\centering
	\includegraphics{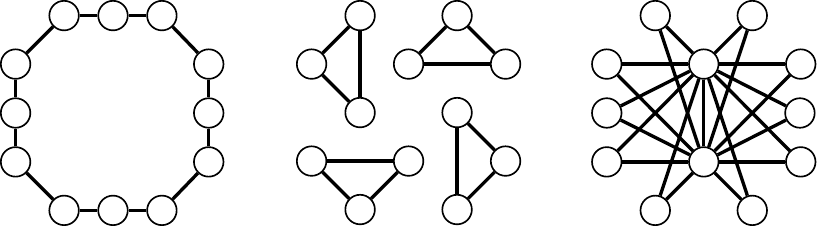}
	\caption{From left to right: the ring, the disjoint triangles and the double-star for $n = 12$.}
	\label{fig: graphs}
\end{figure}

Consider the three undirected graph topologies depicted in Figure \ref{fig: graphs}. All the nodes in the ring and the disjoint triangles have exactly two neighbors, and the degree distribution of the double-star is given by $p_n(2) = (n - 2) / n$ and $p_n(n - 1) = 2 / n$. Hence, the limiting degree distribution is the point mass at $d = 2$ in all three cases. Nonetheless, there are striking structural differences between these graphs.
\begin{enumerate}
	\item[(a)] The ring and the double-star are connected, whereas the other graph topology has multiple connected components.
	
	\item[(b)] The maximum degree of the double-star is $n - 1$, whereas the maximum degree of the ring and the disjoint triangles is $2$.
	
	\item[(c)] The diameter of the ring is $\floor{n / 2}$, whereas the diameter of the double-star is $2$.
\end{enumerate}

Below we report the results of various numerical experiments based on the three graph topologies of Figure \ref{fig: graphs}. First we evaluate the performance of the load balancing algorithm studied in this paper when the graph is static, and we compare this performance with the dynamic case. Then we show that \eqref{eq: fluid dynamics} accurately describes the behavior of the occupancy process in a large system, and we observe that Theorem \ref{the: fluid limit} does not seem to apply in a regime where the pseudo-separation property does not hold.

\subsection{Performance of static graphs}
\label{sub: performance of static graphs}

Figure \ref{fig: time averages} compares the performance of static graphs with that of dynamic graphs, for the topologies depicted in Figure \ref{fig: graphs}; in the dynamic case the resampling procedure is carried out by just reassigning the servers to the nodes of a fixed graph uniformly at random. By Theorem \ref{the: interchange of limits}, if the graph is resampled as a Poisson process, then the steady-state queue length distribution is asymptotically equivalent for the three topologies, and given by the sequence $q^*$ defined in Proposition \ref{prop: fixed point}. In contrast, Figure \ref{fig: time averages} shows that the steady-state queue length distribution depends on the topology of the graph in the static setting, and that the time average of $\bq_n(i)$ is larger than $q^*(i)$. This shows that performance improves when the graph is resampled over time for any of the topologies of Figure \ref{fig: graphs}.

\begin{figure}
	\centering
	\begin{subfigure}{0.49\columnwidth}
		\centering
		\includegraphics[width = \columnwidth]{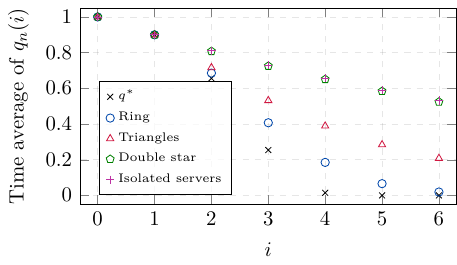}
		\caption{Time averages.}
		\label{fig: time averages}
	\end{subfigure}
	\hfill
	\begin{subfigure}{0.49\columnwidth}
		\centering
		\includegraphics[width = \columnwidth]{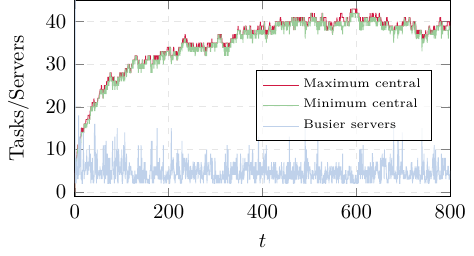}
		\caption{Double-star topology.}
		\label{fig: central servers}
	\end{subfigure}
	\caption{Load balancing on static graphs with the topologies depicted in Figure \ref{fig: graphs}. In all the cases the system starts empty, $n = 1500$ and $\lambda_n = 9 n / 10$. The plot on the left shows time averages computed over the second half of the simulation and the equilibrium point $q^*$ defined in Proposition \ref{prop: fixed point}. The plot on the right concerns the double-star topology. It shows the maximum and minimum numbers of tasks across the central servers and the number of servers that have more tasks than, or as many tasks as, the central server with the fewest tasks. The maximum and minimum numbers of tasks across the central servers remain very close together and thus are difficult to distinguish from each other.}
	\label{fig: static}
\end{figure}

Remarkably, the performance of the double-star is equivalent to that of $n$ independent single-server queues when the graph is static. This is explained by Figure \ref{fig: central servers}, which shows the number of tasks at the two servers placed in the center of the double-star, and the number of servers that have more tasks than, or as many tasks as, the central server with the fewest tasks. At time zero all the servers have the same number of tasks, but the percentage of servers with strictly less tasks than both of the central servers is approximately 99\% or larger throughout the rest of the simulation. When a task arrives to any of these servers, the server places the task in its own queue, as if it was isolated, because its only two neighbors are the central servers, which have longer queues.

The behavior of the double-star topology may be explained as follows. The arrival rate of tasks to the central server with the fewest tasks is at least $\lambda_n$ times the fraction of servers that have strictly more tasks, whereas tasks leave from this server at unit rate. As a result, the number of tasks at the central servers increases quickly and remains large throughout the simulation, while the fraction of servers with strictly more tasks than the central server with the fewest tasks remains small. This property is rigorously treated in \cite{goldsztajn2024server}, where it is further generalized to graphs with a more involved structure.

\subsection{Accuracy of the fluid approximation}
\label{sub: accuracy of the fluid approximation}

We now discuss simulations of the occupancy process for dynamic graphs. Figures~\ref{fig: loglog ring} and \ref{fig: loglog triangles} show sample paths of $\bq_n$ that remain close to the solution of \eqref{eq: fluid dynamics} for the ring topology and the disjoint triangles, both in the transient and stationary regimes and for a resampling rate as low as $\mu_n = \log \log n$. Figures \ref{fig: double star large n} and \ref{fig: double star very large n} show sample paths of $\bq_n$ when the graph topology remains double-star and the resampling rate is $\mu_n = \log n$.

\begin{figure}
	\centering
	\begin{subfigure}{0.49\columnwidth}
		\centering
		\includegraphics[width = \columnwidth]{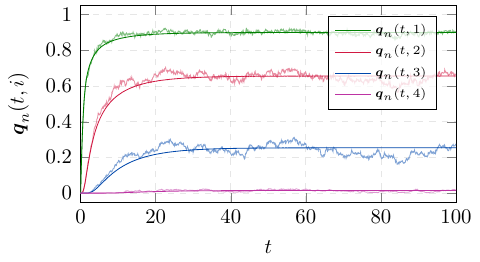}
		\caption{Ring with $n = 1500$ and $\mu_n = \log \log n$.}
		\label{fig: loglog ring}
	\end{subfigure}
	\hfill
	\begin{subfigure}{0.49\columnwidth}
		\centering
		\includegraphics[width = \columnwidth]{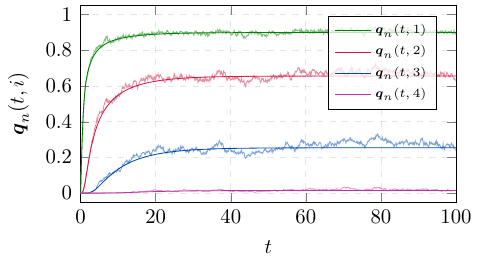}
		\caption{Triangles with $n = 1500$ and $\mu_n = \log \log n$.}
		\label{fig: loglog triangles}
	\end{subfigure}
	\vfill
	\vspace{3mm}
	\vfill
	\begin{subfigure}{0.49\columnwidth}
		\centering
		\includegraphics[width = \columnwidth]{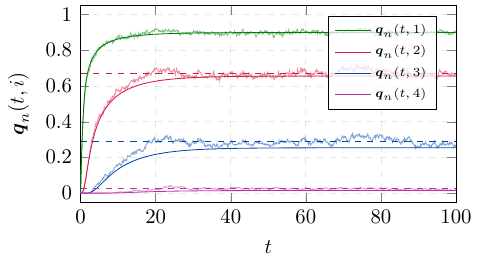}
		\caption{Double-star with $n = 1500$ and $\mu_n = \log n$.}
		\label{fig: double star large n}
	\end{subfigure}
	\hfill
	\begin{subfigure}{0.49\columnwidth}
		\centering
		\includegraphics[width = \columnwidth]{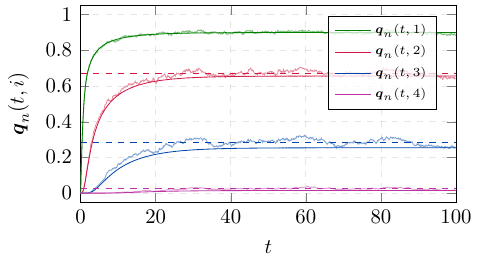}
		\caption{Double-star with $n = 5000$ and $\mu_n = \log n$.}
		\label{fig: double star very large n}
	\end{subfigure}
	\caption{Solution of \eqref{eq: fluid dynamics} and sample paths of $\bq_n$ for dynamic graphs. In all the cases the system starts empty, $\lambda_n = 9 n / 10$ and the resampling process is Poisson. The dashed lines depicted in the two plots on the bottom correspond to time averages computed over the interval $[40, 100]$.}
	\label{fig: loglog}
\end{figure}

Note that $d_n^- + 1 = n$ for the double-star, thus \eqref{eq: arrival and resampling rates conditions} does not hold when $\mu_n = \log n$, and in fact the approximation provided by \eqref{eq: fluid dynamics} does not seem accurate. Moreover, as the number of servers increases from $n = 1500$ to $n = 5000$, the accuracy of the approximation does not seem to improve since the sample path of $\bq_n$ does not get closer to the solution of~\eqref{eq: fluid dynamics}. This indicates that Theorem \ref{the: fluid limit} may not apply when $\mu_n = \log n$ and the graph always has a double-star topology, which suggests that some condition, besides \eqref{eq: limiting degree distribution condition}, on the random graph law used to sample the graph, is necessary for the fluid limit. While the pseudo-separation property and \eqref{eq: arrival and resampling rates conditions} are clearly not necessary conditions for the fluid limit to hold, the latter observations indicate that the dependence of the pseudo-separation property on the maximum indegrees $d_n^-$ is not just an artifact of our proof technique but possibly a manifestation of some fundamental condition required for the fluid limit to hold.

\section{Proofs of several results}
\label{app: auxiliary results}

\begin{proof}[Proof of Lemma \ref{lem: derivative of phi}.]
	It follows from the mean value theorem that for each $x \in (0, 1)$ and $d \geq 0$ there exists $\theta(x, d) \in (x, 1)$ such that
	\begin{equation*}
	\frac{1 - x^d}{1 - x} = d\left[\theta(x, d)\right]^{d - 1}.
	\end{equation*}
	We conclude that
	\begin{equation*}
	\left|\frac{1 - x^d}{1 - x} - d\right| = d\left|\left[\theta(x, d)\right]^{d - 1} - 1\right| \leq d \quad \text{for all} \quad x \in (0, 1) \quad \text{and} \quad d \geq 0.
	\end{equation*}
	
	Given $\varepsilon > 0$, there exists $k \geq 0$ such that
	\begin{equation*}
	\sum_{d = k + 1}^\infty d p(d) \leq \varepsilon.
	\end{equation*}
	Therefore, we have
	\begin{align*}
	\lim_{x \to 1^-} \left|\frac{\varphi(1) - \varphi(x)}{1 - x} - \sum_{d = 0}^\infty dp(d)\right| &\leq  \lim_{x \to 1^-} \sum_{d = 0}^\infty \left|\frac{1 - x^d}{1 - x} - d\right|p(d) \\
	&\leq \lim_{x \to 1^-} \sum_{d = 0}^k \left|\frac{1 - x^d}{1 - x} - d\right|p(d) + \sum_{d = k + 1}^\infty dp(d) \leq \varepsilon.
	\end{align*}
	Since $\varepsilon$ is arbitrary, this proves the first identity in the claim of the lemma. The second identity follows from Abel's theorem.
\end{proof}

\begin{proof}[Proof of Proposition \ref{prop: existence and uniqueness of solutions}.]
	For the existence part, it suffices to construct occupancy processes $\bq_n$ such that $\bq_n(0) \Rightarrow q$ in $\ell_1$. The initial states $\bq_n(0)$ can be obtained by letting $\bq_n(0, i)$ be deterministic, equal to the number in $\set{m / n}{m = 0, \dots, n}$ that is closest to $q(i)$. Processes $\bq_n$ with these initial states can be constructed as in Section \ref{sub: stochastic equations}. It follows from Theorem \ref{the: fluid limit} that a fluid trajectory $\bq$ with initial condition $q$ and continuous from $[0, \infty)$ into $\ell_1$ exists.
	
	Suppose that $\bx$ and $\by$ solve \eqref{eq: fluid differential equation}. By Lemma \ref{lem: derivative of phi}, the derivative $\varphi'$ is continuous in $[0, 1]$, thus bounded. As a result, the function $x \mapsto x\varphi(x)$ is Lipschitz in $[0, 1]$. It follows from \eqref{eq: fluid differential equation} that there exists a constant $L \geq 0$ such that
	\begin{equation*}
	|\bx(t, i) - \by(t, i)| \leq |\bx(0, i) - \by(0, i)| + L\int_0^t \sum_{j = i - 1}^{i + 1} |\bx(s, j) - \by(s, j)|ds
	\end{equation*}
	for all $i \geq 1$ and $t \geq 0$. We conclude that
	\begin{equation*}
	\norm{\bx(t) - \by(t)}_1 \leq \norm{\bx(0) - \by(0)}_1 + 3L\int_0^t \norm{\bx(s) - \by(s)}_1ds \quad \text{for all} \quad t \geq 0.
	\end{equation*}
	By Gronwall's inequality, $\norm{\bx(t) - \by(t)}_1 \leq \norm{\bx(0) - \by(0)}_1 \e^{3Lt}$ for all $t \geq 0$. This implies that solutions are continuous with respect to the initial condition. Moreover, by setting $\bx(0) = \by(0)$ we conclude that there exists a unique solution for each initial condition.
\end{proof}

\begin{proof}[Proof of Lemma \ref{lem: monotonicity}.]
	We say that $\map{\bx}{[0, \infty)}{\ell_1}$ is a truncated-$k$ fluid trajectory if it satisfies \eqref{eq: fluid differential equation} for all $1 \leq i \leq k$ and the boundary condition $\bx(t, i) = 0$ for $t \geq 0$ and $i > k$. Using similar arguments as in \cite[Secction 3]{vvedenskaya1996queueing}, we will first establish that the monotonicity property holds for truncated fluid trajectories and then use that fact to show that the monotonicity property holds for untruncated fluid trajectories as well.
	
	Note that truncated-$k$ fluid trajectories can be regarded as solutions to a finite system of differential equations. The right-hand side of this system is Lipschitz by Lemma \ref{lem: derivative of phi}, which yields existence and uniqueness of solutions, as well as continuity of solutions with respect to the initial condition. Also, if $\bx$ is a truncated-$k$ fluid trajectory with $\bx(0) \in Q$, then
	\begin{equation}
		\label{aux: properties of truncated fluid trajectories}
		\bx(t) \in Q \quad \text{and} \quad \bx(t, i) \leq \sum_{j = 0}^i \bx(0, j)\frac{[(\lambda + 1)t]^{i - j}}{(i - j)!} \quad \text{for all} \quad t \geq 0 \quad \text{and} \quad i \geq 1.
	\end{equation}
	The first property corresponds to \cite[Lemma 1]{vvedenskaya1996queueing} and is obtained by noting that $\bx(t) \in Q$ and $\bx(t, i - 1) = \bx(t, i)$ for some $1 \leq i \leq k$ imply that $\dot{\bx}(t, i - 1) \geq 0$ and $\dot{\bx}(t, i) \leq 0$. The second property is similar to \cite[Lemma 4]{vvedenskaya1996queueing} and is derived by induction: the property obviously holds for $i = 0$, and if the property holds for $i - 1$, then
	\begin{align*}
		\bx(t, i) &\leq \bx(0, i) + \int_0^t\left[\lambda \bx(s, i - 1) \varphi\left(\bx(s, i - 1)\right) + \bx(s, i + 1)\right]ds \\
		&\leq \bx(0, i) + \int_0^t(\lambda + 1) \bx(s, i - 1)ds \\
		&\leq \bx(0, i) + \sum_{j = 0}^{i - 1} \bx(0, j)\frac{[(\lambda + 1)t]^{i - j}}{(i - j)!} = \sum_{j = 0}^i \bx(0, j)\frac{[(\lambda + 1)t]^{i - j}}{(i - j)!} \quad \text{for all} \quad t \geq 0,
	\end{align*}
	which means that the property holds for $i$ as well.
	
	Let $\bx$ be a truncated-$k_x$ fluid trajectory and $\by$ be a truncated-$k_y$ fluid trajectory. Next we show that $k_x \leq k_y$ and $\bx(0) \leq \by(0)$ imply that $\bx(t) \leq \by(t)$ for all $t \geq 0$, which is the counterpart of \cite[Lemma 2]{vvedenskaya1996queueing}. For this purpose, observe that the continuity of solutions with respect to  initial conditions implies that it suffices to prove the following property: if $\bx(0, i) < \by(0, i)$ for all $1 \leq i \leq k_y$, then $\bx(t, i) < \by(t, i)$ for all $t \geq 0$ and $1 \leq i \leq k_y$.
	
	Suppose then that $\bx(0, i) < \by(0, i)$ for all $1 \leq i \leq k_y$ and let
	\begin{equation*}
		\tau \defeq \inf\set{t \geq 0}{\bx(t, i) \geq \by(t, i)\ \text{for some}\ 1 \leq i \leq k_y}.
	\end{equation*}
	If $\tau = \infty$, then the monotonicity property holds. Hence, let us assume that $\tau < \infty$ and show that this leads to a contradiction. Observe that for each $1 \leq i \leq k_x$, we have
	\begin{align*}
		\dot{\bx}(t, i) - \dot{\by}(t, i) &= \lambda\left[\bx(t, i - 1)\varphi\left(\bx(t, i - 1)\right) - \by(t, i - 1)\varphi\left(\by(t, i - 1)\right)\right] \\
		&- \lambda\left[\bx(t, i)\varphi\left(\bx(t, i)\right) - \by(t, i)\varphi\left(\by(t, i)\right)\right] \\
		&- \left[\bx(t, i) - \by(t, i)\right] + \bx(t, i + 1) - \by(t, i + 1) \\
		&\leq \lambda\left[\by(t, i)\varphi\left(\by(t, i)\right) - \bx(t, i)\varphi\left(\bx(t, i)\right)\right] + \left[\by(t, i) - \bx(t, i)\right] \quad \text{if} \quad t \in [0, \tau],
	\end{align*}
	where the last step follows from $x \mapsto x\varphi(x)$ being increasing. Moreover, the inequality also holds for each $k_x < i \leq k_y$ since $\bx(i)$ is identically zero and $\dot{\by}(i) \geq -\lambda \by(i)\varphi\left(\by(i)\right) - \by(i)$.
	
	By Lemma \ref{lem: derivative of phi}, $\varphi'$ is bounded in $[0, 1]$, so there exists $L \geq 0$ such that
	\begin{align*}
		\by(i)\varphi\left(\by(i)\right) - \bx(i)\varphi\left(\bx(i)\right) &= \left|\by(i)\varphi\left(\by(i)\right) - \bx(i)\varphi\left(\bx(i)\right)\right| \\
		&\leq L\left|\by(i) - \bx(i)\right| = L\left[\by(i) - \bx(i)\right]
	\end{align*}
	in the interval $[0, \tau]$ for all $1 \leq i \leq k_y$. Therefore,
	\begin{equation*}
		\dot{\bx}(t, i) - \dot{\by}(t, i) \leq (\lambda L + 1)\left[\by(t, i) - \bx(t, i)\right] = -(\lambda L + 1)\left[\bx(t, i) - \by(t, i)\right]
	\end{equation*}
	for all $t \in [0, \tau]$ and $1 \leq i \leq k_y$. It follows that
	\begin{equation*}
		\bx(t, i) - \by(t, i) \leq \left[\bx(0, i) - \by(0, i)\right]\e^{-(\lambda L + 1)t} < 0 \quad \text{for all} \quad 1 \leq i \leq k_y \quad \text{and} \quad t \in [0, \tau].
	\end{equation*}
	Since $\bx(i)$ and $\by(i)$ are continuous for all $1 \leq i \leq k_y < \infty$, we must have $\bx(\tau, i) = \by(\tau, i)$ for some $1 \leq i \leq k_y$. However, the above inequality contradicts this observation, so the monotonicity property must hold for the truncated fluid trajectories, i.e.,
	\begin{equation}
		\label{aux: monotonicity of truncated fluid trajectories}
		\begin{split}
			&k_x \leq k_y \quad \text{and} \quad \bx(0, i) \leq \by(0, i) \quad \text{for all} \quad 1 \leq i \leq k_y \\
			&\text{implies that} \quad \bx(t, i) \leq \by(t, i) \quad \text{for all} \quad t \geq 0 \quad \text{and} \quad 1 \leq i \leq k_y.
		\end{split}
	\end{equation}

	Next we prove the monotonicity property for untruncated fluid trajectories. For this purpose, we associate to each fluid trajectory $\bx$ a sequence of truncated fluid trajectories $\bx_k$ such that $\bx_k$ is the truncated-$k$ fluid trajectory with initial condition $\bx_k(0, i) = \bx(0, i)$ for all $i \leq k$. It follows from \eqref{aux: monotonicity of truncated fluid trajectories} that $0 \leq \bx_k(t, i) \leq \bx_{k + 1}(t, i) \leq 1$ for all $t \geq 0$ and $i \geq 1$, and thus the following limit exists for all $t \geq 0$ and $i \geq 1$:
	\begin{equation*}
		\bx_\infty(t, i) \defeq \lim_{k \to \infty} \bx_k(t, i).
	\end{equation*}
	Moreover, we conclude from \eqref{aux: properties of truncated fluid trajectories} that $\bx_\infty(t) \in Q$ for all $t \geq 0$, because
	\begin{align*}
		\norm{\bx_\infty(t)}_1 = \sum_{i = 0}^\infty \lim_{k \to \infty} \bx_k(t, i) &\leq \sum_{i = 0}^\infty \sum_{j = 0}^i \bx(0, j)\frac{\left[(\lambda + 1)t\right]^{i - j}}{(i - j)!} \\
		&= \sum_{j = 0}^\infty \sum_{i = j}^\infty \bx(0, j)\frac{\left[(\lambda + 1)t\right]^{i - j}}{(i - j)!} \leq \norm{\bx(0)}_1 \e^{(\lambda + 1)t} < \infty.
	\end{align*}
	Note that $\bx_k$ satisfies \eqref{eq: fluid differential equation} for all $k \geq i$. Taking the limit as $k \to \infty$ on both sides of the integral version of this equation we conclude that $\bx_\infty$ satisfies \eqref{eq: fluid differential equation} as well. Proposition \ref{prop: existence and uniqueness of solutions} then implies that $\bx_\infty = \bx$, the unique fluid trajectory with initial condition $\bx(0)$.
	
	Now suppose that $\bx$ and $\by$ are fluid trajectories such that $\bx(0) \leq \by(0)$ and consider truncated-$k$ fluid trajectories $\bx_k$ and $\by_k$ such that $\bx_k(0, i) = \bx(0, i)$ and $\by_k(0, i) = \by(0, i)$ for all $i \leq k$. The monotonicity property for the truncated-$k$ fluid trajectories yields
	\begin{equation*}
		\bx(t, i) = \lim_{k \to \infty} \bx_k(t, i) \leq \lim_{k \to \infty} \by_k(t, i) = \by(t, i) \quad \text{for all} \quad t \geq 0 \quad \text{and} \quad i \geq 1.
	\end{equation*}
	Hence, the monotonicity property also holds for the untruncated fluid trajectories.
\end{proof}

\begin{proof}[Proof of Proposition \ref{prop: global attractivity}.]
	By Lemma \ref{lem: monotonicity}, it suffices to prove the proposition in the following two cases: $\bq(0) \geq q^*$ and $\bq(0) \leq q^*$. We only prove the proposition in the case $\bq(0) \geq q^*$ because the proof is analogous in the other case.
	
	Suppose then that $\bq(0) \geq q^*$. We proceed as in \cite[Propostion 4.5]{gamarnik2018delay}, letting
	\begin{align*}
	&\bs(t, i) \defeq \sum_{j = i}^\infty \bq(t, j) \quad \text{for all} \quad i \geq 1 \quad \text{and} \quad t \geq 0, \\
	&\dot{\bs}(t, i) \defeq \sum_{j = i}^\infty \dot{\bq}(t, j) = \lambda\bq(t, i - 1)\varphi\left(\bq(t, i - 1)\right) - \bq(t, i) \quad \text{for all} \quad i \geq 1 \quad \text{and} \quad t \geq 0.
	\end{align*}
	Since fluid trajectories take values in $Q \subset \ell_1$, both sums converge pointwise; i.e., for each fixed $t \geq 0$. It follows from \eqref{eq: fluid differential equation} that $|\dot{\bq}(i)|$ is bounded by $\lambda + 1$ for all $i \geq 1$, thus $\bq(i)$ is Lipschitz of modulus $\lambda + 1$. Moreover, because $\varphi'$ is bounded in the interval $[0, 1]$, there exists $M \geq 0$ such that the function $x \mapsto x\varphi(x)$ is Lipschitz of modulus $M$. Therefore,
	\begin{equation*}
	\sum_{j = i}^k \dot{\bq}(j) = \lambda\left[\bq(i - 1)\varphi\left(\bq(i - 1)\right) - \bq(k)\varphi(\bq(k))\right] - \left[\bq(i) - \bq(k + 1)\right]
	\end{equation*}
	is Lipschitz of modulus $4\max\{\lambda(\lambda + 1)M, \lambda + 1\}$ for all $k \geq i$. By the Arezl\'a-Ascoli theorem, the above partial sums converge uniformly over compact sets to $\dot{\bs}(i)$. Hence, it follows from \cite[Theorem 7.17]{rudin1976principles} that $\bs(i)$ also converges uniformly over compact sets and $\dot{\bs}(t, i)$ is the derivative of $\bs(i)$ at $t$ for all $t > 0$.
	
	Since $\varphi'$ is bounded in $[0, 1]$, there exists a constant $L \geq 0$ such that $\varphi$ is a Lipschitz function of modulus $L$. From this observation we conclude that
	\begin{align*}
	\dot{\bs}(i) &= \lambda\bq(i - 1)\varphi\left(\bq(i - 1)\right) - \bq(i) \\
	&= \lambda\left[\bq(i - 1)\varphi\left(\bq(i - 1)\right) - q^*(i - 1)\varphi\left(q^*(i - 1)\right)\right] + q^*(i) - \bq(i) \\
	&= \lambda\left[\bq(i - 1) - q^*(i - 1)\right]\varphi\left(\bq(i - 1)\right) + \lambda q^*(i - 1)\left[\varphi\left(\bq(i - 1)\right) - \varphi\left(q^*(i - 1)\right)\right] \\
	&+ q^*(i) - \bq(i) \\
	&\leq \lambda(1 + L)\left[\bq(i - 1) - q^*(i - 1)\right] - \left[\bq(i) - q^*(i)\right] \quad \text{for all} \quad i \geq 1.
	\end{align*}
	For the inequality observe that $\bq(t, i - 1) \geq q^*(i - 1)$ by assumption and Lemma \ref{lem: monotonicity}. In addition, note that $\varphi$ is an increasing function, thus $\varphi\left(\bq(t, i - 1)\right) \geq \varphi\left(q^*(i - 1)\right)$ for all $i \geq 1$ and $t \geq 0$. Integrating on both sides of the inequality, we conclude that
	\begin{equation*}
	\int_0^t \left[\bq(s, i) - q^*(i)\right]ds \leq \bs(0, i) - \bs(t, i) + \lambda(1 + L)\int_0^t \left[\bq(s, i - 1) - q^*(i - 1)\right]ds.
	\end{equation*}
	
	Setting $i = 1$, we obtain
	\begin{equation*}
	\int_0^t \left[\bq(s, 1) - q^*(1)\right]ds \leq \bs(0, 1) - \bs(t, 1) \leq \bs(0, 1),
	\end{equation*}
	and taking the limit as $t \to \infty$, we get
	\begin{equation*}
	\int_0^\infty \left[\bq(s, 1) - q^*(1)\right]ds \leq \bs(0, 1) < \infty.
	\end{equation*}
	
	By induction in $i$, we conclude that
	\begin{equation*}
	\int_0^\infty \left[\bq(s, i) - q^*(i)\right]ds \leq \bs(0, i) + \lambda(1 + L)\int_0^\infty \left[\bq(s, i - 1) - q^*(i - 1)\right]ds < \infty \quad \text{if} \quad i \geq 1.
	\end{equation*}
	Recalling that $\bq(t, i) \geq q^*(t, i)$ for all $t \geq 0$ and noting that $\bq(i)$ is Lipschitz by \eqref{eq: fluid differential equation}, we conclude from the above equation that $\bq(t, i) \to q^*(i)$ as $t \to \infty$ for all $i \geq 1$. 
\end{proof}

\begin{proof}[Proof of Corollary \ref{cor: total number of tasks in equilibrium}.]
	In order to compute the limit of $R_n$, note that
	\begin{equation*}
	\left|\norm{q_n}_1 - \norm{q^*}_1\right| \leq \norm{q_n - q_*}_1 \quad \text{for all} \quad n \geq 1.
	\end{equation*}
	The limit $q_n \Rightarrow q^*$ in $\ell_1$ is equivalent to $\norm{q_n - q^*}_1 \Rightarrow 0$, and hence $\norm{q_n}_1 \Rightarrow \norm{q_*}_1$. By Lemma \ref{lem: tightness and uniform integrability}, the sequence $\set{\norm{q_n}_1}{n \geq 1}$ is uniformly integrable, so the latter limit also holds in expectation. We conclude that
	\begin{equation*}
	\lim_{n \to \infty} R_n = \lim_{n \to \infty} \frac{n}{\lambda_n} E\left[\norm{q_n}_1 - 1\right] = \frac{\norm{q^*}_1 - 1}{\lambda}.
	\end{equation*}
	This completes the proof.
\end{proof}

\begin{proof}[Proof of Lemma \ref{lem: moments of number of arrivals and departures between resampling times}]
	The process that describes the times of departures from the servers can be described as follows. Every server experiences potential departures as a Poisson process of unit intensity and if a server has at least one task at the time of a potential departure, then a task departs from the server.  Given $\sigma_n^m - \sigma_n^{m - 1}$, the number of arrivals and potential departures in $\left(\sigma_n^{m - 1}, \sigma_n^m\right]$ are Poisson distributed with mean $\lambda_n \left(\sigma_n^m - \sigma_n^{m - 1}\right)$ and $n \left(\sigma_n^m - \sigma_n^{m - 1}\right)$, respectively. We get \eqref{ch5-sau: arrivals} directly from this observation, and \eqref{ch5-sau: departures} follows by further noting that $D_n^m$ is upper bounded by the number of potential departures and that this number is independent of the number of arrivals $A_n^m$. 
	
	Assume that condition (c) of Proposition \ref{prop: admissible resampling processes} holds. In order to establish \eqref{eq: bounds for condition c}, define
	\begin{equation*}
	\Gamma_n(t) \defeq \sum_{m = 1}^{\calR_n(t) + 1} \left(\sigma_n^m - \sigma_n^{m - 1}\right)^2 \quad \text{and} \quad \varphi_n(t) \defeq E\left[\Gamma_n(t)\right] \quad \text{for all} \quad t \geq 0.
	\end{equation*}
	Applying a renewal argument we conclude that
	\begin{equation*}
	\expect*{\Gamma_n(t) | \sigma_n^1 = \sigma} = \sigma^2 + \varphi_n(t - \sigma) \ind{\sigma \leq t} \quad \text{for all} \quad t, \sigma \geq 0.
	\end{equation*}
	Hence, if we let $F_n(\sigma) \defeq P\left(\sigma_n^1 \leq \sigma\right)$ denote the cumulative distribution function of the holding times, then integration with respect to $F_n$ yields
	\begin{equation*}
	\varphi_n(t) = \int_0^\infty \expect*{\Gamma_n(t) | \sigma_n^1 = \sigma}dF_n(\sigma) = E\left[\left(\sigma_n^1\right)^2\right] + \int_0^t \varphi_n(t - \sigma)dF_n(\sigma) \quad \text{for all} \quad t \geq 0.
	\end{equation*}
	
	By \cite[Theorem 12.24]{kallenberg2021foundations}, the solution of this renewal equation is
	\begin{equation*}
	\varphi_n(t) = E\left[\left(\sigma_n^1\right)^2\right] + \int_0^t E\left[\left(\sigma_n^1\right)^2\right] dR_n(t) = E\left[\left(\sigma_n^1\right)^2\right]\left[1 + R_n(t)\right] \quad \text{for all} \quad t \geq 0,
	\end{equation*}
	where $R_n(t) = E\left[\calR_n(t)\right]$. Also, condition (c) of Proposition \ref{prop: admissible resampling processes} implies that $\sigma_n^1 = \sigma_1^1 / \mu_n$, and the elementary renewal theorem yields $R_n(t) / \mu_n \to t$ as $n \to \infty$. Therefore,
	\begin{equation*}
	\lim_{n \to \infty} \mu_n \varphi_n(t) = \lim_{n \to \infty} E\left[\left(\sigma_1^1\right)^2\right] \frac{1 + R_n(t)}{\mu_n} = E\left[\left(\sigma_1^1\right)^2\right] t.
	\end{equation*}
	This completes the proof.
\end{proof}

\begin{proof}[Proof of Lemma \ref{lem: convergence in probability from components}]
	The topology of $D_{\R^\N}[0, \infty)$ is compatible with the metric
	\begin{equation*}
	\eta(\bx, \by) \defeq \sum_{T = 1}^\infty \frac{\min\left\{\sup_{t \in [0, T]} d\left(\bx(t), \by(t)\right), 1\right\}}{2^T} \quad \text{for all} \quad \bx, \by \in D_{\R^\N}[0, \infty).
	\end{equation*}
	It is straightforward to check that (a) implies (b), thus we only prove that (b) implies (a). For this purpose, fix $\varepsilon > 0$ and assume that
	\begin{equation*}
	\bx_n(i) \Rightarrow 0 \quad \text{in} \quad D_\R[0, T] \quad \text{as} \quad n \to \infty \quad \text{for all} \quad i \geq 0 \quad \text{and} \quad T \geq 0.
	\end{equation*}
	
	Choose $l$ and $m$ such that
	\begin{equation*}
	\sum_{T = l + 1}^\infty \frac{1}{2^T} \leq \frac{\varepsilon}{2} \quad \text{and} \quad \sum_{i = m}^\infty \frac{1}{2^i} \leq \frac{\varepsilon}{4l}.
	\end{equation*}
	By the choice of $l$, we have
	\begin{equation*}
	P\left(\eta(\bx_n, 0) \geq \varepsilon\right) \leq \sum_{T = 1}^l P\left(\sup_{t \in [0, T]} d\left(\bx_n(t), 0\right) \geq \frac{\varepsilon}{2l}\right),
	\end{equation*}
	and by the choice of $m$,
	\begin{equation*}
	P\left(\sup_{t \in [0, T]} d\left(\bx_n(t), 0\right) \geq \frac{\varepsilon}{2l}\right) \leq \sum_{i = 0}^{m - 1} P\left(\norm{\bx_n(i)}_T \geq \frac{\varepsilon}{4lm}\right).
	\end{equation*}
	
	Therefore, we have
	\begin{equation*}
	P\left(\eta(\bx_n, 0) \geq \varepsilon\right) \leq \sum_{T = 1}^l \sum_{i = 0}^{m - 1} P\left(\norm{\bx_n(i)}_T \geq \frac{\varepsilon}{4lm}\right),
	\end{equation*}
	and the right-hand side converges to zero as $n \to \infty$ by assumption; indeed, note that weak convergence to zero is equivalent to convergence in probability to zero.
\end{proof}

\begin{proof}[Proof of Lemma \ref{lem: joint convergence in distribution}]
	Let $\rho_j$ be the metric of $S_j$ and endow $\Pi$ with the metric defined by
	\begin{equation*}
	\varrho\left(\left(y_1, \dots, y_m\right), \left(z_1, \dots, z_m\right)\right) \defeq \max_{j = 1, \dots, m} \rho_j\left(y_j, z_j\right) \quad \text{for all} \quad \left(y_1, \dots, y_m\right), \left(z_1, \dots, z_m\right) \in \Pi,
	\end{equation*}
	which is compatible with the product topology. If $\map{f}{\Pi}{\R}$ is continuous and bounded, then $y_1 \mapsto f(y_1, x_2, \dots, x_m)$ defines a continuous and bounded function on $S_1$. Hence,
	\begin{equation*}
	\lim_{k \to \infty} E\left[f\left(X_k^1, x_2, \dots, x_m\right)\right] = E\left[f\left(X_1, x_2, \dots, x_m\right)\right],
	\end{equation*}
	and we conclude that $\left(X_k^1, x_2, \dots, x_m\right) \Rightarrow \left(X_1, x_2, \dots, x_m\right)$ in $\Pi$ as $k \to \infty$. Moreover,
	\begin{equation*}
	P\left(\varrho\left(\left(X_k^1, X_k^2, \dots, X_k^m\right), \left(X_k^1, x_2, \dots, x_m\right)\right) \geq \varepsilon\right) \leq \sum_{j = 2}^m P\left(\rho_j\left(X_k^j, x_j\right) \geq \varepsilon\right),
	\end{equation*}
	and for every $\varepsilon > 0$ the right-hand side goes to zero as $k \to \infty$ by assumption. It follows from \cite[Theorem 3.1]{billingsley1999convergence} that $\left(X_k^1, X_k^2, \dots, X_k^m\right) \Rightarrow \left(X_1, x_2, \dots, x_m\right)$ in $\Pi$.
\end{proof}

\begin{proof}[Addition to Remark \ref{rem: measurability}]
	Let $S_\R[0, \infty)$ denote the space of real c\`adl\`ag functions endowed with the Skorohod $J_1$-topology, and let $\Pi \defeq S_\R[0, \infty) \times S_\R[0, \infty) \times S_{\ell_1}[0, \infty) \times S_{\R^\N}[0, \infty)$ with the product topology and the Borel $\sigma$-algebra. The right-hand side of \eqref{eq: functional equation} defines a measurable function from $\Pi$ into $\R$, thus the probability that \eqref{eq: functional equation} holds is equal to
	\begin{equation*}
	P\left(\left(\calN_k^a, \calR_k, \bq_k, \bv_k\right) \in A\right)
	\end{equation*}
	for some set $A$ in the Borel $\sigma$-algebra of $\Pi$. Skorohod's representation theorem implies that the law of $\left(\calN_k^a, \calR_k, \bq_k, \bv_k\right)$ is as in Section \ref{sub: decomposition of the equations}, so the latter probability equals one.
	
	As an example, let us establish that $\bu_k(t, i)$ is a measurable function of $\left(\calN_k^a, \calR_k, \bq_k, \bv_k\right)$ when \eqref{eq: general decomposition} holds. For this purpose, consider partitions $0 = t_0^l < \dots < t_{J_l}^l = t$ such that
	\begin{equation*}
	\lim_{l \to \infty} \max_{0 \leq j \leq J_l - 1} \left(t_{j + 1}^l - t_j^l\right) = 0.
	\end{equation*}
	In addition, define $s_m^l \defeq t_{j_m^l}^l$ for each $l, m \geq 1$, where
	\begin{equation*}
	j_m^l \defeq \max \set{0 \leq j \leq J_l}{\calN_k^a\left(t_j^l\right) < m\ \text{and}\ j = 0\ \text{or}\ \calR_k\left(t_j^l\right) > \calR_k\left(t_{j - 1}^l\right)}.
	\end{equation*}
	
	Because the finite-dimensional projections $\map{\pi_{t_1, \dots, t_l}}{\Pi}{\R^l}$ are measurable for all $t_1, \dots, t_l \geq 0$, we conclude that $j_m^l$ is measurable as a function from $\Pi$ into $\R$. Also,
	\begin{equation*}
	\lim_{l \to \infty} s_m^l = \max \set{\sigma_k^j}{\sigma_k^j < \tau_k^m},
	\end{equation*}
	and $s_m^l \geq \max \set{\sigma_k^j}{\sigma_k^j < \tau_k^m}$ for all large enough $l$. Since $\bq_k$ is right-continuous,
	\begin{equation*}
	\lim_{l \to \infty} \bq_k\left(s_m^l, i\right) = \bar{q}_k^m(i).
	\end{equation*}
	
	It follows that
	\begin{equation*}
	\bu_k(t, i) = \lim_{l \to \infty} \frac{1}{k} \sum_{m = 1}^l \sum_{j = 1}^{J_l} \beta_k \left(\bq_k\left(t_j^l, i\right)\right) \ind{j = j_m^l} \ind{m \leq \calN_k^a(t)}.
	\end{equation*}
	Each of the functions inside of the limit sign is measurable from $\Pi$ into $\R$ and the limit of measurable functions is a measurable function as well. Hence, we conclude that $\bu_k(t, i)$ is a measurable function of $\left(\calN_k^a, \calR_k, \bq_k, \bv_k\right)$.
\end{proof}

\begin{proof}[Proof of Lemma \ref{lem: uniform convergence of sampling functions}]
	In order to prove \eqref{eq: limit of alpha n}, fix an arbitrary $\varepsilon > 0$ and note that
	\begin{equation*}
	\sup_{x \in [0, d/ n)} \left|\alpha_n(d + 1, x) - x^{d + 1}\right| =  \sup_{x \in [0, d/ n)} x^{d + 1} \leq \varepsilon
	\end{equation*}
	for all sufficiently large $n$. Now consider the function $\map{f}{[0, 1]^{d + 1}}{\R}$ that assigns to each vector the product of its entries. The bound
	\begin{equation*}
	\max_{0 \leq k \leq d}\left|\frac{nx - k}{n - k} - x\right| = \max_{0 \leq k \leq d} \left|\frac{k(x - 1)}{n - k}\right| \leq \frac{d}{n - d} \quad \text{for all} \quad x \in \left[0, 1\right],
	\end{equation*}
	and the uniform continuity of $f$ imply that
	\begin{equation*}
	\sup_{x \in [d/ n, 1]} \left|\alpha_n(d + 1, x) - x^{d + 1}\right| \leq \sup_{x \in [0, 1]} \left|f\left(x, \frac{nx - 1}{n - 1}, \dots, \frac{nx - d}{n - d}\right) - f(x, x, \dots, x)\right| \leq \varepsilon
	\end{equation*}
	for all large enough $n$. Because $\varepsilon$ is arbitrary, this proves \eqref{eq: limit of alpha n}.
	
	For \eqref{eq: limit of beta n}, observe that
	\begin{equation*}
	\beta_n(x) - x\varphi(x) = \sum_{d = 0}^\infty \left[\alpha_n(d + 1, x)p_n(d) - x^{d + 1}p(d)\right] \quad \text{for all} \quad x \in [0, 1] \quad \text{and} \quad n \geq 1.
	\end{equation*}
	Fix $\theta \in [0, 1)$ and some $k \geq 0$. If $x \in [0, \theta]$, then
	\begin{align*}
	\left|\beta_n(x) - x\varphi(x)\right| &\leq \sum_{d = 0}^k \left|\alpha_n(d + 1, x)p_n(d) - x^{d + 1}p(d)\right| + \sum_{d = k + 1}^\infty \alpha_n(d + 1, x) + \sum_{d = k + 1}^\infty x^{d + 1} \\
	&\leq \sum_{d = 0}^k \left|\alpha_n(d + 1, x) - x^{d + 1}\right|p_n(d) + \sum_{d = 0}^k x^{d + 1}\left|p_n(d) - p(d)\right| + 2\sum_{d = k + 1}^\infty x^{d + 1} \\
	&\leq \sum_{d = 0}^k \left|\alpha_n(d + 1, x) - x^{d + 1}\right| + \sum_{d = 0}^k \left|p_n(d) - p(d)\right| + \frac{2\theta^{k + 2}}{1 - \theta}.
	\end{align*}
	For the second inequality, note that $\alpha_n(d + 1, x) \leq x^{d + 1}$ for all $d \geq 0$ and $n \geq 1$. Given an arbitrary $\varepsilon > 0$, we may choose $k$ such that $2\theta^{k + 2} \leq (1 - \theta) \varepsilon$. Then
	\begin{equation*}
	\lim_{n \to \infty} \sup_{x \in [0, \theta]}\left|\beta_n(x) - x\varphi(x)\right| \leq \lim_{n \to \infty} \left[\sum_{d = 0}^k \sup_{x \in [0, 1]} \left|\alpha_n(d + 1, x) - x^{d + 1}\right| + \sum_{d = 0}^k \left|p_n(d) - p(d)\right|\right] + \varepsilon.
	\end{equation*}
	Since $\varepsilon$ is arbitrary, we conclude from \eqref{eq: limiting degree distribution condition} and \eqref{eq: limit of alpha n} that \eqref{eq: limit of beta n} holds.
	
	Suppose now that $p(\infty) = 0$ and thus $\varphi(1) = 1$. It follows from Abel's theorem that $\varphi$ is continuous on $[0, 1]$. Hence, if $\varepsilon > 0$, then there exists $\delta \in (0, 1)$ such that
	\begin{equation*}
	|x\varphi(x) - y\varphi(y)| \leq \frac{\varepsilon}{3} \quad \text{for all} \quad x, y \in [0, 1] \quad \text{such that} \quad |x - y| \leq \delta.
	\end{equation*}
	Choose $\theta \in (1 - \delta, 1)$ and note that \eqref{eq: limit of beta n} implies that there exists $m$ such that
	\begin{equation*}
	|\beta_n(x) - x\varphi(x)| \leq \frac{\varepsilon}{6} \quad \text{for all} \quad x \in [0, \theta] \quad \text{and} \quad n \geq m.
	\end{equation*}
	If $x \in (\theta, 1]$, then we have
	\begin{align*}
	|\beta_n(x) - x\varphi(x)| &\leq |\beta_n(x) - \beta_n(\theta)| + |\beta_n(\theta) - \theta\varphi(\theta)| + |\theta\varphi(\theta) - x\varphi(x)| \\
	&\leq 1 - \beta_n(\theta) + |\beta_n(\theta) - \theta\varphi(\theta)| + |\theta\varphi(\theta) - x\varphi(x)| \\
	&\leq 1 - \theta\varphi(\theta) + 2|\beta_n(\theta) - \theta\varphi(\theta)| + |\theta\varphi(\theta) - x\varphi(x)| \leq \varepsilon \quad \text{for all} \quad n \geq m.
	\end{align*}
	For the second inequality, note that $\beta_n$ is nondecreasing and $\beta_n(1) = 1$ because $\alpha_n(d + 1)$ has these properties for each $d \leq n - 1$. For the last inequality, recall that $\varphi(1) = 1$ by assumption, and note that $|1 - x| < |1 - \theta| < \delta$. Since $\varepsilon$ is arbitrary, \eqref{eq: limit of beta n non-degenarate case} holds.
\end{proof}

\section{Construction of sample paths}
\label{app: construction of sample paths}

The processes $\bq_n$ and $\bX_n$ are constructed inductively. At time zero,
\begin{equation*}
	\bX_n(0) = X_n \quad \text{and} \quad \bq_n(0, i) = q_n(i) = \frac{1}{n}\sum_{u = 1}^n \ind{X_n(u) \geq i} \quad \text{for all} \quad i \geq 0.
\end{equation*}
We refer to time zero and the times of arrivals and departures of tasks as event times. If both processes have already been defined up to event time $\tau$, then they remain constant until the next event occurs.

Let $\bq_n^\tau$ denote the stopped process defined as
\begin{equation*}
	\bq_n^\tau(t) \defeq \bq_n(t) \quad \text{if} \quad 0 \leq t \leq \tau \quad \text{and} \quad \bq_n^\tau(t) \defeq \bq_n(\tau) \quad \text{if} \quad t > \tau.
\end{equation*}
The next event after $\tau$ corresponds to the first jump after $\tau$ of one the processes
\begin{equation}
	\label{eq: arrival and departure processes}
	t \mapsto \calN_n^a(t) \quad \text{and} \quad t \mapsto \calN_i^d\left(n\int_0^t\left[\bq_n^\tau(s, i) - \bq_n^\tau(s, i + 1)\right]ds\right).
\end{equation}
Only finitely many of these processes have a positive intensity since the initial number of tasks in the system is finite by assumption, hence the time of the next event is strictly larger than $\tau$. The intensity of process $i$ on the right corresponds to the departure rate from servers with exactly $i$ tasks, a jump of this process indicates such a departure.

Once the time of the next event is determined, the processes $\bq_n$ and $\bX_n$ are updated in different ways depending on the type of the new event. If the first event after $\tau$ is an arrival that occurs at time $\tau_n^m$, then we set
\begin{equation*}
	\bq_n\left(\tau_n^m, i\right) = \bq_n\left(\tau_n^{m-}, i\right) + \frac{1}{n}\left[I_n^m\left(\bX_n\left(\tau_n^{m-}\right), i - 1\right) - I_n^m\left(\bX_n\left(\tau_n^{m-}\right), i\right)\right] \quad \text{for all} \quad i \geq 1.
\end{equation*}
The difference between the last two terms equals one if the task is placed in the queue of a server with exactly $i - 1$ tasks and equals zero otherwise. In addition, we let
\begin{equation*}
	v_n^m = \min \argmin{v} \set{\bX_n\left(\tau_n^{m-}, v\right)}{v = u_n^m\ \text{or}\ \left(u_n^m, v\right) \in \bE_n\left(\tau_n^{m-}\right)}
\end{equation*}
be the server with the smallest index among the servers with the least number of tasks in the neighborhood of $u_n^m$ and we set
\begin{equation*}
	\bX_n\left(\tau_n^m, v\right) = \bX_n\left(\tau_n^{m-}, v\right) + \ind{v = v_n^m} \quad \text{for all} \quad v \in V_n.
\end{equation*}
In order to choose the server $v_n^m$ in the neighborhood of $u_n^m$ that will receive the new task, we break ties between servers with the least number of tasks by selecting the server with the smallest index. But any other criterion could be used instead.

Suppose instead that the first event after $\tau$ is a departure from a server with exactly $i$ tasks, triggered by a jump of process $i$ on the right of \eqref{eq: arrival and departure processes}. We denote the time of this departure by $\tau_{i, n}^m$ and we set
\begin{equation*}
	\bq_n\left(\tau_{i, n}^m, j\right) = \bq_n\left(\tau_{i, n}^{m-}, j\right) - \frac{1}{n}\ind{j = i} \quad \text{for all} \quad j \geq 1.
\end{equation*}
In addition, we use the random variable $U_{i, n}^m$ to select a server $v_{i, n}^m$ uniformly at random among all the servers $v$ with exactly $i$ tasks: $\bX_n\left(\tau_{i, n}^{m-}, v\right) = i$. Then we set
\begin{equation*}
	\bX_n\left(\tau_{i, n}^m, v\right) = \bX_n\left(\tau_{i, n}^{m-}, v\right) - \ind{v_{i, n}^m = v} \quad \text{for all} \quad v \in V_n.
\end{equation*}

The above construction determines $\bX_n$ and $\bq_n$ on a set of probability one that excludes certain events of probability zero, such as tasks arriving and departing simultaneously. Both $\bq_n$ and $\bX_n$ are piecewise constant c\`adl\`ag  processes defined on $[0, \infty)$. In addition, the jumps of $\bq_n$ are of size $1 / n$ and $\bX_n$ has jumps of unit size.

\section{Relative compactness}
\label{app: tightness of occupancy processes}

The uniform topology of $D_{\ell_1}[0, \infty)$ is compatible with the metric defined by
\begin{equation*}
\rho(\bx, \by) \defeq \sum_{T = 1}^\infty \frac{\min\left\{\sup_{t \in [0, T]} \norm{\bx(t) - \by(t)}_1, 1\right\}}{2^T} \quad \text{for all} \quad \bx, \by \in D_{\ell_1}[0, \infty).
\end{equation*}
In addition, given a function $\bx \in D_{\ell_1}[0, \infty)$, its local moduli of continuity is defined as
\begin{equation*}
w_T(\bx, h) \defeq \sup \set{\norm{\bx(s) - \bx(t)}_1}{s, t \in [0, T]\ \text{and}\ |s - t| \leq h}
\end{equation*}
for all $h > 0$ and $T \geq 0$. The modified local moduli of continuity is defined as
\begin{equation*}
\tilde{w}_T(\bx, h) \defeq \inf_\calI \max_{I \in \calI} \sup_{s, t \in I} \norm{\bx(s) - \bx(t)}_1.
\end{equation*}
Here the infimum extends over all partitions $\calI$ of $[0, T)$ into subintervals $I = [u, v)$ such that $v - u \geq h$ if $v < T$. By \cite[Theorems 23.8 and 23.9]{kallenberg2021foundations}, and the fact that
\begin{equation*}
\tilde{w}_T(\bx, h) \leq w_T(\bx, h) \quad \text{for all} \quad \bx \in D_{\ell_1}[0, \infty), \quad h > 0 \quad \text{and} \quad T \geq 0,
\end{equation*}
the relative compactness of $\set{\bq_n}{n \geq 1}$ can be obtained by proving the next properties.
\begin{enumerate}
	\item[(a)] $\set{\bq_n(t)}{n \geq 1}$ is tight in $\ell_1$ for all $t$ in some dense subset of $[0, \infty)$.
	
	\item[(b)] If $T > 0$, then
	\begin{equation*}
	\lim_{h \to 0} \limsup_{n \to \infty} E\left[\min\left\{w_T(\bq_n, h), 1\right\}\right] = 0.
	\end{equation*}
\end{enumerate}
The above properties correspond to (i) and (ii) of \cite[Theorem 23.8]{kallenberg2021foundations}, where the supremum in (ii) can be replaced by a limit superior using similar arguments as in \cite[Theorem 13.2]{billingsley1999convergence}. Properties (a) and (b) imply that the sequence $\set{\bq_n}{n \geq 1}$ is tight with respect to the Skorohod-$J_1$ topology by \cite[Theorem 23.8]{kallenberg2021foundations}. Hence, every convergent subsequence has a further subsequence that converges weakly in the Skorohod-$J_1$ topology. It follows from (b) and \cite[Theorem 23.9]{kallenberg2021foundations} that any such subsequence also converges in the uniform topology to a process that is continuous, and this implies that Proposition \ref{prop: tightness} holds.

The following two lemmas are used to establish property (a).

\begin{lemma}
	\label{lem: tightness in l1}
	Let $\set{q_n}{n \geq 1}$ be a sequence of random variables with values in
	\begin{equation*}
	Q \defeq \set{q \in \ell_1}{0 \leq q(i + 1) \leq q(i) \leq q(0) = 1\ \text{for all}\ i \geq 1}.
	\end{equation*}
	The sequence $\set{q_n}{n \geq 1}$ is tight in $\ell_1$ if and only if
	\begin{equation*}
	\lim_{m \to \infty} \limsup_{n \to \infty} P\left(\sum_{i > m} q_n(i) > \varepsilon\right) = 0 \quad \text{for all} \quad \varepsilon > 0.
	\end{equation*}
\end{lemma}

The previous lemma is taken from \cite[Lemma 2]{mukherjee2018universality}, where the proof can be found. We use it in the following lemma in order to establish property (a).

\begin{lemma}
	\label{lem: tightness of qn(t)}
	If $\set{\bq_n(0)}{n \geq 1}$ is tight in $\ell_1$, then so is $\set{\bq_n(t)}{n \geq 1}$ for all $t \geq 0$.
\end{lemma}

\begin{proof}
	By Lemma \ref{lem: tightness in l1}, it suffices to prove that
	\begin{equation}
	\label{aux: tightness condition}
	\lim_{m \to \infty} \limsup_{n \to \infty} P\left(\sum_{i > m} \bq_n(t, i) > \varepsilon\right) = 0 \quad \text{for all} \quad \varepsilon, t > 0.
	\end{equation}
	Fix any $\varepsilon, t > 0$, let $\theta \defeq t(\e - 1) + 1$ and choose constants $k, \delta_0, \delta_1 > 0$ such that
	\begin{equation*}
	\delta_0 + \delta_1 \lambda \theta < \varepsilon \quad \text{and} \quad \delta_1 > \delta_0 + \frac{\lambda\theta}{k}.
	\end{equation*}
	In addition, fix $n_0 \geq 0$ such that all $n \geq n_0$ satisfy:
	\begin{subequations}
		\begin{align}
		&\delta_0 + \frac{\delta_1\lambda_n\theta}{n} < \varepsilon, \label{saux1: conditions on deltas and k}\\
		&\delta_1 > \delta_0 + \frac{\lambda_n\theta}{kn}. \label{saux2: conditions on deltas and k}
		\end{align}
	\end{subequations}
	
	For each $m > k$ and $n \geq n_0$, consider the following events:
	\begin{equation*}
	A_{m, n} \defeq \left\{\sum_{i > m} \bq_n(t, i) > \varepsilon\right\} \quad \text{and} \quad B_{m, n} \defeq \left\{\sum_{i > m - k} \bq_n(0, i) > \delta_0\right\}.
	\end{equation*}
	Also, define $C_n$ as the event that the total number of arrivals in the interval $[0, t]$ is strictly larger than $\lambda_n\theta$, and let $D_{m, n}$ be the event that the number of arrivals in the interval $[0, t]$ to servers with at least $m$ tasks is strictly larger than $\delta_1\lambda_n\theta$. In the definition of $D_{m, n}$ we refer to all the tasks that appear at servers with at least $m$ tasks, even if these tasks are then dispatched to a server with fewer than $m$ tasks.
	
	It is clear that
	\begin{equation*}
	P\left(A_{m, n}\right) \leq P\left(A_{m, n} \cap B_{m, n}^c \cap C_n^c\right) + P\left(B_{m, n}\right) + P(C_n).
	\end{equation*}
	Note that \eqref{aux: tightness condition} holds for $\varepsilon = \delta_0$ and $t = 0$ since $\set{\bq_n(0)}{n \geq 1}$ is tight, so $P\left(B_{m, n}\right) \to 0$ as $n \to \infty$ and then $m \to \infty$. Moreover, a Chernoff bound yields
	\begin{equation*}
	P\left(C_n\right) \leq \frac{\e^{\lambda_nt(\e - 1)}}{\e^{\lambda_n \theta}} = \e^{-\lambda_n},
	\end{equation*}
	hence $P\left(C_n\right) \to 0$ as $n \to \infty$. This takes care of the last two term on the right-hand side.
	
	If $n \geq n_0$, then \eqref{saux1: conditions on deltas and k} implies that $A_{m, n} \cap B_{m, n}^c \cap C_n^c \subset B_{m, n}^c \cap C_n^c \cap D_{m, n}$, thus
	\begin{equation*}
	P\left(A_{m, n}\right) \leq P\left(B_{m, n}^c \cap C_n^c \cap D_{m, n}\right) + P\left(B_{m, n}\right) + P(C_n)
	\end{equation*}
	and it only remains to deal with the first term on the right-hand side.
	
	In $B_{m, n}^c \cap C_n^c$ there are at most $\lambda_n\theta$ arrivals in $[0, t]$ and
	\begin{equation*}
	\bq_n(0, i) \leq \delta_0 \quad \text{for all} \quad m - k + 1 \leq i \leq m.
	\end{equation*}
	It follows from \eqref{saux2: conditions on deltas and k} that $\bq_n(s, m) < \delta_1$ for all $s \in [0, t]$ and $n \geq n_0$. Otherwise $\bq_n(s, i) \geq \delta_1$ for all $m - k + 1 \leq i \leq m$, which requires at least $kn(\delta_1 - \delta_0) > \lambda_n \theta$ arrivals.
	
	The process $\bq_n$ can be constructed using Poisson processes
	\begin{equation*}
	\calN_1^n(s) \defeq \calN_1\left(\lambda_n\int_0^s \left[1 - \bq_n(\tau, m)\right]d\tau\right) \quad \text{and} \quad \calN_2^n(s) \defeq \calN_2\left(\lambda_n\int_0^s \bq_n(\tau, m)d\tau\right)
	\end{equation*}
	for counting tasks that appear in servers with less than $m$ tasks and at least $m$ tasks, respectively, where $\calN_1$ and $\calN_2$ are independent Poisson processes of intensity one. Hence,
	\begin{align*}
	P\left(B_{m, n}^c \cap C_n^c \cap D_{m, n}\right) &= P\left(B_{m, n}^c \cap C_n^c \cap \left\{\calN_2^n(t) > \delta_1\lambda_n\theta\right\}\right) \\
	&\leq P\left(B_{m, n}^c \cap C_n^c \cap \left\{\calN_2\left(\delta_1\lambda_nt\right) > \delta_1\lambda_n\theta\right\}\right) \\
	&\leq P\left(\calN_2\left(\delta_1\lambda_nt\right) > \delta_1\lambda_n\theta\right) \leq \frac{\e^{\delta_1\lambda_nt(\e - 1)}}{\e^{\delta_1\lambda_n \theta}} = \e^{-\delta_1\lambda_n},
	\end{align*}
	where the last step uses a Chernoff bound. Since the right-hand side goes to zero as $n \to \infty$, we conclude that \eqref{aux: tightness condition} holds.
\end{proof}

Next we establish property (b) and we complete the proof of Proposition \ref{prop: tightness}.

\begin{proof}[Proof of Proposition \ref{prop: tightness}.]
	An alternative construction of the process $\bq_n$ can be carried out using a single Poisson process for counting both arrivals and potential departures. Let $\calN_n$ be this process, which has intensity $\nu_n \defeq \lambda_n + n$. Then
	\begin{equation}
	\label{aux: bound for modulus of continuity}
	E\left[w_T(\bq_n, h)\right] \leq E\left[\sup_{t \in [0, T]} \frac{\calN_n(t + h) - \calN_n(t)}{n}\right].
	\end{equation}
	Indeed, note that $\bq_n$ has jumps of size $1 / n$ and remains constant between two consecutive jumps of $\calN_n$. It follows that if $s < \tau_1 < \tau_2 < \dots < \tau_k \leq t$ are the arrival and potential departure times in $(s, t]$, then we must have
	\begin{align*}
		\norm{\bq_n(s) - \bq_n(t)}_1 &\leq \norm{\bq_n(s) - \bq_n(\tau_1)}_1 + \sum_{j = 1}^{k - 1} \norm{\bq_n(\tau_j) - \bq_n(\tau_{j + 1})}_1 + \norm{\bq_n(\tau_k) - \bq_n(t)}_1 \\
		&= \norm{\bq_n(s) - \bq_n(\tau_1)}_1 + \sum_{j = 1}^{k - 1} \norm{\bq_n(\tau_j) - \bq_n(\tau_{j + 1})}_1 \leq \frac{k}{n} = \frac{\calN_n(t) - \calN_n(s)}{n}.
	\end{align*}
	This remark leads to \eqref{aux: bound for modulus of continuity}. We now observe that
	\begin{align*}
	E\left[w_T(\bq_n, h)\right] &\leq E\left[\sup_{t \in [0, T]} \left|\frac{\calN_n(t + h) - \nu_n(t + h)}{n}\right|\right] + E\left[\sup_{t \in [0, T]} \left|\frac{\calN_n(t) - \nu_nt}{n}\right|\right] + \frac{\nu_n h}{n}.
	\end{align*}
	It follows from Jensen's inequality and Doob's maximal inequality that
	\begin{align*}
	E\left[w_T(\bq_n, h)\right] &\leq 2\left(\sqrt{E\left[\left|\frac{\calN_n(T + h) - \nu_n(T + h)}{n}\right|^2\right]} + \sqrt{E\left[\left|\frac{\calN_n(T) - \nu_nT}{n}\right|^2\right]}\right) + \frac{\nu_nh}{n} \\
	&= 2\left(\frac{\sqrt{\nu_n(T + h)}}{n} + \frac{\sqrt{\nu_nT}}{n}\right) + \frac{\nu_nh}{n}.
	\end{align*}
	We conclude that (b) holds since
	\begin{equation*}
	\lim_{h \to 0} \limsup_{n \to \infty} E\left[\min\left\{w_T(\bq_n, h), 1\right\}\right] \leq \lim_{h \to 0} \limsup_{n \to \infty} E\left[w_T(\bq_n, h)\right] \leq \lim_{h \to 0} (\lambda + 1)h = 0.
	\end{equation*}
	
	The last identity and \cite[Theorem 23.9]{kallenberg2021foundations} also imply that the limit in distribution of any convergent subsequence of $\set{\bq_n}{n \geq 1}$ is an almost surely continuous process.
\end{proof}

\end{appendices}
	
\newcommand{\noop}[1]{}
\bibliographystyle{plain}
\bibliography{bibliography}
	
\end{document}